\documentclass[11pt]{article}

\usepackage{amsmath,amsfonts,amsthm,amscd,amssymb,graphicx}
\numberwithin{equation}{section}

\usepackage{color}

\newtheorem{theorem}{Theorem}[section]

\newtheorem{lemma}[theorem]{Lemma}

\newtheorem{corollary}[theorem]{Corollary}

\newtheorem{remark}[theorem]{Remark}

\def\D{\partial}
\def\dt{\partial_t}

\def\ep{\epsilon}
\def\lb{\lambda} 
\def\oD{\mathrm{D}}

\def\R{\Re e}
\def\I{\Im m}

\def\I{\Im m}

\newcommand{\RR}{\mathbb{R}}

\def\cH{\mathcal{H}}

\def\vj{{\bf j}}
\def\vh{\tilde{\bf h}}
\def\vg{\tilde{\bf g}}

\def\vA{{\bf A}}
\def\vE{{\bf E}}
\def\vB{{\bf B}}

\def\tv{\tilde v}

\def\A{\mathcal{A}}
\def\B{\mathcal{B}}
\def\S{\mathcal{S}}
\def\T{\mathcal{T}}

\begin{document}

\title{\vspace{-1in} Linear Stability Analysis of a Hot Plasma in a Solid Torus\footnotemark[1]}

\author{Toan T. Nguyen\footnotemark[2] \and Walter A. Strauss\footnotemark[3]}

\date{\today}

\maketitle

\begin{abstract} 

This paper is a first step toward understanding the effect of toroidal geometry on the rigorous stability theory of plasmas.  
We consider a collisionless plasma inside a torus, modeled by the relativistic Vlasov-Maxwell system.  
The surface of the torus is perfectly conducting and it reflects the particles specularly.  
We provide sharp criteria for the stability of equilibria under the assumption that the particle distributions and the electromagnetic fields 
depend only on the cross-sectional variables of the torus.  

\end{abstract}

\renewcommand{\thefootnote}{\fnsymbol{footnote}}

\footnotetext[2]{Department of Mathematics, Pennsylvania State University, University Park, PA~16802, USA. Email: nguyen@math.psu.edu.
}

\footnotetext[3]{Department of Mathematics and Lefschetz Center for Dynamical Systems,
Brown University, Providence, RI 02912, USA. Email: wstrauss@math.brown.edu.} 

\footnotetext[1]{Research of the authors was supported in part by the NSF under grants DMS-1108821 and DMS-1007960.}

\tableofcontents


\section{Introduction}

Stability analysis is a central issue in the theory of plasmas (e.g., \cite{nicholson-plasma}, \cite{trivelpiece}). 
In the search for practical fusion energy, the tokamak has been the central focus of research for many years.  
The classical tokamak has two features, the toroidal geometry and a mechanism (magnetic field, laser beams) to confine the plasma.  
Here we concentrate on the effect of the toroidal geometry on the stability analysis of equilibria.  

When a plasma is very hot (or of low density), electromagnetic forces have a much faster effect on the particles 
than the collisions, so the collisions can be ignored as compared with the electromagnetic forces.  
So such a plasma is modeled by the relativistic Vlasov-Maxwell system (RVM)  
\begin{equation}\label{Vlasov-equations}\left\{\begin{aligned} 
&\dt f^+ + \hat v \cdot \nabla_x f^+ + (\vE + \hat v \times \vB )\cdot \nabla_v f^+   =0,
\\&\dt f^- + \hat v \cdot \nabla_x f^- - (\vE + \hat v \times \vB ) \cdot \nabla_v f^-  =0, 
\end{aligned}\right.
\end{equation}
\begin{equation}\label{Gauss-laws} 
\nabla_x \cdot \vE = \rho,\qquad \nabla_x \cdot \vB =0,
\end{equation}
\begin{equation}\label{Ampere-law}
\dt  \vE  - \nabla_x \times \vB  = - \vj, \quad 
\dt \vB + \nabla_x \times \vE =0,
\end{equation}
$$\rho = \int_{\RR^3} (f^+ - f^-) \;\mathrm{d}v , \qquad \vj = \int_{\RR^3} \hat v (f^+-f^-)\; \mathrm{d}v.$$ 
Here $f^\pm(t,x,v)\ge0$ denotes the density distribution of ions and electrons, respectively,  $x \in \Omega \subset \RR^3$ 
is the particle position, $\Omega$ is the region occupied by the plasma,  $v\in \RR^3$ is the particle momentum, 
$\langle v \rangle = \sqrt{1+|v|^2}$ is the particle energy, 
$\hat v = v/\langle v \rangle$ the particle velocity, $\rho$  the charge density, $\vj$  the current density, 
$\vE$  the electric field, $\vB$  the magnetic field and $\pm(\vE + \hat v \times \vB)$ the electromagnetic force.  For simplicity all the constants have been set equal to 1; however, our results do not depend on this normalization.

The Vlasov-Maxwell system is assumed to be valid inside a solid torus (see Figure \ref{fig-torus}), which we take for simplicity to be 
$$\Omega = \Big\{x = (x_1,x_2,x_3)\in \RR^3~:~ \Big(a - \sqrt{x_1^2+x_2^2}\Big)^2 + x_3^2 < 1 \Big\}.$$  
The specular condition at the boundary is 
\begin{equation}\label{bdry-specular} f^\pm (t,x,v) = f^\pm(t,x,v - 2(v\cdot n(x))n(x)),\qquad n(x)\cdot v <0,\qquad x \in \D \Omega, \end{equation}
where $n(x)$ denotes the outward normal vector of $\D \Omega$ at $x$. 
The perfect conductor boundary condition is 
\begin{equation}\label{bdry-EBcond} 
\vE(t,x)\times n(x) = 0, \qquad \vB(t,x) \cdot n(x)  =0 ,\qquad x \in \D \Omega.\end{equation}
A fundamental property of RVM with these  boundary conditions is that the total energy
$$\mathcal{E}(t)  = \int_\Omega \int_{\RR^3}\langle v \rangle (f^+ + f^-)\; dvdx 
+ \frac 12 \int_\Omega \Big( |\vE|^2 + |\vB|^2\Big)\; dx $$ 
is conserved in time.  In fact, the system  admits infinitely many equilibria.  
{\em The main focus of the present paper is to investigate the stability properties of the equilibria.}


Our analysis is closely related to the spectral analysis approach in \cite{LS1,LS3} which tackled the stability problem 
in domains without any spatial boundaries.  
A first such analysis in a domain with boundary appears in \cite{NStr1}, which treated a 2D plasma inside a circle.  
Roughly speaking, these papers provided a sharp stability  
criterion $\mathcal{L}^0\ge 0$, where $\mathcal{L}^0$ is a certain nonlocal self-adjoint operator 
that acts merely on scalar functions depending only on the spatial variables.  
This positivity condition was verified explicitly for a number of interesting examples. 
It may also be amenable to numerical verification.   
Now, in the presence of a boundary, every integration by parts brings in boundary terms 
and the curvature of the torus plays an important role.  
We consider a certain class of equilibria and make some symmetry assumptions, 
which are spelled out in the next two subsections.  Our main theorems are stated in the third subsection.  


Of course, this paper is a rather small step in the direction of mathematically understanding a confined plasma.  
Most stability studies 
(\cite{friedberg-mhd}, \cite {garabedian}, \cite{goedbloed}, \cite{han-kwan}, \cite{white}) 
are based on macroscopic  MHD or other approximate fluids-like models.   
But because many plasma instability phenomena have an essentially microscopic nature, 
kinetic models like Vlasov-Maxwell are required.  The Vlasov-Maxwell system is a rather
accurate description of a plasma when collisions are negligible, as occurs for instance in a hot plasma.  
The methods of this paper should also shed light on approximate models like MHD.

Instabilities in Vlasov plasmas reflect the collective behavior of all the particles. Therefore the
instability problem is highly nonlocal and is difficult to study analytically and numerically.
In most of the physics literature on stability (e.g., \cite{trivelpiece}),  only a {\it homogeneous} equilibrium with vanishing
electromagnetic fields is treated, in which case there is a dispersion
relation that is rather easy to study analytically.
The classical result of this type is Penrose's sharp linear
instability criterion (\cite{Pen}) for a homogeneous equilibrium of the Vlasov-Poisson system.
Some further papers on the stability problem, including nonlinear stability, for 
general inhomogeneous equilibria of the Vlasov-Poisson system can be found 
in \cite{rein94},  \cite{gs2},  \cite{gs4}, \cite{gs5}, \cite{BRV} and \cite{lin01}.  
Among these papers the closest analogue to our work in a domain with specular boundary conditions is \cite{BRV}.  

However, as soon as magnetic effects are included and even for a homogeneous
equilibrium, the stability problem becomes quite complicated, as for the Bernstein modes in a
constant magnetic field \cite{trivelpiece}.
The stability problem for
inhomogeneous (spatially-dependent) equilibria with nonzero electromagnetic fields is yet more
complicated and so far there are relatively few rigorous results, namely, 
\cite{Guo97}, \cite{Guo99}, \cite{gs7}, \cite{gs6}, \cite{LS1}, \cite{LS3} and \cite{NStr1}.  
We have already mentioned \cite{LS1} and \cite{LS3}, which are precursors of our work in the absence of a boundary.  
Among these papers the only ones that treat domains with boundary are \cite{Guo99} and \cite{NStr1}.  
In his important paper \cite{Guo99}, 
Guo uses a variational formulation to find conditions that are sufficient for nonlinear stability 
in a class of bounded domains that includes a torus with the specular and perfect conductor boundary conditions.  
The class of equilibria in  \cite{Guo99} is  less general than ours.  The stability condition omits several terms so that 
it is far from being a necessary condition.   Our recent paper for a plasma in a disk (\cite{NStr1}) is a precursor of 
our current work but is restricted to two dimensions.



\begin{figure}[ht]  \centering   \includegraphics[scale=.8]{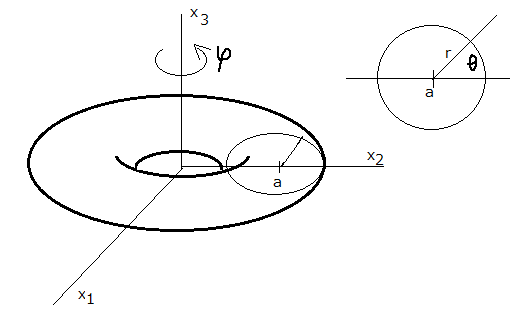}
\caption{{\em The picture illustrates the simple toroidal geometry. \label{fig-torus}}}  \end{figure}


\subsection{Toroidal symmetry}

We shall work with the simple toroidal coordinates $(r,\theta,\varphi)$ with 
$$\begin{aligned}
 x_1 = (a+ r\cos \theta ) \cos \varphi , \quad x_2 = (a+r \cos \theta) \sin \varphi, \quad x_3=  r \sin \theta.\end{aligned}$$
Here $0\le r\le 1$ is the radial coordinate in the minor cross-section, $0\le \theta<2\pi$ is the poloidal angle, 
and $0\le \varphi <2\pi$ is the toroidal angle; see Figure \ref{fig-torus}. 
For simplicity we have chosen the minor radius to be 1 and called the major radius $a>1$.
We denote the corresponding unit vectors by 
$$\left\{ \begin{aligned} e_r &= (\cos \theta \cos \varphi,\cos \theta \sin \varphi, \sin \theta),
\\e_\theta &= (-\sin \theta \cos \varphi, -\sin \theta \sin \varphi, \cos \theta), 
\\
e_\varphi &= (-\sin \varphi,\cos \varphi,0). \end{aligned}\right.$$
Of course,  $e_r (x)= n(x)$ is the outward normal vector at $x\in\D\Omega$, and we note that   
 $$ e_\theta \times e_r = e_\varphi, \qquad e_r \times e_\varphi = e_\theta, \qquad e_\varphi \times e_\theta = e_r.$$ 
In the sequel, we write 
$$v = v_re_r + v_\theta e_\theta + v_\varphi e_\varphi, \quad \vA = A_re_r +A_\theta e_\theta +  A_\varphi e_\varphi .$$  
Throughout the paper it will be convenient to denote by $\tilde\RR^2$ the subspace in $\RR^3$ that consists of the vectors 
orthogonal to $e_\varphi$. The subspace $\tilde \RR^2$ depends on the toroidal angle $\varphi$. We denote by $\tilde v, \tilde \vA$ the projection of $v,\vA$ onto $\tilde \RR^2$, respectively, and we write $\tilde v = v_r e_r + v_\theta e_\theta$ and 
$\tilde \vA = A_r e_r + A_\theta e_\theta$.

It is convenient and standard when dealing with the Maxwell equations to introduce the electric scalar potential $\phi$ 
and magnetic vector potential $\vA$ through  
\begin{equation}\label{def-potentials} \vE = -\nabla \phi - \D_t \vA, \qquad  \vB = \nabla \times \vA , \end{equation}
in which without loss of generality we impose the Coulomb gauge  $\nabla \cdot \vA =0$.  
Throughout this paper, we  assume {\it toroidal symmetry}, which means that all four  potentials 
$\phi, A_r, A_\theta,A_\varphi$ are independent of $\varphi$. 
In addition, we assume that the density distribution $f^\pm $ has the form $f^\pm (t,r,\theta,v_r,v_\theta,v_\varphi)$. 
That is, $f$ does not depend explicitly on $\varphi$, although it does so implicitly through the components of $v$, which depend on the basis vectors.  
Thus, although in the toroidal coordinates all the functions are independent of the angle $\varphi$, 
the unit vectors $e_r,e_\theta,e_\varphi$ and therefore the toroidal components of $v$ do depend on $\varphi$. Such a symmetry assumption leads to a partial decoupling of the Maxwell equations and is fundamental throughout the paper.

\subsection{Equilibria}\label{sec-equilibria}
We denote an (time-independent) equilibrium by $(f^{0,\pm},\vE^0,\vB^0)$. We assume that the equilibrium magnetic field $\vB^0$ has no component in the $e_\varphi$ direction.  
Precisely, the equilibrium field has the form 
\begin{equation}\label{def-EB0}
\begin{aligned}
\vE^0 &= -\nabla \phi^0 =   - e_r \frac {\D \phi^0}{\D r} - e_\theta  \frac 1r \frac {\D \phi^0}{\D \theta} ,
\\ \vB^0 &= \nabla \times \vA^0 = - \frac{e_r}{r(a+r\cos \theta)} \frac {\D}{\D \theta}( (a+r\cos \theta)  A^0_\varphi ) + \frac{e_\theta}{a+r\cos \theta}  \frac {\D}{\D r} ((a+r\cos \theta)A^0_\varphi) 
\end{aligned}
\end{equation} 
with $\vA^0 = A^0_\varphi e_\varphi$ and $B^0_\varphi =0$.   
Here and in many other places it is convenient to consult the vector formulas that are collected in Appendix A.  


As for the particles, we observe that their energy and angular momentum  
\begin{equation} \label{ep}
e^\pm (x,v):= \langle v \rangle  \pm \phi^0(r,\theta), \qquad p^\pm(x,v) := (a+r\cos \theta)(v_\varphi \pm  A_\varphi^0(r,\theta)),  
\end{equation} 
are invariant along the particle trajectories.  
By direct computation, $\mu^\pm(e^\pm,p^\pm)$ solve the Vlasov equations for any pair of smooth functions 
$\mu^\pm$ of two variables.   The equilibria we consider have  the form  
\begin{equation}\label{f-equil} 
f^{0,+}(x,v) = \mu^+ (e^+(x,v),p^+(x,v)),\quad f^{0,-}(x,v) = \mu^- (e^-(x,v),p^-(x,v)).  
\end{equation}
Let $(f^{0,\pm},\vE^0,\vB^0)$ be an equilibrium as just described with $f^{0,\pm} = \mu^\pm(e^\pm,p^\pm)$. 
We assume that $\mu^\pm(e,p)$ are nonnegative $C^1$ functions which satisfy 
\begin{equation}\label{mu-cond} 
\mu^\pm_e (e,p)<0 ,\qquad |\mu^\pm|+ |\mu_p^\pm(e,p)| + |\mu_e^\pm(e,p)| 
\ \le\ \frac{ C_\mu}{1+|e|^\gamma}
\end{equation}
for some constant $C_\mu$ and some $\gamma>3$, where the subscripts $e$ and $p$ denote the partial derivatives.  The decay assumption is to ensure that $\mu^\pm$ and its partial derivatives are $v$-integrable. 

What remains are the Maxwell equations for the equilibrium.  
In terms of the potentials, they take the form 
\begin{equation}\label{Maxwell-equi}\begin{aligned} - \Delta \phi^0 &= \int_{\RR^3} (\mu^+ - \mu^-)\; \mathrm{d}v ,\\
- \Delta A^0_\varphi  +  \frac{1}{(a+r\cos \theta)^2}A^0_\varphi  &=  \int_{\RR^3}\hat v_\varphi (\mu^+ - \mu^-)\; \mathrm{d}v .\end{aligned}\end{equation}
We assume that $\phi^0$ and $A^0_\varphi$ are continuous in $\overline{\Omega}$. In Appendix \ref{App-equilibria}, we will show that $\phi^0$ and $A^0_\varphi $ are in fact in $C^2(\overline{\Omega})$ and so $\vE^0,\vB^0 \in C^1(\overline \Omega)$.

As for the assumption that $B^0_\varphi =0$,  it is sufficient to assume that $B^0_\varphi =0$ 
on the boundary of the torus.    Indeed, since $f^{0,\pm}$ is even in $v_r$ and $v_\theta$ (being a function of $e^\pm, p^\pm)$, 
it follows that $j_r$ and $j_\theta$ vanish for the equilibrium, and therefore by \eqref{Maxwell-Bphi} below, $B^0_\varphi$ solves 
\begin{equation}\label{eq-Bphi}- \Delta B^0_\varphi  +  \frac{1}{(a+r\cos \theta)^2}B^0_\varphi  =0.\end{equation}


\subsection{Spaces and operators} \label{sec-spaces}
We will consider the Vlasov-Maxwell system linearized around the equilibrium. Let us denote by $\oD ^\pm$ the first-order linear differential operator: 
\begin{equation}\label{def-Dpm}\oD^\pm = \hat v \cdot \nabla_x  \pm (\vE^0 + \hat v\times \vB^0)\cdot \nabla_v. \end{equation}
The linearization is then 
\begin{equation}  \label{linearized-sys} 
 \begin{aligned}
 \dt f^\pm+\oD ^\pm f^\pm  &=  \mp (\vE + \hat v\times \vB)\cdot \nabla_v f^{0,\pm} ,
\end{aligned} 
\end{equation}
together with the Maxwell equations and the specular and perfect conductor boundary conditions.  

In order to state precise results, we have to define certain spaces and operators.    
We denote by $\cH^\pm = L^2_{|\mu_e^\pm|}(\Omega \times \RR^3)$ the weighted $L^2$ space consisting of functions 
$f^\pm(x,v)$ which are toroidally symmetric in $x$ such that 
$$ \int_\Omega   \int_{\RR^3}|\mu_e^\pm| |f^\pm|^2 \; \mathrm{d}v \mathrm{d}x <+\infty. $$ 
The main purpose of the weight function is to control the growth of $f^\pm$ as $|v| \to\infty$.  
Note that 
due to the assumption \eqref{mu-cond} the weight $|\mu_e^\pm|$ never vanishes and it decays like a power of $v$ as $|v| \to\infty$.  
When there is no danger of confusion, we will 
write $\cH=\cH^\pm$.  

For $k\ge 0$  
we denote by $H^k_\tau  (\Omega  )$ the usual $H^k$ space on $\Omega $ 
that consists of scalar functions that are toroidally symmetric. 
If $k=0$ we write  $L^2_\tau  (\Omega  )$. In addition, we shall denote by $H^k(\Omega; \RR^3)$ 
 the analogous space of vector functions.  
By $\mathcal{X}$ we denote the space consisting of the (scalar) functions in $H^2_\tau (\Omega   )$ which satisfy 
the Dirichlet boundary condition.  We will sometimes drop the subscript $\tau$, although all functions are assumed to be toroidally symmetric.

{\it We denote by $\mathcal{P}^\pm$ the orthogonal projection on the kernel of $\oD ^\pm$ 
in the weighted  space $\cH^\pm$.}  In the  spirit of \cite{LS1,LS3},  our main results involve three linear operators
on $L^2(\Omega )$, two of which are unbounded, namely, 
\begin{equation}\label{operators-0}
\begin{aligned}  \mathcal{A}_1^0 h & =  \Delta h  +  \sum_\pm \int_{\RR^3}\mu^\pm_e(1-\mathcal{P}^\pm) h \; \mathrm{d}v , 
\\
 \mathcal{A}_2^0 h & = - \Delta h + \frac{1}{(a+r\cos\theta)^2} h  - \sum_\pm  \int_{\RR^3}\hat v_\varphi \Big [ (a+r\cos\theta) \mu_p^\pm   h  
 +  \mu_e^\pm \mathcal{P}^\pm(\hat v_\varphi h)  \Big] \; \mathrm{d}v , 
\\
\mathcal{B}^0 h  &= - \sum_\pm  \int_{\RR^3}\hat v_\varphi\mu_e^\pm ( 1 - \mathcal{P}^\pm) h \; \mathrm{d}v . 
\end{aligned}\end{equation}
Here $\mu^\pm$ is shorthand for $\mu^\pm(e^\pm,p^\pm) $. 
Note the opposite signs of $\Delta$ in $ \mathcal{A}_1^0$ and $ \mathcal{A}_2^0$.   
Both of these operators have the domain $\mathcal{X}$.  
We will show in Section \ref{sec-instability} that all three operators are naturally derived from the Maxwell equations when $f^+$ and $f^-$ are written in integral 
form by integrating the Vlasov equations along the trajectories. 
In particular, in Section \ref{sec-op} we will show that both $\mathcal{A}_1^0$ and $\mathcal{A}_2^0$ 
with domain  $\mathcal{X}$ are self-adjoint operators on $L^2_\tau  (\Omega )$.   
Furthermore, the inverse of $\mathcal{A}_1^0$ is well-defined on $L^2_\tau(\Omega)$, and so we are able 
to introduce our key operator 
\begin{equation}\label{operator-L0}
\mathcal{L}^0 = \mathcal{A}_2^0 -  \mathcal{B}^0 (\mathcal{A}_1^0)^{-1} (\mathcal{B}^0)^*, \end{equation} with $ (\mathcal{B}^0)^*$ being the adjoint operator of $ \mathcal{B}^0$ in $L^2_\tau(\Omega)$. 
The operator $\mathcal{L}^0$ will then be self-adjoint on $L^2_\tau  (\Omega )$ with its domain $\mathcal{X}$.   
As the next theorem states, $\mathcal{L}^0 \ge 0$  is the condition for stability.  This condition means that 
$(\mathcal{L}^0h,h)_{L^2} \ge 0$ for all $h \in \mathcal{X}$.  

Finally, by a {\it growing mode} we mean a solution of the linearized system (including the 
boundary conditions) of the form 
$(e^{\lambda t} f^\pm, e^{\lambda t}{\vE}, e^{\lambda t} {\vB})$ with $\R\lambda>0$ 
such that $f^\pm \in \cH^\pm$ and $\vE, \vB\in L_\tau^2(\Omega ; \RR^3)$.  
The derivatives and the boundary conditions are considered in the weak sense, which will be 
 justified in Lemma \ref{growprops}.  
In particular, the weak meaning of the specular condition on $f^\pm$ will be given by \eqref{skewadj}.

\bigskip  
\subsection{Main results}
The first main result provides a necessary and sufficient condition for linear stability in the spectral sense.   
\begin{theorem}\label{theo-main} 
Let $(f^{0,\pm},\vE^0,\vB^0)$ be an equilibrium of the Vlasov-Maxwell system satisfying \eqref{f-equil} and \eqref{mu-cond}. Assume that $\mu^\pm \in C^1(\RR^2)$ and $\phi^0, A^0_\varphi \in C(\overline\Omega)$. 
Consider the linearization \eqref{linearized-sys}. 
%
Then

(i) if  $\mathcal{L}^0\ge 0$, there exists no growing mode of the linearized system;   

(ii) any growing mode, if it does exist, must be purely growing; that is, the exponent of instability must be real; 

(iii) if $\mathcal{L}^0\not\ge 0$, there exists a growing mode.

\end{theorem}

Our second main result provides explicit examples for which the stability condition does or does not hold. 
For more precise statements of this result, see Section \ref{sec-examples}.

\begin{theorem} \label{theo-examples} Let $(\mu^{\pm},\vE^0,\vB^0)$ be an equilibrium as above.   

(i) The condition $p \mu^\pm_p(e,p)\le 0$ for all $(e,p)$ implies $\mathcal{L}^0 \ge 0$, 
provided that 
$A_\varphi^0$ is sufficiently small in $L^\infty(\Omega)$.     (So such an equilibrium is linearly stable.) 

(ii) The condition $|\mu_p^\pm (e,p)|  \le  \frac \epsilon {1+|e|^\gamma} $ for some $\gamma>3$ and for $\epsilon$ sufficiently small implies $\mathcal{L}^0 \ge 0$. Here $A_\varphi^0$  is not necessarily small.  (So such an equilibrium is linearly stable.) 
 
(iii) The conditions $\mu^+(e,p) = \mu^-(e,-p)$ and $p \mu^-_p(e,p)\ge c_0 p^2 \nu(e)$, for some 
nontrivial nonnegative function $\nu(e)$, imply that for a suitably scaled version of  $(\mu^\pm,0,\vB^0)$,  
 $\mathcal{L}^0\ge 0$ is violated.  (So such an equilibrium is linearly unstable.)  
\end{theorem}


Theorems \ref{theo-main} and \ref{theo-examples} are concerned with the {\it linear} stability and instability of equilibria. However, their nonlinear counterparts remain as an outstanding open problem. 
In the full nonlinear problem singularities might occur at the boundary, and the particles could repeatedly bounce off the boundary, which makes it difficult to analyze their trajectories; see \cite{Guo93}.
For the periodic $1\frac12$D problem in the absence of a boundary, \cite{LS2} proved the {\it nonlinear instability} 
of equilibria by using a very careful analysis of the trajectories and a delicate duality argument to show that the linear behavior is dominant.  It would be a difficult task to use this kind of argument to handle our higher-dimensional case with trajectories that reflect at the boundary but it is conceivable.   As for {\it nonlinear stability}, it could definitely not be proven from linear stability, 
as is well-known even for very simple conservative systems. The nonlinear invariants must be used directly and the nonlinear particle trajectories must be analyzed in detail.  
Even the simpler case studied in \cite{LS2} required an intricate proof to handle a special class of equilibria.  
Another natural question that we do not address in this paper is the well-posedness of the nonlinear initial-value problem.
 It is indeed a famous open problem in 3D, even in the  case without a boundary.

In Section 2 we write the whole system explicitly in the toroidal coordinates.  
The boundary conditions are given in Section \ref{sec-bdry}.  The specular condition is expressed in the weak form $\langle \oD^\pm f , g \rangle_\cH = - \langle f, \oD^\pm g \rangle_\cH$, for all toroidally symmetric $C^1$ functions $g$ with $v$-compact support that satisfy the specular condition.  Section 3 is devoted to the proof of stability under the condition $\mathcal L^0\ge0$,  notably using the time
 invariants, namely the generalized energy $\mathcal{I}(f^\pm,\vE,\vB)$ and the casimirs $\mathcal{K}_g(f^\pm,\vA)$.  A key lemma involves the minimization of the energy with the magnetic potential being held fixed. Using the minimizer, we find a key inequality \eqref{key-ineq} which leads to the non-existence of growing modes.  It is also shown, via a proof that is considerably simpler than the 
 ones in \cite{LS1,LS3}, that any growing mode must be pure; that is, the exponent $\lambda$ of instability is real.  
Our proof of instability in Section 4 makes explicit use of the particle trajectories to construct a family of operators 
$\mathcal L^\lambda$ that approximates $\mathcal L^0$ as $\lambda \to 0$ as in \cite{LS1,LS3}; see Lemma \ref{lem-Mprop}. It is a rather complicated argument which involves a careful analysis of the various components.  The trajectories reflect a countable number of times at the boundary of the torus,  like billiard balls.  An important property is the self-adjointness of $\mathcal{L}^\lambda$ and some associated operators; see Lemma \ref{lem-self-adj}. We employ Lin's continuity method \cite{lin01} which interpolates between $\lambda =0$ and $\lambda = \infty$. However, it is not necessary to employ a magnetic super-potential as in \cite{LS3}.  
The whole problem is reduced to finding a null vector of a matrix of operators in equation \eqref{matrix-Maxwell} and its reduced form $\mathcal{M}^\lambda$ in equation \eqref{matrix-Maxwell01}. We also require in Subsection 4.5 
a truncation of part of the operator to finite dimensions.  
In Section 5 we provide some examples where we verify the stability criteria
explicitly.

\section{The symmetric system}

\subsection{The equations in toroidal coordinates}
We define the electric and magnetic potentials $\phi$ and $\vA$ through \eqref{def-potentials}.  
Under the toroidal symmetry assumption, the fields  take the form 
\begin{equation}\label{def-EB}
\begin{aligned}
\vE &=  - e_r  \Big[ \frac {\D \phi}{\D r}  + \D_t A_r\Big] - e_\theta \Big[  \frac 1r \frac {\D \phi}{\D \theta}  + \D_t A_\theta \Big] -  e_\varphi \D_t  A_\varphi, 
\\ \vB &= - \frac{e_r}{r(a+r\cos \theta)} \frac {\D}{\D \theta}( (a+r\cos \theta)   A_\varphi ) + \frac{e_\theta}{a+r\cos \theta}  \frac {\D}{\D r} ((a+r\cos \theta) A_\varphi) + e_\varphi B_\varphi,
\end{aligned}
\end{equation} with $B_\varphi = \frac{1}{r} [ \D_\theta A_r - \D_r(rA_\theta)]$. 
We note that (\ref{def-EB}) implies \eqref{def-potentials}, which implies 
    the two Maxwell equations 
$$\dt \vB + \nabla \times \vE =0 ,\qquad \nabla \cdot \vB = 0.   $$
The remaining two Maxwell equations become
\begin{equation}\label{Maxwell-pots} - \Delta \phi = \rho ,\qquad \quad \partial_t^2 \vA  - \Delta \vA + \partial_t \nabla \phi  = \vj.\end{equation}
We shall study this form \eqref {Maxwell-pots}  of the  Maxwell equations 
coupled to the Vlasov equations \eqref{Vlasov-equations}. 
By direct calculations (see Appendix \ref{app-cal}), 
we observe that $\Delta \tilde \vA \in \mbox{span}\{e_r,e_\theta\}$ 
and $\Delta (A_\varphi e_\varphi) = \Big[\Delta  A_\varphi  -  \frac{1}{(a+r\cos \theta)^2} A_\varphi \Big] e_\varphi$.  
The Maxwell equations in \eqref{Maxwell-pots} thus become
\begin{equation}\label{Maxwell-system} \begin{aligned}
- \Delta \phi &= \rho 
\\
\Big( \D_t^2 - \Delta +  \frac{1}{(a+r\cos \theta)^2}\Big)  A_\varphi  &= j_\varphi
\\
\partial_t^2 \tilde \vA  - \Delta \tilde \vA + \partial_t \nabla \phi  &= \tilde \vj.
\end{aligned}\end{equation}
Here we have denoted $\tilde \vA = A_re_r + A_\theta e_\theta$ and $ \tilde \vj = j_r e_r + j_\theta e_\theta$.  
It is  interesting to note that  $B_\varphi = \frac 1r \Big[ \D_\theta  A_r - \D_r(rA_\theta)\Big ]$ satisfies the equation 
\begin{equation}\label{Maxwell-Bphi} \Big( \D_t^2 - \Delta + \frac{1}{(a+r\cos \theta)^2} \Big) B_\varphi  = \frac 1r \Big (\D_\theta j_r - \D_r (r j_\theta) \Big ).\end{equation}
Next we write the equation for the density distribution $f^\pm = f^\pm (t,r,\theta,v_r,v_\theta,v_\varphi)$.  
In the toroidal coordinates we have 
$$\begin{aligned}
\hat v\cdot\nabla_x f  &=  \hat v_r \frac{\D f}{\D r}  + \frac 1r \hat v_\theta \frac{\D f}{\D \theta}  + \frac 1{a+r\cos \theta} \hat v_\varphi \frac{\D f}{\D \varphi }  
\\&= \hat v_r \D_rf  + \frac 1r \hat v_\theta \D_\theta f  + 
\frac 1r  \hat v_\theta \Big\{ v_\theta \partial_{v_r} - 
v_r\partial_{v_\theta} \Big\} f 
\\&\quad +  \frac 1{a+r\cos \theta} \hat v_\varphi  \Big\{\cos \theta v_\varphi \D_{v_r}   - \sin \theta v_\varphi \D_{v_\theta}   - (\cos \theta v_r - \sin \theta v_\theta) \D_{v_\varphi} \Big\} f. \end{aligned}$$
Thus in these coordinates the Vlasov equations become   
\begin{equation}\label{VMsys}
\begin{aligned}
\dt f^\pm &+\hat v_r \D_rf^\pm  + \frac 1r \hat v_\theta \D_\theta f^\pm \pm \Big(E_r + \hat v_\varphi B_\theta - \hat v_\theta B_\varphi \pm \frac 1r v_\theta \hat v_\theta \pm \frac{\cos \theta}{a+r\cos \theta} v_\varphi \hat v_\varphi\Big)\D_{v_r}f^\pm
\\&\pm  \Big(E_\theta + \hat v_r B_\varphi  - \hat v_\varphi B_r\mp \frac 1r v_r \hat v_\theta  \mp \frac{\sin\theta}{a+r\cos \theta} v_\varphi \hat v_\varphi\Big)\D_{v_\theta}f^\pm
\\&\pm  \Big(E_\varphi + \hat v_\theta B_r - \hat v_r B_\theta \mp \frac{\cos \theta v_r - \sin\theta v_\theta}{a+r\cos \theta} \hat v_\varphi\Big)\D_{v_\varphi}f^\pm =0.
\end{aligned}
\end{equation}

\subsection{Boundary conditions}\label{sec-bdry}
Since $e_r(x)$ is the outward normal vector, the specular condition \eqref{bdry-specular} now becomes 
\begin{equation}\label{bdry-specular01} f^\pm (t,r,\theta,v_r,v_\theta,v_\varphi) = f^\pm(t,r,\theta, -v_r , v_\theta, v_\varphi),\qquad v_r <0,\qquad x \in \D \Omega. \end{equation}
The perfect conductor boundary condition is 
$$
E_\theta = 0, \qquad E_\varphi = 0, \qquad B_r  =0 , \qquad x \in \D \Omega,$$
or equivalently, $$ \D_\theta \phi + \D_t A_\theta  =0, \qquad  A_\varphi = \frac{const.}{a+\cos \theta}.$$
Desiring a time-independent boundary condition, we take $\phi = const.$ on the boundary. 
The Coulomb gauge gives an extra boundary condition for the potentials:  
$$ (a+\cos \theta) \D_r A_r + (a+2\cos \theta )A_r  + \D_\theta ((a+\cos \theta)A_\theta) =0 ,\qquad x \in \D \Omega, $$ 
which leads us to assume 
$$ (a+\cos \theta) \D_r A_r + (a+2\cos \theta )A_r  =0, \qquad A_\theta =\frac{const.}{a+\cos \theta} ,\qquad x \in \D \Omega, $$ 
To summarize, 
we are assuming  that the potentials on the boundary $\D\Omega$ satisfy 
\begin{equation}\label{Maxwell-bdry}
\phi = const., \qquad A_\varphi = \frac{const.}{a+\cos \theta}, \qquad A_\theta =\frac{const.}{a+\cos \theta},\qquad  \D_r A_r + \frac{a+2\cos \theta }{a+\cos \theta}A_r  =0.
\end{equation}


\subsection{Linearization} \label{sec-linVM} 
We linearize the Vlasov-Maxwell system \eqref{Maxwell-system} and \eqref{VMsys} around the equilibrium $(f^{0,\pm},\vE^0,\vB^0)$.  Let $(f^\pm,\vE,\vB)$ be the perturbation of the equilibrium.  Of course, the linearization of \eqref{Maxwell-system} remains the same. From \eqref{VMsys}, the linearized 
Vlasov equations are 
\begin{equation}  \label{linearization} 
 \begin{aligned}
 \dt f^\pm +\oD^\pm f^\pm  &=  \mp (\vE + \hat v\times \vB)\cdot \nabla_v f^{0,\pm} .
\end{aligned}                        \end{equation}
The first-order differential operators 
$\oD^\pm = \hat v \cdot \nabla_x  \pm (\vE^0 + \hat v\times \vB^0)\cdot \nabla_v $ 
now take the form 
\begin{equation}\label{def-opD}
\begin{aligned}
\oD^\pm &=\hat v_r \D_r + \frac 1r \hat v_\theta \D_\theta  \pm \Big(E_r^0 + \hat v_\varphi B^0_\theta  \pm \frac 1r v_\theta \hat v_\theta \pm \frac{\cos \theta}{a+r\cos \theta} v_\varphi \hat v_\varphi\Big)\D_{v_r} \\&\quad \pm  \Big(E_\theta^0 - \hat v_\varphi B^0_r\mp \frac 1r v_r \hat v_\theta  \mp \frac{\sin\theta}{a+r\cos \theta} v_\varphi \hat v_\varphi\Big)\D_{v_\theta}
\\&\quad \pm  \Big( \hat v_\theta B^0_r - \hat v_r B^0_\theta \mp \frac{\cos \theta v_r - \sin\theta v_\theta}{a+r\cos \theta} \hat v_\varphi\Big)\D_{v_\varphi} .
\end{aligned}
\end{equation}
The last terms on the last two lines come from $\nabla_x$ acting on the basis vectors in $v=v_re_r+v_\theta e_\theta+v_\varphi e_\varphi$.  
 Note that $\oD^\pm$ are odd in the pair $(v_r,v_\theta)$. 
Let us compute the right-hand sides of \eqref{linearization} more explicitly. 
Differentiating the equilibrium $f^{0,\pm} = \mu^\pm(e^\pm,p^\pm)$ with respect to $v$,  we then have 
$$\begin{aligned}
 \mp (\vE + \hat v\times \vB)\cdot \nabla_v f^{0,\pm}  
 &= \mp (\vE + \hat v\times \vB)\cdot ( \mu_e^\pm \hat v + (a+r\cos\theta) \mu_p^\pm e_\varphi)
 \\
 &= \mp \mu_e^\pm \vE \cdot \hat v \mp  (a+r\cos\theta) \mu_p^\pm (E_\varphi +  \hat v_\theta B_r - \hat v_r B_\theta), 
 \end{aligned}$$
in which we use the form \eqref{def-EB} of the fields, $E_\varphi = -\partial_t A_\varphi$ and 
$$(a+r\cos\theta) (\hat v_\theta B_r - \hat v_r B_\theta) = - \hat v \cdot \nabla ((a+r\cos \theta)  A_\varphi)  = - \oD^\pm ((a+r\cos \theta)  A_\varphi) .$$ 
Thus the {\it linearization} \eqref{linearization} takes the explicit form  
\begin{equation}\label{lin-VM} \begin{aligned}
 \dt f^\pm+\oD ^\pm f^\pm   = \mp \mu_e^\pm \vE \cdot \hat v  \pm (\D_t + \oD^\pm ) \Big ((a+r\cos \theta) \mu_p^\pm  A_\varphi\Big ) ,
 \end{aligned}
\end{equation}
which is accompanied by the Maxwell equations in \eqref{Maxwell-system}, 
the specular boundary condition \eqref{bdry-specular01} on $f^\pm$, and the boundary conditions 
\begin{equation}\label{Maxwell-BCs}
\phi = 0, \qquad A_\varphi = 0, \qquad A_\theta =0,\qquad  \D_r A_r + \frac{a+2\cos \theta }{a+\cos \theta}A_r  =0.
\end{equation}
The last two boundary conditions are sometimes written as $0 = e_\theta\cdot\tilde \vA  =  \nabla_x\cdot((e_r\cdot\tilde \vA)e_r)$. Often throughout the paper, we set \begin{equation}\label{def-F} F^\pm : =  f^\pm \mp (a+r\cos \theta) \mu_p^\pm  A_\varphi ,\end{equation} and abbreviate the equation \eqref{lin-VM} as
\begin{equation}\label{lin-VMF} \begin{aligned}
 ( \D_t &+\oD ^\pm ) F^\pm  
 =  \mp \mu_e^\pm \hat v \cdot \vE .
 \end{aligned}
\end{equation}
Note that if $f^\pm$ is specular on the boundary, then so is $F^\pm$ since $\mu_p$ is even in $v_r$.  

\subsection{The Vlasov operators}\label{ss-Vlasov}

The Vlasov operators $\oD ^\pm$ are formally given by \eqref{def-opD}.  Their relationship to the boundary condition 
is given in the next lemma.
\begin{lemma}\label{lem-on-D}
Let $g(x,v)=g(r,\theta, v_r,v_\theta,v_\varphi)$ be a $C^1$ toroidal function on $\overline \Omega \times \RR^3$.  Then $g$ satisfies the specular boundary condition \eqref{bdry-specular01}
if and only if  
$$ \int_\Omega \int_{\RR^3} g\ \oD^\pm h\ \mathrm{d}v\mathrm{d}x  =  -  \int_\Omega \int_{\RR^3} \oD^\pm g\ h\ \mathrm{d}v\mathrm{d}x $$
(either + or -) for all toroidal $C^1$ functions $h$ with $v$-compact support that satisfy the specular condition.   
\end{lemma}
\begin{proof} Integrating by parts in $x$ and $v$, we have 
$$ 
\int_\Omega  \int_{\RR^3} \Big\{g\ \oD ^\pm h  +  \oD^\pm g\ h\Big \} \ \mathrm{d}v\mathrm{d}x =  2\pi \int_0^{2\pi}\int_{\RR^3}(a+\cos \theta) g h\ \hat v \cdot e_r \; \mathrm{d}v\mathrm{d}\theta\Big|_{r=1}.  $$
If $g$ satisfies the specular condition, then $g$ and $h$ are even functions of $v_r=v\cdot e_r$ on $\D\Omega$, 
so that the last integral vanishes.  
Conversely, if $\int_0^{2\pi}\int_{\RR^3}(a+\cos \theta)g h\ \hat v \cdot e_r \; \mathrm{d}v\mathrm{d}\theta=0$ on $\partial\Omega$ 
for all $h$ as above, then 
$\int_0^{2\pi}\int_{\RR^3}g\ k(\theta,v_r,v_\theta,v_\varphi) \mathrm{d}v\mathrm{d}\theta =0$ for all test functions $k$ that are odd in $v_r$.  
So it follows that $g(1,\theta,v_r,v_\theta, v_\theta)$ is an even function of $v_r$, which is the specular condition.  
\end{proof}

Therefore we {\it define} the domain of $\oD^\pm$ to be 
\begin{equation}\label{domainD} 
 \text{dom}(\oD^\pm)  = \Big \{g\in\cH ~\Big|~\ \oD^\pm g \in \cH,\ 
 \langle\oD^\pm g,h\rangle_\cH  =  -  \langle g,\oD^\pm h\rangle_\cH, \  
\forall h\in \mathcal C \Big\},  
\end{equation}
where $\mathcal C$ denotes the set of toroidal $C^1$ functions $h$ with $v$-compact support that satisfy the specular condition.  
We say that a function $g\in\cH$ with $ \oD^\pm g \in \cH$ satisfies the {\it specular boundary condition in the weak sense} 
if $g\in$ dom$(\oD^\pm)$.  Clearly, $ \text{dom}(\oD^\pm)$ is dense in $\cH$ since by Lemma \ref{lem-on-D} it contains the space $\mathcal{C}$ of test functions, which is of course dense in $\cH$. 

It follows that 
\begin{equation}\label{skewadj}
 \langle\oD^\pm g,h\rangle_\cH  =  -  \langle g,\oD^\pm h\rangle_\cH  \end{equation}
for all $g,h\in \text{dom}(\oD^\pm)$.  Indeed, given $h\in \text{dom}(\oD^\pm)$, we just approximate 
it in $\cH$ by a sequence of test functions in $\mathcal{C}$, and so \eqref{skewadj} holds thanks to Lemma \ref{lem-on-D}. 

Furthermore, with these domains, $\oD^\pm$ are {\it skew-adjoint} operators on $\cH$. 
Indeed, the skew-symmetry has just been stated.  To prove the skew-adjointness of $\oD^\pm$, 
suppose that $f,g\in\cH$ and $ \langle f,h\rangle_\cH  =  -  \langle g,\oD^\pm h\rangle_\cH$  
for all  $h\in \text{dom}(\oD^\pm)$.  
Taking $h\in \mathcal{C}$ to be a test function, we see that $\oD^\pm g = f $ in the sense of distributions.  
Therefore \eqref{skewadj} is valid for all such $h$, which means that $g\in\text{dom}(\oD^\pm)$.  

\subsection{Growing modes}\label{ss-growing}
Now we can state some necessary properties of any growing mode. 
Recall that by definition a growing mode satisfies  $f^\pm \in \cH $ 
and $\vE, \vB\in L^2(\Omega;\RR^3)$. 
\begin{lemma} 
\label{growprops}  
Let $(e^{\lambda t} f^\pm, e^{\lambda t}{\vE}, e^{\lambda t} {\vB})$ with $\R\lambda>0$ 
be any growing mode, and let $F^\pm =  f^\pm \mp (a+r\cos \theta) \mu_p^\pm  A_\varphi$.  Then 
$\vE,\vB\in H^1(\Omega;\RR^3)$ and 
$$
\iint_{\Omega\times \RR^3} \ (|F^\pm|^2  +  |\oD^\pm F^\pm|^2)\  \frac {\mathrm{d}v\mathrm{d}x}{|\mu_e^\pm|}  \quad <\quad   \infty.$$
\end{lemma}
\begin{proof}
The fields are given by  \eqref{def-EB} where $\phi, \vA$ satisfy the elliptic system \eqref{Maxwell-system} 
with the corresponding boundary conditions, expressed weakly.  
From \eqref{lin-VMF}, $F^\pm$ solves 
\begin{equation}\label{Vlasov3} \begin{aligned}
 ( \lambda &+\oD ^\pm ) F^\pm  =  \mp \mu_e^\pm \hat v \cdot \vE .
 \end{aligned}
\end{equation}
%
%
%
 This equation implies that $\oD^\pm F^\pm\in \cH$ since $f^\pm \in \cH$, $A_\varphi \in L^2(\Omega)$, and $\sup_x \int_{\RR^3}|\mu_p^\pm|\;\mathrm{d}v<\infty$. The specular boundary condition on $f^\pm$ implies the specular condition on $F^\pm$. This is expressed 
 weakly by saying that $f^\pm, F^\pm \in\text{dom}(\oD^\pm)$.  
Dividing  by $|\mu_e^\pm|$ and defining  $k^\pm  = F^\pm/ |\mu_e^\pm|$,  we write 
the equation in the form $\lambda k^\pm + \oD^\pm k^\pm =\pm \hat v \cdot \vE$, where we note that the right side 
$\hat v \cdot \vE$ belongs to $ \cH = L^2_{ |\mu_e^\pm|}(\Omega\times \RR^3)$, thanks to the decay assumption \eqref{mu-cond}. 

Letting 
$w_\ep =  |\mu_e^\pm|/(\ep  +   |\mu_e^\pm|)$ for $\ep>0$ and $k_\ep = w_\ep k^\pm$, we have 
 $\langle \lambda k_\ep + \oD^\pm k_\ep , k_\ep \rangle_\cH 
= \pm \langle w_\ep \hat v \cdot \vE , k_\ep\rangle_\cH. $  
It easily follows that 
$k_\ep\in\cH$.  In fact, $k_\ep\in\text{dom}(\oD^\pm)$, which means that the specular boundary condition holds 
in the weak sense, so that \eqref{skewadj} is valid for it.  In  \eqref{skewadj} we take both functions to be $k_\ep$ 
and therefore $\langle \oD^\pm k_\ep , k_\ep \rangle_\cH  = 0$.  It follows that 
$$| \lambda| \|k_\ep\|_\cH^2  = | \langle w_\ep \hat v \cdot \vE , k_\ep\rangle_\cH |\le \|\vE \|_{\cH} \|k_\ep\|_{\cH}. $$ 
Letting $\ep\to0$, we infer that $k^\pm \in \cH$ and $\int_\Omega\int_{\RR^3} |F^\pm|^2 / |\mu_e^\pm| \mathrm{d}v\mathrm{d}x < \infty$.    

Now let us consider the elliptic system \eqref{Maxwell-system} for the field (with $\partial_t=\lambda$) 
together with the boundary conditions that are expressed weakly. The right hand sides of this system are 
$$\begin{aligned}
\rho &= \int_{\RR^3} (f^+ - f^-) \;\mathrm{d}v  = \int_{\RR^3} (F^+ - F^-) \; \mathrm{d}v +  (a+r\cos \theta) A_\varphi \int_{\RR^3} (\mu_p^+ + \mu_p^-)\; \mathrm{d}v, 
\\
\vj &= \int_{\RR^3} \hat v (f^+-f^-)\; \mathrm{d}v = \int_{\RR^3} \hat v (F^+ - F^-) \; \mathrm{d}v +  (a+r\cos \theta) A_\varphi \int_{\RR^3} \hat v (\mu_p^+ + \mu_p^-)\; \mathrm{d}v.
\end{aligned}$$ 
Since $\int_\Omega\int_{\RR^3} |F^\pm|^2 / |\mu_e^\pm| \mathrm{d}v\mathrm{d}x < \infty$, $\sup_x \int_{\RR^3}(|\mu_e^\pm| + |\mu_p^\pm|)\; \mathrm{d}v < \infty$  and $A_\varphi \in L^2(\Omega)$,
$\rho$ and $\vj$ 
 are now known to be finite a.e. and to belong to $L^2(\Omega)$.  It follows by standard elliptic theory that $\phi,A_\varphi \in H^2(\Omega), \tilde \vA \in H^2(\Omega;\tilde\RR^2)$, 
and so $\vE, \vB\in H^1(\Omega;\RR^3)$.  This is the first assertion of the lemma.   
Nevertheless, we emphasize that neither $\oD^\pm f^\pm$ nor $\oD^\pm F^\pm$ satisfies the specular boundary condition.  
However, directly from \eqref{Vlasov3} it is now clear that $\int_\Omega\int_{\RR^3} |\oD^\pm F^\pm|^2 / |\mu_e^\pm| \mathrm{d}v\mathrm{d}x < \infty  $.  
\end{proof}

\subsection{Properties of $\mathcal{L}^0$}\label{sec-op}

 In this subsection, we shall derive the basic properties of 
the operators $\mathcal{A}_1^0, \mathcal{A}_2^0$, and $\mathcal{B}^0$ defined in \eqref{operators-0} and of $\mathcal{L}^0$ in \eqref{operator-L0}.  Let us recall that $\mathcal{P}^\pm$ are the orthogonal projections of $\cH$ onto the kernels 
$$\ker(\oD^\pm) = \Big\{ f \in \text{dom}(\oD^\pm) ~\Big|~ \oD^\pm f = 0\Big\}.$$

\begin{lemma}\label{lem-AB0property}~

(i) $\mathcal{A}_1^0$ is self-adjoint and negative definite on $L_\tau^2(\Omega )$ with domain $\mathcal{X}$ and it is a one-to-one map of $\mathcal{X}$ onto $L^2_\tau(\Omega)$. Its inverse $(\mathcal{A}_1^0)^{-1}$ is a bounded operator from $L^2_\tau(\Omega)$ into $\mathcal{X}$. 

(ii) $\mathcal{B}^0$ is a bounded operator on $L^2_\tau(\Omega )$. 

(ii) $\mathcal{A}_2^0$ and $\mathcal{L}^0$ are self-adjoint on $L^2_\tau(\Omega )$ with common domain $\mathcal{X}$. 

\end{lemma}

\begin{proof} By orthogonality, the projection $\mathcal{P}^\pm$ is self-adjoint on $\cH$ and is bounded with operator norm equal to one. It follows that 
$\mathcal{A}_1^0,\mathcal{A}_2^0,$ and $\mathcal{L}^0$ are self-adjoint  on $L^2_\tau(\Omega )$ as long as they are well-defined, and that  the adjoint operator of $\mathcal{B}^0$ is defined as 
$$
(\mathcal{B}^0)^* h  := - \sum_\pm  \int_{\RR^3}\mu_e^\pm ( 1 - \mathcal{P}^\pm) (\hat v_\varphi h) \; \mathrm{d}v 
$$ 
for $h \in L^2_\tau(\Omega)$. Now, for all $h,g\in L^2_\tau (\Omega)$, we have 
\begin{equation}\label{cal-A1} \langle (\mathcal{A}_1^0 - \Delta) h, g \rangle_{L^2}  = - \sum_\pm \langle (1-\mathcal{P}^\pm ) h, g\rangle_\cH .
\end{equation} Since $\mathcal{P}^\pm$ is bounded and $\|h\|_{\cH} \le \Big(\sup_x \int_{\RR^3}|\mu_e^\pm|\; \mathrm{d}v\Big) \| h\|_{L^2} \le C_\mu \| h \|_{L^2}$ for some constant $C_\mu$ as in \eqref{mu-cond}, $\mathcal{A}_1^0 - \Delta$ is bounded from $L^2_\tau(\Omega)$ to itself. Similarly, so are $\mathcal{A}_2^0 + \Delta$ and $\mathcal{B}^0$.  This proves that $\mathcal{A}_1^0$ and $\mathcal{A}_2^0$ are well-defined operators on $L^2_\tau(\Omega )$ with their common domain $\mathcal{X}$. It also proves {\em (ii)}. 

In addition, taking $g = h \in L^2_\tau(\Omega)$ in \eqref{cal-A1} and using the orthogonality of $\mathcal{P}^\pm$ and $1-\mathcal{P}^\pm$, we obtain $\langle (\mathcal{A}_1^0 - \Delta) h, h \rangle_{L^2} =  - \sum_\pm \| (1-\mathcal{P}^\pm) h\|_\cH^2 \le 0$. This,  together with the positivity of $-\Delta$ accompanied by the Dirichlet boundary condition, 
shows that $\mathcal{A}^0_1 \le 0$.  In fact,  
$$ -\langle \mathcal{A}^0_1 h,h\rangle_{L^2} \ge -\langle \Delta h,h\rangle_{L^2} = \| \nabla h \|_{L^2}^2, \qquad \forall~h \in \mathcal{X}.$$
Since $\Omega$ is bounded and $h = 0$ on $\partial\Omega$, the above inequality shows that the inverse of $\mathcal{A}^0_1$ exists and is a bounded operator from $L^2_\tau(\Omega)$ into $\mathcal{X}$. Thus the operator $\mathcal{L}^0 = \mathcal{A}_2^0 -  \mathcal{B}^0 (\mathcal{A}_1^0)^{-1} (\mathcal{B}^0)^*$ is now well-defined and has the same domain as that of $\mathcal{A}_2^0$, which is $\mathcal{X}$. 
\end{proof}

\section{Linear stability}\label{sec-stability}

\subsection{Invariants} \label{sec-inv} 
We will derive the two invariants, namely the linearized energy functional and the casimirs. We begin with the  equation \eqref{lin-VMF} for $F^\pm  =  f^\pm \mp (a+r\cos \theta) \mu_p^\pm  A_\varphi . $ Multiplying it by $\frac{1}{|\mu_e^\pm|} F^\pm  $, we get 
$$\begin{aligned}
 \frac{1}{2|\mu^\pm_e|}\frac{d}{dt} |F^\pm|^2 + \frac{1}{|\mu_e^\pm|} F^\pm\oD ^\pm F^\pm  
 =  \pm\hat v \cdot \vE F^\pm .
\end{aligned} $$
By the specular boundary condition expressed as the skew symmetry property \eqref{skewadj} applied to $g=h=F^\pm/|\mu_e^\pm|$, 
the integral of the second term vanishes.  Therefore 
$$  
   \frac{1}{2} \frac{d}{dt}\int_\Omega \int_{\RR^3}\frac{1}{|\mu^\pm_e|} |F^\pm  |^2   \mathrm{d}v\mathrm{d}x 
   = \pm \int_\Omega \int_{\RR^3}\hat v \cdot \vE \Big ( f^\pm \mp (a+r\cos \theta)\mu_p^\pm  A_\varphi\Big ) \mathrm{d}v\mathrm{d}x .$$
Summing up these two identities, recalling that $\vj = \int_{\RR^3}\hat v (f^+ - f^-)\; \mathrm{d}v$, 
and using the fact that $\mu_p^\pm$ are even functions of $v_r$ and  $v_\theta$, we get 
$$\begin{aligned}
  \frac{1}{2} \frac{d}{dt}\sum_\pm \int_\Omega \int_{\RR^3}\frac{1}{|\mu^\pm_e|} |F^\pm |^2  
  & = \int_\Omega \vE \cdot \vj \; \mathrm{d}x -\sum_\pm\int_\Omega \int_{\RR^3}(a+r\cos \theta)\mu_p^\pm  A_\varphi  \hat v \cdot \vE
   \\& = \int_\Omega \vE \cdot \vj \; \mathrm{d}x -\sum_\pm\int_\Omega \int_{\RR^3}(a+r\cos \theta)\mu_p^\pm  A_\varphi  \hat v_\varphi  E_\varphi   
     \\& = \int_\Omega \vE \cdot \vj \; \mathrm{d}x +\frac 12\frac{d}{dt}\sum_\pm\int_\Omega \int_{\RR^3}(a+r\cos \theta)\mu_p^\pm \hat v_\varphi  | A_\varphi |^2   .\end{aligned}$$
Rearranging this identity, we obtain   
\begin{equation}\label{invariant-Vpm}\begin{aligned}
  \frac{1}{2} \frac{d}{dt}\sum_\pm \int_\Omega \int_{\RR^3} \Big[ \frac {1}{|\mu^\pm_e|} | F^\pm |^2 
- (a+r\cos \theta) \mu_p^\pm \hat v_\varphi | A_\varphi|^2   \Big]  = \int_\Omega \vE \cdot \vj \; \mathrm{d}x.\end{aligned}\end{equation}
In addition, by taking the time derivative of the energy term generated by the electromagnetic fields $\mathcal{I}_M(\vE,\vB ) :=   \int_\Omega  \Big[ |\vE|^2 + |\vB|^2  \Big] \; \mathrm{d}x,$
we readily get 
$$\begin{aligned}
\frac12\frac {d}{dt}\mathcal{I}_M(\vE,\vB )&= \int_\Omega  \Big[ \vE \cdot \D_t \vE + \vB \cdot \D_t \vB \Big] \; \mathrm{d}x
\\
&= \int_\Omega  \Big[ - \vE \cdot \vj + \vE \cdot (\nabla \times \vB) - \vB \cdot (\nabla \times \vE) \Big] \; \mathrm{d}x.
\\
&= -\int_\Omega   \vE \cdot \vj \; \mathrm{d}x + \int_{\D\Omega} (\vE \times \vB) \cdot n(x) \; \mathrm{d}S_x 
= -\int_\Omega   \vE \cdot \vj \; \mathrm{d}x 
\end{aligned}$$
Here we have  used the fact that $(\vE \times \vB) \cdot n = (\vE \times n) \cdot \vB $, which vanishes on the boundary $\partial\Omega$  
due to the perfect conductor condition $\vE\times n =0$. 
Together with \eqref{invariant-Vpm}, we obtain the invariance of the linearized energy functional as in the following lemma. 
\begin{lemma} \label{lem-invariance} 
Suppose that  $(f^\pm, \vE,\vB)$ is a solution of the linearized system  \eqref{lin-VM} with its boundary conditions 
such that 
$F^\pm\in C^1(\RR, L^2_{1/|\mu_e|}(\Omega \times \RR^3))$ and $\vE,\vB \in C^2(\RR; H^2(\Omega))$, where $f^\pm$ and $F^\pm$ are related by \eqref{def-F}. 
Then the linearized energy functional 
$$\begin{aligned}\mathcal{I}(f^\pm,\vE,\vB ) := \sum_\pm \int_\Omega  \int_{\RR^3}\Big[ \frac {1}{|\mu^\pm_e|} | F^\pm |^2 
- (a+r\cos \theta) \mu_p^\pm \hat v_\varphi | A_\varphi|^2   \Big] \; \mathrm{d}v\mathrm{d}x 
+  \int_\Omega \Big[ |\vE|^2 + |\vB|^2 \Big] \; \mathrm{d}x\end{aligned}$$
is independent of time.  
\end{lemma} 
      Furthermore, 
we  obtain the following additional invariants. 
\begin{lemma} \label{lem-invK} 
Under the same hypothesis as in Lemma \ref{lem-invariance}, the functional 
\begin{equation}\label{inv-K}  
\mathcal{K}^\pm_g(f^\pm, \vA )
 = \int_\Omega  \int_{\RR^3}\Big( F^\pm\mp \mu_e^\pm \hat v \cdot \vA \Big) g \; \mathrm{d}v\mathrm{d}x\end{equation} 
 is independent in time, for all $g \in \ker\oD^\pm $ and for both + and $-$.  
 Also, the integrals 
$ \int_\Omega  \int_{\RR^3}f^\pm (t,x,v)\; \mathrm{d}v\mathrm{d}x $ are time-invariant; that is, the total masses are conserved. 
\end{lemma} 
        \begin{proof} 
Such invariants $\mathcal{K}^\pm_g(f^\pm, \vA )$ can easily be discovered by writing the Vlasov equations 
in the form: 
 \begin{equation}\label{lin-VM02} \begin{aligned}
 \dt \Big ( F^\pm\mp \mu_e^\pm \hat v \cdot \vA \Big) +\oD ^\pm \Big ( F^\pm \mp   \mu_e^\pm \phi \Big)  =  0.
 \end{aligned}
\end{equation}
Using the skew-symmetry property  \eqref{skewadj} of $\oD ^\pm$ which incorporates the specular boundary condition, we have 
$$\begin{aligned}
\frac{d}{dt}\mathcal{K}^\pm_g(f^\pm, \vA ) &= \int_\Omega  \int_{\RR^3}\D_t \Big( F^\pm\mp \mu_e^\pm \hat v \cdot \vA \Big) g \; \mathrm{d}v\mathrm{d}x   \\
&= - \int_\Omega  \int_{\RR^3}\oD ^\pm \Big ( F^\pm \mp   \mu_e^\pm \phi \Big) g \; \mathrm{d}v\mathrm{d}x   \\
&= \int_\Omega\int_{\RR^3}\Big ( F^\pm \mp   \mu_e^\pm \phi  \Big) \oD^\pm g\; \mathrm{d}v\mathrm{d}x \\&= 0 
\end{aligned}
$$
due to the specular conditions on $f^\pm,F^\pm,$ and $g$ and the evenness in $v_r$ of $\mu^\pm$.   
In particular, if we take $g=1$ in \eqref{inv-K} and note that 
$\D_{v_\varphi}\mu^\pm = \mu^\pm_e \hat v_\varphi + (a+r\cos \theta) \mu^\pm_p$, 
it becomes clear that the integrals $ \int_\Omega  \int_{\RR^3}f^\pm (t,x,v)\; \mathrm{d}v\mathrm{d}x $ are  time-invariant. 
\end{proof}

\subsection{Growing modes are pure} \label{ss-pure}
In this subsection, we show that if 
$(e^{\lambda t} f^\pm, e^{\lambda t}\vE, e^{\lambda t} \vB)$ with $\R \lambda>0$ is a 
complex growing mode, then $\lambda$ must be real.  
See subsection \ref{ss-growing} for some properties of a growing mode.  
We will roughly follow the splitting method in \cite{friedberg-mhd, LS1}, but our proof is fundamentally simpler.   
As before, 
we denote $F^\pm  =  f^\pm \mp (a+r\cos \theta) \mu_p^\pm  A_\varphi$.  
Let 
$F_{\mathrm{ev}}^\pm $ and $F_{\mathrm{od}}^\pm$ be the even and odd parts  of $F^\pm$ 
with respect to the variable $(v_r,v_\theta)$.  
Thus we have the splitting  $ F^\pm = F_{\mathrm{ev}}^\pm + F_{\mathrm{od}}^\pm$.
By inspection 
of the definition \eqref{def-opD}, the operators $\oD^\pm $ map even functions to odd functions 
and vice versa. We therefore obtain, from the Vlasov equations \eqref{lin-VMF}, the split equations 
\begin{equation}\label{split}   
\left\{\begin{aligned}
\lambda F^\pm_{\mathrm{ev}} + \oD^\pm F^\pm_{\mathrm{od}}   \quad 
&=\quad \mp\mu_e^\pm \hat v_\varphi E_\varphi \\
\lambda F^\pm_{\mathrm{od}} + \oD^\pm F^\pm_{\mathrm{ev}} \quad
&=\quad \mp \mu_e^\pm \hat v \cdot \tilde\vE,
\end{aligned}\right.  
\end{equation} 
where $\tilde \vE: = E_r e_r + E_\theta e_\theta$. 
The split equations imply that 
\begin{equation}\label{wave-fod} 
\begin{aligned}
(\lambda^2  - {\oD^\pm}^2) F^\pm_{\mathrm{od}} &= \mp \lambda \mu_e^\pm \hat v \cdot \tilde \vE \pm \mu_e^\pm \oD^\pm(\hat v_\varphi E_\varphi) .
\end{aligned}  
 \end{equation}
Let ${\overline  F}^\pm$ denote the complex conjugate of $F^\pm$.   
By \eqref{def-opD} and the specular boundary condition on $F^\pm$ in its weak form \eqref{skewadj}, 
it follows that $F_{\mathrm{od}}^\pm$  also satisfies the specular condition. 
Moreover, since $\oD^\pm F_{\mathrm{od}}^\pm$ is even in the pair $(v_r,v_\theta)$, 
$\oD^\pm F_{\mathrm{od}}^\pm$  also satisfies the specular condition.  
Thus when we multiply  equation \eqref{wave-fod} by $\frac{1}{ {|\mu_e^\pm|}} \overline F^\pm_{\mathrm{od}} $   
and integrate the result over $\Omega \times \RR^3$, 
we may apply the skew-symmetry property \eqref{skewadj} of $\oD^\pm$  
to obtain 
\begin{equation}\label{id-pm} 
\begin{aligned}
\lambda^2 &\int_\Omega \int_{\RR^3}\frac{1}{|\mu_e^\pm|} |F^\pm_{\mathrm{od}}|^2 \; \mathrm{d}v\mathrm{d}x 
+  \int_\Omega \int_{\RR^3}\frac{1}{|\mu_e^\pm|} |\oD^\pm F^\pm_{\mathrm{od}}|^2 \; \mathrm{d}v\mathrm{d}x 
\\&= \quad \pm \int_\Omega \int_{\RR^3} \Big( \lambda \hat v \cdot \tilde \vE \overline F^\pm_{\mathrm{od}} +  \hat v_\varphi E_\varphi  \oD^\pm \overline F^\pm_{\mathrm{od}} \Big) \; \mathrm{d}v\mathrm{d}x        
\\&= \quad \pm \int_\Omega \int_{\RR^3} \Big( \lambda \hat v \cdot \tilde \vE \overline F^\pm_{\mathrm{od}} +  \hat v_\varphi E_\varphi  (- \overline \lambda \overline F^\pm_{\mathrm{ev}} \mp \mu^\pm_e \hat v_\varphi \overline E_\varphi) \Big) \; \mathrm{d}v\mathrm{d}x .    
\end{aligned}
\end{equation} 
	Adding  up 
the (+) and (-) identities in \eqref{id-pm} and examining the {\it imaginary} part of the resulting identity,  
	we get 
\begin{equation*}  \begin{aligned}
2 \R &\lambda \I\lambda  \int_\Omega \int_{\RR^3}\Big( \frac{1}{|\mu_e^+|} |F^+_{\mathrm{od}}|^2 
+ \frac{1}{|\mu_e^-|} |F^-_{\mathrm{od}}|^2\Big) \; \mathrm{d}v\mathrm{d}x \\
& =  \I \int_\Omega \int_{\RR^3} \lambda \hat v  ( \overline F^+_{\mathrm{od}}  
- \overline  F^-_{\mathrm{od}}) \cdot \tilde \vE \; \mathrm{d}v\mathrm{d}x  
- \I \int_\Omega\int_{\RR^3}\overline \lambda \hat v_\varphi E_\varphi (\overline F^+_{\mathrm{ev}}  
- \overline  F^-_{\mathrm{ev}} )\; \mathrm{d}v\mathrm{d}x.  
\end{aligned} 
\end{equation*} 
Combining \eqref{split} with $\vj = \int_{\RR^3}\hat v (f^+ - f^-) \mathrm{d}v$ and using the oddness and evenness, 
we can write the right side as 
$$
    \I \int_\Omega \Big[\lambda  (\overline j_r E_r +\overline  j_\theta E_\theta ) - \overline \lambda E_\varphi  \overline j_\varphi \Big] \;\mathrm{d}x  
+ \I \int_\Omega\int_{\RR^3} \overline \lambda \hat v_\varphi (a+r\cos\theta) (\mu_p^+ + \mu_p^-) E_\varphi \overline A_\varphi \; \mathrm{d}v\mathrm{d}x . 
$$
	The imaginary part of the last integral vanishes due to $E_\varphi = - \lambda A_\varphi$ from \eqref{def-EB}.  
Thus the identity simplifies to 
\begin{equation}\label{key-id-fod01}\begin{aligned}
2 \R \lambda \I\lambda  \int_\Omega \int_{\RR^3}\Big( \frac{1}{|\mu_e^+|} |F^+_{\mathrm{od}}|^2 
+ \frac{1}{|\mu_e^-|} |F^-_{\mathrm{od}}|^2\Big) \; \mathrm{d}v\mathrm{d}x &=  \I \int_\Omega  \Big[ \lambda  \overline \vj \cdot \vE  -(\lambda + \overline \lambda) E_\varphi \overline  j_\varphi \Big] \; \mathrm{d}x   .\end{aligned}     
\end{equation}

We now use the Maxwell equations to compute the terms on the right side of \eqref{key-id-fod01}. 
By the first and then the second equation in \eqref{Ampere-law}, we have
$$\begin{aligned}  \int_\Omega \overline \vj \cdot \lambda \vE\;\mathrm{d}x   
&=  \int_\Omega  (-  \overline \lambda \overline \vE  + \nabla \times \overline \vB ) \cdot \lambda \vE\;\mathrm{d}x   
\\
&= \int_\Omega\Big (-  |\lambda \vE |^2  + \lambda (\nabla \times \vE ) \cdot \overline \vB\Big) \;\mathrm{d}x  + \int_{\D\Omega} (\vE \times n(x)) \cdot \overline \vB  \; \mathrm{d}S_x  
\\
&= -  \int_\Omega\Big ( |\lambda \vE |^2  + \lambda^2 |\vB|^2 \Big) \;\mathrm{d}x ,
\end{aligned}$$ in which the boundary term vanishes due to the perfect conductor condition $\vE \times n =0$. 
It remains to calculate the imaginary part of $\int_\Omega E_\varphi \overline  j_\varphi \; \mathrm{d}x $,   
which appears in \eqref{key-id-fod01}.  
	By the second 
Maxwell equation in \eqref{Maxwell-system} together with $E_\varphi = -\lambda A_\varphi$,   we get 
$$
\begin{aligned}
- (\lambda + \overline \lambda)\int_\Omega E_\varphi \overline  j_\varphi \; \mathrm{d}x & = \lambda (\lambda+ \overline \lambda)\int_\Omega A_\varphi \Big( \overline \lambda^2 - \Delta +  \frac{1}{(a+r\cos \theta)^2}\Big) \overline A_\varphi  \; \mathrm{d}x 
 \\
 & = (\lambda^2 + |\lambda|^2) \int_\Omega \Big( \overline \lambda^2|A_\varphi |^2  + |\nabla A_\varphi|^2 +  \frac{|A_\varphi|^2 }{(a+r\cos \theta)^2}\Big)\; \mathrm{d}x ,
 \end{aligned}
 $$
 where we have integrated by parts and used the Dirichlet boundary condition \eqref{Maxwell-BCs} on $A_\varphi$.  
Combining these estimates 
with   \eqref{key-id-fod01} and dropping real terms, we obtain
$$\begin{aligned}
2 \R \lambda \I\lambda & \int_\Omega \int_{\RR^3}\Big( \frac{1}{|\mu_e^+|} |F^+_{\mathrm{od}}|^2 
+ \frac{1}{|\mu_e^-|} |F^-_{\mathrm{od}}|^2\Big) \; \mathrm{d}v\mathrm{d}x  
\\&=  - \I  \lambda^2  \int_\Omega  \Big[|\vB|^2 - |\nabla A_\varphi|^2 
-  \frac{|A_\varphi|^2 }{(a+r\cos \theta)^2} \Big] \; \mathrm{d}x   
+ \I \overline \lambda ^2 |\lambda|^2 \int_\Omega |A_\varphi|^2 \; \mathrm{d}x 
\\&=  - 2 \R \lambda \I\lambda \int_\Omega  \Big[|\vB|^2 - |\nabla A_\varphi|^2 
-  \frac{|A_\varphi|^2 }{(a+r\cos \theta)^2}   +   |E_\varphi|^2      \Big] \; \mathrm{d}x .  
\end{aligned}     
$$
	Now  by definition \eqref{def-EB}  we have 
$$\begin{aligned} \int_\Omega |\vB|^2\; \mathrm{d}x  
&= \int_\Omega \Big[ \frac 1{r^2}|\D_\theta A_\varphi|^2 + \frac{\sin^2 \theta}{(a+r\cos \theta)^2}|A_\varphi|^2 
- \frac{\sin\theta}{(a+r\cos\theta)} \frac 1r \D_\theta |A_\varphi|^2  \\&\quad\qquad + |\D_r A_\varphi|^2 
+ \frac{\cos^2 \theta }{(a+r\cos \theta)^2}|A_\varphi|^2 + \frac{\cos\theta}{(a+r\cos\theta)}\D_r|A_\varphi|^2  
+ |B_\varphi|^2\Big ]\; \mathrm{d}x 
\\
&= \int_\Omega \Big[  |\nabla A_\varphi|^2 + \frac{|A_\varphi|^2 }{(a+r\cos \theta)^2} 
+ \frac{1}{a+r\cos \theta} ( e_r \cos \theta - e_\theta \sin\theta ) \cdot \nabla |A_\varphi|^2  + |B_\varphi|^2\Big ] \; \mathrm{d}x . 
\end{aligned}$$
	The integral of the third term on the right vanishes if we integrate by parts, 
	noting that the resulting divergence $\nabla \cdot  \frac{1}{a+r\cos \theta} ( e_r \cos \theta - e_\theta \sin\theta ) =0$ 
	(see Appendix A), and using the Dirichlet boundary condition on $A_\varphi$. Thus we obtain 
$$\begin{aligned}
2 \R \lambda \I\lambda & \int_\Omega \int_{\RR^3}\Big( \frac{1}{|\mu_e^+|} |F^+_{\mathrm{od}}|^2 
+ \frac{1}{|\mu_e^-|} |F^-_{\mathrm{od}}|^2\Big) \; \mathrm{d}v\mathrm{d}x  
=  - 2 \R \lambda \I\lambda \int_\Omega  \Big[|B_\varphi|^2 +  |E_\varphi|^2  \Big]\; \mathrm{d}x .
\end{aligned}     
$$
The opposite signs of the integrals, together with the assumption that $\R\lambda>0$, imply that $\lambda$ must be real.

\subsection{Minimization} \label{ss-min}
In this subsection we prove an identity that will be fundamental to the proof of stability. 
Throughout this subsection we fix $\vA \in L^2_\tau  (\Omega  )$.   We define the functional 
$$\begin{aligned}\mathcal{J}_ {\vA }(F^+,F^-) & 
:=\sum_\pm \int _\Omega   \int_{\RR^3}\frac {1}{|\mu^\pm_e|} |F^\pm |^2 \; \mathrm{d}v\mathrm{d}x 
+  \int_\Omega   |\nabla\phi|^2 \; \mathrm{d}x,  
\end{aligned}$$
where $\phi = \phi(r,\theta)$ satisfies the Poisson equation 
\begin{equation}\label{f-to-phi}  
-\Delta \phi  = \int_{\RR^3} (F^+ - F^-) \; \mathrm{d}v +  (a+r\cos \theta) A_\varphi \int_{\RR^3} (\mu_p^+ + \mu_p^-)\; \mathrm{d}v, \qquad \phi_{\vert_{x\in \D\Omega}}=0.
\end{equation} 
Let $\mathcal{F}_{\vA }$ be the  linear manifold in $[L^2_{1/|\mu^\pm_e|}(\Omega \times \RR^3)]^2$ consisting of all pairs of toroidally symmetric functions $(F^+,F^-)$ that satisfy the constraints 
\begin{equation}\label{eqs-constraint} 
 \int_\Omega  \int_{\RR^3}\Big( F^\pm\mp \mu_e^\pm \hat v \cdot \vA \Big) g\; \mathrm{d}v\mathrm{d}x =0 ,
 \end{equation}
for all $g \in \ker\oD ^\pm.$ 
Similarly,  let $\mathcal{F}_0$ be the space of pairs $(h^+,h^-)$ in $L^2_{1/|\mu^\pm_e|}(\Omega \times \RR^3)$ that satisfy
\begin{equation}\label{eqs-constraint-0} 
\int_\Omega   \int_{\RR^3} h^\pm g \; \mathrm{d}v\mathrm{d}x = 0, \qquad \forall ~g \in \ker\oD ^\pm.
\end{equation}
Note that the right hand side of the Poisson equation in \eqref{f-to-phi} belongs to $L^2(\Omega)$, 
and so for such a pair of functions $F^\pm$,  standard elliptic theory yields a 
unique solution $\phi\in \mathcal{X}$ of the problem \eqref{f-to-phi}.  
Thus the functional $\mathcal{J}_{\vA }$ is well-defined and nonnegative on $\mathcal{F}_{\vA }$, and 
 its infimum over $\mathcal{F}_{\vA }$ is finite.   
 We will show that it indeed admits a minimizer on $\mathcal{F}_{\vA }$. 

\begin{lemma}\label{lem-estJ} 
For each fixed $ \vA \in L^2_\tau  (\Omega  )$, there exists a pair of functions $F^\pm_*$ that 
minimizes the functional $\mathcal{J}_{\vA }$ 
on $\mathcal{F}_{\vA }$.   
Furthermore, if we let $\phi_*\in \mathcal{X}$ be the associated solution of  the problem \eqref{f-to-phi} with $F^\pm=F^\pm_*$, then  
\begin{equation}\label{id-f0}\begin{aligned}
 F^\pm_*&=\pm \mu^\pm_e (1-\mathcal{P}^\pm)\phi_*  \pm \mu^\pm_e \mathcal{P}^\pm(\hat v\cdot \vA).
\end{aligned}\end{equation}
\end{lemma}
\begin{proof} 
Take a minimizing sequence $F^\pm_n$ in $\mathcal{F}_{\vA }$ 
such that $\mathcal{J}_{\vA }(F^+_n,F_n^-) $ converges to the infimum of 
$\mathcal{J}_{\vA }$.      Since $\{F^\pm_n\}$ are bounded sequences in $L^2_{1/|\mu_e^\pm|}$,  
there are subsequences with weak limits in $L^2_{1/|\mu_e^\pm|}$, which we denote by $F^\pm_*$. 
It is  clear that the limiting functions $F^\pm_*$ also satisfy the constraint \eqref{eqs-constraint}, 
and so they belong to $\mathcal{F}_{\vA }$. That is, $(F^+_*,F_*^-)$ must be a minimizer. 

In order to derive  identity \eqref{id-f0}, 
let the pair $(F^+_*,F^-_*)\in \mathcal{F}_{\vA }$ be a minimizer and 
let  $\phi_*\in H^2(\Omega )$ be the associated solution of  the problem \eqref{f-to-phi} with $F^\pm=F^\pm_*$. 
For each $(F^+,F^-) \in \mathcal{F}_{\vA }$, we denote 
\begin{equation}\label{def-hf} h^\pm  :=  F^\pm\mp \mu_e^\pm \hat v \cdot \vA  .\end{equation}
In particular, $ h^\pm_* :=  F^\pm_*\mp \mu_e^\pm \hat v \cdot \vA $.  
It is clear that $(F^+,F^-) \in \mathcal{F}_{\vA }$ if and only if $(h^+,h^-) \in \mathcal{F}_0$. 
Since $\D_{v_\varphi}[\mu^\pm] = \mu^\pm_e \hat v_\varphi+ (a+r\cos\theta) \mu^\pm_p$ 
and $\mu_e^\pm(\hat v_r A_r + \hat v_\theta A_\theta)$ is odd in $(v_r,v_\theta)$, we have
$$ \int_{\RR^3}(F^+ - F^-)\; \mathrm{d}v +  (a+r\cos \theta) A_\varphi \int_{\RR^3} (\mu_p^+ + \mu_p^-)\; \mathrm{d}v = \int_{\RR^3}(h^+ - h^-)\; \mathrm{d}v. $$
Thus if we let $\phi \in \mathcal{X}$ be the solution of the problem \eqref{f-to-phi}, $\phi$ is independent of the 
change of variables in \eqref{def-hf}.  Consequently, $(F^+_*,F^-_*)$ is a minimizer of $\mathcal{J}_{\vA }(F^+,F^-)$ 
on $\mathcal{F}_{\vA }$ if and only if $(h^+_*,h^-_*)$ is a minimizer of the functional
$$\begin{aligned} 
\mathcal{J}_0(h^+,h^-) =\sum_\pm \int _\Omega   \int_{\RR^3}\frac {1}{|\mu^\pm_e|} | h^\pm  \pm \mu^\pm_e \hat v\cdot \vA |^2 \; \mathrm{d}v\mathrm{d}x +  \int_\Omega   |\nabla \phi|^2 \; \mathrm{d}x 
\end{aligned}$$
on $\mathcal{F}_0$.     By minimization, the first variation of $\mathcal{J}_0$  is  
\begin{equation}\label{eqs-Lag}  
\sum_\pm\int _\Omega   \int_{\RR^3}\frac {1}{|\mu^\pm_e|} (h^\pm_* \pm \mu^\pm_e \hat v \cdot \vA ) h^\pm \; \mathrm{d}v\mathrm{d}x 
+ \int_\Omega   \nabla \phi_* \cdot \nabla \phi \; \mathrm{d}x =0
\end{equation}
for all $(h^+,h^-) \in \mathcal{F}_0$ where $\phi$ solves \eqref{f-to-phi}. 
By the Dirichlet boundary condition on $\phi$, we have
 $$
 \int_\Omega   \nabla \phi_* \cdot \nabla \phi \; \mathrm{d}x = - \int_\Omega   \phi_* \Delta \phi \; \mathrm{d}x = \int_\Omega   \int_{\RR^3}\phi_* (h^+-h^-) \; \mathrm{d}v\mathrm{d}x.$$
		Adding this to  identity \eqref{eqs-Lag}, we obtain
\begin{equation}\label{zero-id} 
\sum_\pm \int _\Omega   \int_{\RR^3}\frac {1}{|\mu^\pm_e|} (h^\pm _* \pm  \mu^\pm_e \hat v\cdot \vA  \mp \mu^\pm_e \phi_*)  h^\pm \; \mathrm{d}v\mathrm{d}x  =0  
\end{equation}
for all $(h^+,h^-) \in \mathcal{F}_0$. 
In particular, we can take $h^- =0$ in \eqref{zero-id} and obtain the identity 
for all $h^+\in L^2_{1/|\mu_e^+|}$ satisfying $\int_\Omega   \int_{\RR^3} h^+ g^+ \; \mathrm{d}v\mathrm{d}x = 0,$ for all $g^+ \in \ker\oD ^+.$ 

We claim that this identity implies $h^+_* + \mu^+_e \hat v\cdot \vA  - \mu^+_e \phi_* \in \ker\oD ^+$. 
Indeed, let 
$$ k_* = |\mu_e^+|^{-1} (h^+_* + \mu^+_e \hat v \cdot \vA  - \mu^+_e \phi_*), \quad \ell = |\mu_e^+|^{-1} h^+ .$$
Using the inner product in $\cH = L^2_{|\mu_e^+|}$, we  have 
$$ \langle k_*, \ell \rangle_\cH = 0 \quad \forall \ell\in (\ker \oD^+)^\perp .  $$ 
Because $\oD ^+$ (with the specular condition)  is a skew-adjoint operator on $\cH$
(see \eqref{domainD}, \eqref {skewadj}), we have 
$k_* \in (\ker \oD^+)^{\perp\perp} = \ker \oD^+$. 
Thus 
$$ \oD^+ \{ F_*^+  - \mu_e^+\phi_* \} 
=  \oD^+ \{ h_*^+ + \mu_e^+\hat v \cdot \vA - \mu_e^+\phi_* \} =- \mu_e^+ \oD^+ k_*  =  0.  $$
This proves the claim.  
Similarly $\oD^- \{ F_*^- + \mu_e^-\phi_* \} = 0$.  
Equivalently,  
$$\begin{aligned}
 F^\pm_* \mp \mu^\pm_e \phi_*  &= \mathcal{P}^\pm \Big( F^\pm_* \mp \mu^\pm_e \phi_*  \Big).
 \end{aligned}$$
On the other hand,  
the constraint \eqref{eqs-constraint} can be written as 
$ \mathcal{P}^\pm(F^\pm_* \mp \mu^\pm_e \hat v\cdot \vA ) =0$.   
Combining these identities, we obtain the identity \eqref{id-f0} at once.  
\end{proof}

The following lemma shows a remarkable connection between the minimum of the energy $\mathcal{J}_\vA$ and the operators defined in \eqref{operators-0}.  

\begin{lemma}\label{lem-estJ01} 
For each fixed $\vA\in L^2_\tau(\Omega)$, let $F^\pm_*$ be the minimizer of $\mathcal{J}_\vA$ obtained from Lemma \ref{lem-estJ}. Then, \begin{equation}\label{J-calculation} 
\begin{aligned}&\mathcal{J}_{\vA } (F^+_*,F^-_*)  
= - ( \mathcal{B}^0 (\mathcal{A}_1^0)^{-1} (\mathcal{B}^0)^*   A_\varphi ,  A_\varphi  )_{L^2}  
+\sum_\pm \|\mathcal{P}^\pm(\hat v \cdot \vA)\|^2_{\cH}.\end{aligned}   
\end{equation}
\end{lemma}

\begin{proof}
Plugging $F^\pm_*$ of the form \eqref{id-f0} into $\mathcal{J}_\vA$ and using the orthogonality of $\mathcal{P}^\pm$ and $(1-\mathcal{P}^\pm)$ in $\cH$, we have 
 $$\begin{aligned} 
\mathcal{J}_ {\vA }(F^+_*,F^-_*) = \| ( 1-\mathcal{P}^\pm)\phi_*\|^2_{\cH} 
 +  \int_\Omega   |\nabla\phi_*|^2 \; \mathrm{d}x + \|\mathcal{P}^\pm(\hat v \cdot \vA)\|^2_{\cH}. \end{aligned}$$
By definition \eqref{operators-0} of $\mathcal{A}_1^0$, the first two terms are equal to $-\langle\mathcal{A}_1^0\phi_* , \phi_* \rangle_{L^2} $, upon using the facts that 
$ \int_\Omega   |\nabla \phi_*|^2 \; \mathrm{d}x = -  \int_\Omega   \phi_* \Delta \phi_* \; \mathrm{d}x$ and $ \| ( 1-\mathcal{P}^\pm)\phi_*\|^2_{\cH} = \langle ( 1-\mathcal{P}^\pm)\phi_*,\phi_*\rangle_{\cH}$. 

It remains to show that $\phi_* = -(\mathcal{A}_1^0)^{-1}(\mathcal{B}^0)^*  A_\varphi$. To do so, we plug $F^\pm_*$ of the form \eqref{id-f0} into the Poisson equation \eqref{f-to-phi} to get 
 \begin{equation*}\begin{aligned}
-\Delta \phi_* & = \sum_\pm \Big[\int_{\RR^3}\mu^\pm_e ( 1-\mathcal{P}^\pm)\phi_* \; \mathrm{d}v  +   \int_{\RR^3}(a+r\cos\theta) \mu^\pm_p\; \mathrm{d}v    A_\varphi  +\int_{\RR^3}\mu^\pm_e \mathcal{P}^\pm(\hat v\cdot \vA) \; \mathrm{d}v \Big ].  
\end{aligned}\end{equation*}
By definition 
this is equivalent to the equation $-\mathcal{A}_1^0\phi_* = (\mathcal{B}^0)^*  A_\varphi $, 
where  we have used the oddness  in $(v_r,v_\theta)$ of $\mathcal{P}^\pm(\hat v_r A_r + \hat v_\theta A_\theta)$ 
so that its integral vanishes. The operator $\mathcal{A}_1^0$ is invertible by Lemma \ref{lem-AB0property}, and thus $\phi_* = -(\mathcal{A}_1^0)^{-1}(\mathcal{B}^0)^*  A_\varphi $. 
\end{proof}

\subsection{Proof of stability} \label{ss-stab}
With the above preparations, we are ready to prove the following stability result, which is Part {\it (i)} of Theorem \ref{theo-main}.   
 \begin{lemma}\label{lem-stab} 
 If $\mathcal{L}^0\ge 0$, then there exists no growing mode $(e^{\lambda t} f^\pm, e^{\lambda t}{\vE}, e^{\lambda t} {\vB})$  
 with $\R\lambda >0$. 
  \end{lemma}
\begin{proof} 
Assume the contrary.  
For the basic properties of any growing mode, see Lemma \ref{growprops}. 
By the result in the previous subsection, it is a purely growing mode ($\lambda>0$)  and so we can assume that 
the functions $(  f^\pm, \vE,  \vB)$ are real-valued.   By the time-invariance in Lemma \ref{lem-invariance} 
and the exponential factor $\exp{\lambda t}$,   
the functional $\mathcal{I}(f^\pm,\vE,  \vB)$ must be identically equal to zero. 
Let $\phi$ and $ \vA$ be defined as usual through the relations \eqref{def-potentials}. 
		We thus obtain
$$
\begin{aligned}
0={\mathcal{I}}( f^\pm,\vE,\vB) &
=  \mathcal{J}_{ \vA }( F^+, F^-) - \sum_\pm \int _\Omega   \int_{\RR^3}(a+r\cos\theta)\mu^\pm_p \hat v_\varphi|   A_\varphi|^2\; \mathrm{d}v\mathrm{d}x 
+ \int_\Omega   \Big[ \lambda^2 |\vA|^2 + |\vB|^2 \Big] \; \mathrm{d}x  , 
\end{aligned}
$$ 
where the term $2\lambda\int_{\RR^3}\vA\cdot\nabla\phi\, \mathrm{d}x$ vanishes due to the Coulomb gauge and the boundary condition on $\phi$.  
Furthermore, the integrals $\mathcal{K}^\pm_g( f^\pm,  \vA)$ defined in \eqref{inv-K} must  be zero. 
The vanishing of these latter integrals  is equivalent to the constraint \eqref{eqs-constraint} and therefore  the pair $( F^+, F^-)$ belongs to the linear manifold $\mathcal{F}_{\vA}$. 
We can then apply  Lemma \ref{lem-estJ01} to assert that   
$\mathcal{J}_ {\vA} (F^+,F^-) \ge \mathcal{J}_ {\vA} (F_*^+,F_*^-)$.  
	Therefore  
\begin{equation}\label{lower-boundI}\begin{aligned}
{\mathcal{I}}( f^\pm,\vE,\vB) &  
\ge - ( \mathcal{B}^0(\mathcal{A}_1^0)^{-1} (\mathcal{B}^0)^*     A_\varphi,   A_\varphi )_{L^2} +\sum_\pm \|\mathcal{P}^\pm(\hat v \cdot \vA)\|^2_{\cH}
\\&\qquad  - \sum_\pm \int _\Omega   \int_{\RR^3}(a+r\cos\theta)\mu^\pm_p \hat v_\varphi|   A_\varphi|^2\; \mathrm{d}v\mathrm{d}x 
+ \int_\Omega   \Big[ \lambda^2 |\vA|^2 + |\vB|^2 \Big] \; \mathrm{d}x.
\end{aligned}                                     \end{equation}
	But by the last calculations in Subsection \ref{ss-pure}, we have 
$$\begin{aligned} \int_\Omega |\vB|^2\; \mathrm{d}x  
&= \int_\Omega \Big[  |\nabla A_\varphi|^2 + \frac{|A_\varphi|^2 }{(a+r\cos \theta)^2} + |B_\varphi|^2\Big ] \; \mathrm{d}x .
\end{aligned}$$
In addition, from the definition \eqref{operators-0} of $\mathcal A_2^0$, 
an integration by parts together with the Dirichlet boundary condition 
on $A_\varphi$ yields
\begin{equation}\label{product-A2}\begin{aligned} 
( \mathcal{A}_2^0   A_\varphi,   A_\varphi)_{L^2} & = \int_\Omega   \Big( |\nabla  A_\varphi|^2 
+ \frac 1 {(a+r\cos\theta)^2} |   A_\varphi|^2 \Big)\; \mathrm{d}x +\sum_\pm \|\mathcal{P}^\pm(\hat v_\varphi A_\varphi)\|^2_{\cH}
\\&\qquad  - \sum_\pm \int _\Omega   \int_{\RR^3}(a+r\cos\theta)\mu^\pm_p \hat v_\varphi|   A_\varphi|^2\; \mathrm{d}v\mathrm{d}x 
.\end{aligned}\end{equation}
		Furthermore, 
$$ \|\mathcal{P}^\pm(\hat v\cdot \vA)\|^2_{\cH}   =   \|\mathcal{P}^\pm(\hat v_\varphi A_\varphi)\|^2_{\cH}   
+    \|\mathcal{P}^\pm(\hat v_rA_r + \hat v_\theta A_\theta)\|^2_{\cH}  
+   \langle  \mathcal{P}^\pm(\hat v_\varphi A_\varphi)   ,  \mathcal{P}^\pm(\hat v_rA_r + \hat v_\theta A_\theta) \rangle_{\cH}, 
$$
in which the last term vanishes due to oddness in $(v_r,v_\theta)$.  
		Putting these various identities 
into \eqref{lower-boundI} and using the definition \eqref{operator-L0} of $\mathcal{L}^0$, we then get 
\begin{equation}\label{key-ineq}\begin{aligned}
0 = {\mathcal{I}}( f^\pm, \vE,\vB) &\ge  ( \mathcal{L}^0    A_\varphi,   A_\varphi )_{L^2} +   \int_\Omega   \Big[ \lambda^2 |\vA|^2 + |B_\varphi|^2 \Big] \; \mathrm{d}x+\sum_\pm \|\mathcal{P}^\pm(\hat v_rA_r + \hat v_\theta A_\theta)\|^2_{\cH}. 
\end{aligned}\end{equation}
Since we are assuming $\lambda>0$ and $\mathcal{L}^0\ge 0$, we deduce $\vA=0$.  
From the definition of $\mathcal I(f^\pm, \vE, \vB)$, we deduce that $f^\pm=0,\ \vE=0$.  
 Thus the linearized system has no growing mode.  
\end{proof}

\section{Linear instability} \label{sec-instability}
We now turn to the instability part of Theorem \ref{theo-main}. 
It is based on a spectral analysis of the relevant operators.   We plug the simple form 
$(e^{\lambda t} f^\pm, e^{\lambda t}\vE, e^{\lambda t} \vB)$, with some real $\lambda>0$, 
into the linearized RVM system \eqref{lin-VM} to obtain the Vlasov equations
\begin{equation}\label{Lap-VMsys}
(\lambda + \oD ^\pm ) \Big ( f^\pm \mp  (a+r\cos \theta)  \mu_p^\pm   A_\varphi \mp \mu_e^\pm \phi \Big)  = \pm \lambda \mu^\pm_e (\hat v \cdot \vA - \phi)
 \end{equation} and the Maxwell equations
 \begin{equation}\label{Lap-Maxwell}
\begin{aligned}
- \Delta \phi &= \rho 
\\
\Big(\lambda^2 - \Delta +  \frac{1}{(a+r\cos \theta)^2}\Big)  A_\varphi  &= j_\varphi
\\
( \lambda^2-\Delta ) \tilde \vA + \lambda \nabla \phi &= \tilde \vj,
\end{aligned}\end{equation}
in which the Coulomb gauge $\nabla \cdot \tilde \vA =0$ is imposed. As before, we impose the specular boundary condition on $f^\pm$, the zero Dirichlet boundary condition on $\phi, A_\varphi, A_\theta$, 
and the Robin-type boundary condition on $A_r$, namely $ (a+\cos \theta) \D_r A_r + (a+2\cos \theta) A_r =0$. The quadruple $(f^\pm,\phi, \vA)$ should be regarded as a perturbation of the equilibrium.  

\subsection{Particle trajectories}\label{ss-traj}
We begin  with the $+$ case (ions).  
 For each $(x,v) \in \Omega \times \RR^3$, we introduce the particle trajectories $(X^+(s;x,v), V^+(s;x,v))$ defined by the equilibrium as 
\begin{equation}\label{trajectory} 
\dot X^+ = \hat V^+, \qquad
\dot V^+ = \vE^0(X^+) + \hat V^+\times \vB^0(X^+),
 \end{equation}
with initial values $ (X^+(0;x,v), V^+(0;x,v)) = (x,v).$ Here $\dot{} = \frac{d}{ds}$. 
Because of the $C^1$ regularity of $\vE^0$ and $\vB^0$ in $\overline \Omega$, each trajectory can be continued for 
at least a certain fixed time.  
So each particle trajectory exists and preserves the toroidal symmetry up to the first point where it meets the boundary. Let $s_0$ be a time at which the trajectory $X^+(s_0-;x,v)$ belongs to  $\D \Omega $. 
In general, we write $h(s\pm)$ to mean the limit from the right (left).
According to the specular boundary condition, 
the trajectory $(X^+(s;x,v), V^+(s;x,v))$ can be continued by the rule 
\begin{equation}\label{traj-reflection1}
(X^+(s_0+;x,v), V^+(s_0+;x,v)) = (X^+(s_0-;x,v), \overline  V^+(s_0-;x,v)),
\end{equation}  
with the notation $\overline  V = (-V_r,V_\theta,V_\varphi)$.  Thus $X^+$ is continuous and $V^+$ has a jump at $s_0$.  
Whenever the trajectory meets the boundary, it is reflected in the same way and then continued via the ODE \eqref{trajectory}. Such a continuation is guaranteed for some short time past $s_0$ by regularity of  $\vE^0$ and $\vB^0$ in $\overline{\Omega}$ and standard ODE theory. By Appendix \ref{app-particle}, for almost every particle in $\Omega \times \RR^3$, the trajectory is well-defined and hits the boundary at most a finite number of times in each finite time interval. When there is no possible confusion, 
we will simply write $(X^+(s),V^+(s))$ for the particle trajectories.  
The trajectories $(X^-(s),V^-(s))$ for the $-$ case (electrons) are defined similarly. 

\begin{lemma}\label{lem-trajectory} For almost every $(x,v)\in \Omega \times \RR^3$, the particle trajectories 
$(X^\pm(s;x,v),V^\pm(s;x,v))$ are piecewise $C^1$ smooth in $s\in \RR$, and for each $s\in \RR$, 
the map $(x,v) \mapsto (X^\pm(s;x,v),V^\pm(s;x,v))$ is one-to-one and differentiable with  
Jacobian equal to one at all points $(x,v)$ such that $X^\pm(s;x,v) \not \in \D\Omega$.   
In addition, the standard change-of-variable formula 
\begin{equation}\label{change-variable}
\int_\Omega\int_{\RR^3} g(x,v) \; \mathrm{d}x\mathrm{d}v = \int_\Omega\int_{\RR^3} g(X^\pm(-s;y,w),V^\pm(-s;y,w))\; \mathrm{d}w\mathrm{d}y
\end{equation} 
is valid for each $s\in \RR$ and for each measurable function $g$ for which the integrals are finite. 
\end{lemma}
\begin{proof} Let $(x,v)$ be an arbitrary point in $\Omega \times \RR^3$ so that the particle trajectory $(X^\pm(s;x,v),V^\pm(s;x,v))$ hits the boundary at most a finite number of times in each finite time interval. 
Except when it hits the boundary, the trajectory is smooth in time. So the first assertion is clear.   
Given $s$,  let $\mathcal{S}$ be the set $(x,v)$ such that $X^\pm(s;x,v)\not \in \D\Omega$. Clearly, $\mathcal{S}$ is open and its complement in $\Omega \times \RR^3$ has Lebesgue measure zero. For each $s$, the trajectory map is one-to-one on $\mathcal{S}$ since the ODE \eqref{trajectory} and \eqref{traj-reflection1} is time-reversible and well-defined.   
In addition, a direct calculation shows that the Jacobian determinant is time-independent and is therefore equal to one.   The change-of-variable formula \eqref{change-variable} holds on the open set $\mathcal{S}$ and so on $\Omega\times \RR^3$, as claimed. 
\end{proof}

\begin{lemma}\label{lem-reflection} Let $g(x,v)$ be a $C^1$ radial function on $\overline\Omega\times\RR^3$.  
If $g$ is specular on $\D\Omega$, then for all $s$, $g(X^\pm(s;x,v),V^\pm(s;x,v))$ is continuous and 
also specular on $\D\Omega$.  That is, 
$$
g(X^\pm(s;x,v),V^\pm(s;x,v)) = g(X^\pm(s;x,\overline  v),V^\pm(s;x,\overline  v)),$$ for almost every $(x,v)\in \D\Omega\times \RR^3$, 
where $\overline  v = (-v_r,v_\theta,v_\varphi)$ for all $v = (v_r,v_\theta,v_\varphi)$. 
\end{lemma}
\begin{proof}
It follows directly by definition \eqref{trajectory} and \eqref{traj-reflection1} that for almost every $x\in \partial \Omega$, 
the trajectory is unaffected by whether we start with $v$ or $\overline  v$.   So for all $s$ we have 
\begin{equation}\label{traj-reflection2}  
X^\pm(s;x,v) = X^\pm(s;x,\overline  v), \qquad |V^\pm_r(s;x,v)| = |V^\pm_r(s;x,\overline  v)|, \qquad V^\pm_j(s;x,v) 
= V^\pm_j(s;x,\overline  v), \quad j = \theta, \varphi.
 \end{equation}
In fact we have $V^\pm_r(s;x,v) = V^\pm_r(s;x,\overline  v)$ for any $s$ at which $X^\pm(s;x,v)\not \in \D\Omega$, 
while $V^\pm_r(s+;x,v) = - V^\pm_r(s-;x,\overline  v)$ for $s$ at which $X^\pm(s;x,v)\in \D\Omega$.  
Because $g$ is specular on the boundary, it takes the same value at $v_r$ and $-v_r$.  
Therefore $g(X^\pm(s),V^\pm(s))$ is a continuous function of $s$ at the points of reflection.   
It is specular because of  the rule \eqref{traj-reflection1}.
\end{proof}

\subsection{Representation of the particle densities}\label{ss-rep}
 We now invert the operator $(\lambda +\oD ^\pm)$ in \eqref{Lap-VMsys} to obtain an integral representation of $f^\pm$. To do so, we multiply this equation by $e^{\lambda s}$ and then integrate  
along the particle trajectories  $(X^\pm(s;x,v),V^\pm(s;x,v))$ from $s=-\infty$ to zero. We readily obtain  
\begin{equation}\label{def-f}\begin{aligned} 
f^\pm(x,v)  &=  \pm  (a+r\cos \theta)  \mu_p^\pm   A_\varphi \pm \mu_e^\pm \phi  \pm \mu^\pm_e \mathcal{Q}^\pm_\lambda(\hat v\cdot \vA - \phi),\end{aligned}\end{equation}
where we formally denote 
\begin{equation}\label{def-Q}
\mathcal{Q}^\pm_\lambda (g)(x,v)     
:
=  \int_{-\infty}^0 \lambda e^{\lambda s} g(X^\pm(s;x,v),V^\pm(s;x,v)) \; \mathrm{d}s \end{equation} for each measurable function $g = g(x,v)$. As the particle trajectories $(X^\pm(s;x,v),V^\pm(s;x,v))$ are well-defined for almost every $(x,v)$, the function $\mathcal{Q}^\pm_\lambda (g)(x,v)$ is defined for almost every $(x,v)$. Then by Lemma \ref{lem-reflection}, $\mathcal{Q}^\pm_\lambda(g)$ is specular on $\D \Omega$ if $g$ is. Formally, $\mathcal{Q}^\pm_\lambda$ solves the equation $(\lambda + \oD^\pm) \mathcal{Q}^\pm_\lambda = \lambda I$ with $\mathcal{Q}^\pm_0 = \mathcal{P}^\pm$ and $\lim_{\lambda \to \infty}\mathcal{Q}^\pm_\lambda  = I$ (see Lemma \ref{lem-PropQ}).

%

\subsection{Operators}\label{sec-defOp}
It will be convenient to employ a special space to accommodate the 2-vector function $\tilde\vA$.  
We define the space  
\begin{equation} \label{Y-space}
\tilde{\mathcal{Y}} =  \left\{ \vh \in H_\tau^2(\Omega;\tilde\RR^2)\ \Big | \  
\nabla\cdot \vh=0 \text{ in }\Omega, 
\text{ and } 0 = e_\theta\cdot\vh  =  \nabla_x\cdot((e_r\cdot\vh)e_r) \text{ on } \partial\Omega\right\}. 
\end{equation}
Here the tilde indicates that there is no $e_\varphi$ component.  
The subscript $\tau$ indicates that $\vh$  is  toroidally symmetric; that is, 
$h_r=e_r \cdot\vh$ and $h_\theta = e_\theta\cdot\vh$ are independent of the angle $\varphi$.  
The boundary conditions are exactly those which must be satisfied by $\tilde\vA$ in \eqref{Maxwell-BCs}.  

When we  substitute \eqref{def-f} into the Maxwell equations \eqref{Lap-Maxwell}, several operators will naturally arise.     
We first introduce them formally.  
The following operators  map scalar functions to scalar functions.  
$$\begin{aligned}  \mathcal{A}_{1}^\lambda h :
& =  \Delta h  +  \sum_\pm \int_{\RR^3}\mu^\pm_e (1 - \mathcal{Q}^\pm_\lambda)h \; \mathrm{d}v
\\ \mathcal{A}_{2}^\lambda h  :
& = \Big(\lambda^2 - \Delta +  \frac{1}{(a+r\cos \theta)^2}\Big)   h  
-  \sum_\pm \Big[  \int_{\RR^3} (a+r\cos \theta)  \hat v_\varphi \mu_p^\pm \; \mathrm{d}v \;h  
+ \int_{\RR^3}\hat v_\varphi \mu^\pm_e \mathcal{Q}^\pm_\lambda(\hat v_\varphi h)\; \mathrm{d}v \Big]
\\\mathcal{B}^\lambda h  : 
&=   -\sum_\pm  \int_{\RR^3}\hat v_\varphi\mu^\pm_e (1 - \mathcal{Q}^\pm_\lambda)h  \; \mathrm{d}v,
\\(\mathcal{B}^\lambda)^* h  : 
&=\sum_\pm \Big[  \int_{\RR^3} (a+r\cos \theta)  \mu_p^\pm  h \; \mathrm{d}v + \int_{\RR^3}\mu^\pm_e \mathcal{Q}^\pm_\lambda(\hat v_\varphi h)\; \mathrm{d}v \Big]. 
\end{aligned}
$$
We also introduce an operator that maps vector functions $\vh = h_r e_r + h_\theta e_\theta$ to vector functions by 
$$\begin{aligned}
 \tilde{\mathcal{S}^\lambda} \vh :& =  (-\lambda^2 + \Delta ) \vh  + \sum_\pm \int_{\RR^3}\frac{\tilde v}{\langle v \rangle}\mu^\pm_e \mathcal{Q}^\pm_\lambda(\hat v\cdot \vh) \; \mathrm{d}v, 
 \end{aligned}$$ 
where  $\tilde v = v_r e_r + v_\theta e_\theta$.   
Furthermore, we introduce two operators that map scalar functions to vector functions by 
$$\begin{aligned}
\tilde{\mathcal{T}}^\lambda_1 h :& = -\lambda \nabla h -  \sum_\pm \int_{\RR^3}\frac{\tilde v}{\langle v\rangle} \mu^\pm_e \mathcal{Q}^\pm_\lambda h \; \mathrm{d}v,\qquad 
\tilde{\mathcal{T}}^\lambda_2 h:
=\sum_\pm \int_{\RR^3}\frac{\tilde v}{\langle v\rangle}  \mu^\pm_e \mathcal{Q}^\pm_\lambda(\hat v_\varphi h)\; \mathrm{d}v.  
\end{aligned}$$
Their formal adjoints map vector functions to scalar functions and are given by  
$$\begin{aligned}
(\tilde{\mathcal{T}}^\lambda_1)^* \vh :& = \lambda \nabla \cdot \vh +  \sum_\pm \int_{\RR^3} \mu^\pm_e  \mathcal{Q}^\pm_\lambda (\hat v \cdot \vh) \; \mathrm{d}v,
\qquad (\tilde{\mathcal{T}}^\lambda_2)^* \vh:= - \sum_\pm \int_{\RR^3}\hat v_\varphi  \mu^\pm_e  \mathcal{Q}^\pm_\lambda (\hat v \cdot \vh) \; \mathrm{d}v.
\end{aligned}$$
We shall check below that, when properly defined on certain spaces, they are indeed adjoints. 
Since $\mathcal{Q}^\pm_\lambda(\cdot)(x,v)$ is defined for almost every $(x,v)$, each the above operators 
is defined in a weak sense, that is, by integration against smooth test functions of $x$. 
Hence sets of measure zero can be neglected. 

Moreover, we {\it formally} define each of the corresponding operators at $\lambda =0$ by replacing $\mathcal{Q}^\pm_\lambda$ with the projection $\mathcal{P}^\pm$ of $\cH^\pm$ on the kernel of $\oD^\pm$. 
In Lemma \ref{lem-PropQ} we will justify this notation by letting $\lambda\to0$.  

\begin{lemma} The Maxwell equations \eqref{Lap-Maxwell} are equivalent to the system of equations  
\begin{equation}\label{matrix-Maxwell} \begin{pmatrix}  \mathcal{A}_{1}^\lambda & (\mathcal{B}^\lambda)^* & (\tilde{\mathcal{T}}_{1}^\lambda)^*\\
 \mathcal{B}^\lambda & \mathcal{A}_{2}^\lambda &  (\tilde{\mathcal{T}}_{2}^\lambda)^* \\
\tilde{\mathcal{T}}_{1}^\lambda & \tilde{\mathcal{T}}_{2}^\lambda &    \tilde{\mathcal{S}}^\lambda
\end{pmatrix} \begin{pmatrix}\phi \\A_\varphi \\ \tilde \vA\end{pmatrix} =0 .\end{equation} 
\end{lemma}

\begin{proof} We recall that $\phi$ and $A_\varphi$ are scalars while $\tilde\vA$ is a 2-vector.  
By use of the integral formula \eqref{def-f}, the first equation in \eqref{Lap-Maxwell} becomes
$$\begin{aligned} - \Delta \phi  &= \int_{\RR^3}(f^+ - f^-) (x,v) \;\mathrm{d}v  \\& = \sum_\pm \int_{\RR^3}\Big[  \mu^\pm_e (1 - \mathcal{Q}^\pm_\lambda)\phi + (a+r\cos \theta) \mu_p^\pm   A_\varphi   + \mu^\pm_e \mathcal{Q}^\pm_\lambda(\hat v\cdot \vA) \Big] \; \mathrm{d}v  + \lambda \nabla \cdot \tilde\vA ,\end{aligned}$$
which immediately yields the first identity in \eqref{matrix-Maxwell}. Here, we have added the term $\lambda \nabla \cdot \tilde\vA = \lambda \nabla \cdot \vA$ on the right hand side, which of course vanishes due to the Coulomb gauge. The reason for this addition is to guarantee the self-adjointness of the matrix operators, as will be checked shortly. The second equation in \eqref{Lap-Maxwell} now reads
$$\begin{aligned} \Big(\lambda^2 - \Delta +  \frac{1}{(a+r\cos \theta)^2}\Big)  A_\varphi  
& =   \sum_\pm \int_{\RR^3}\hat v_\varphi\Big[  \mu^\pm_e (1 - \mathcal{Q}^\pm_\lambda)\phi + (a+r\cos \theta) \mu_p^\pm   A_\varphi   + \mu^\pm_e \mathcal{Q}^\pm_\lambda(\hat v\cdot \vA) \Big] \; \mathrm{d}v,
\end{aligned}$$ which is again the second identity in \eqref{matrix-Maxwell}. Similarly, the last vector equation in \eqref{Lap-Maxwell} becomes
$$\begin{aligned}
 (\lambda^2 - \Delta ) \tilde \vA  +  \lambda \nabla \phi & =   \sum_\pm \int_{\RR^3}\frac{\tilde v}{\langle v\rangle} \Big[  \mu^\pm_e \phi  - \mu^\pm_e \mathcal{Q}^\pm_\lambda\phi + (a+r\cos \theta) \mu_p^\pm   A_\varphi   + \mu^\pm_e \mathcal{Q}^\pm_\lambda(\hat v\cdot \vA) \Big] \; \mathrm{d}v,
\end{aligned}$$
 which gives the last identity in \eqref{matrix-Maxwell},  upon noting that the first and third integrals vanish due to evenness of $\mu^\pm_p$ in $v_r$ and $v_\theta$. This proves the lemma. \end{proof}

\begin{lemma}\label{op-bounds} 
Let  $0 <  \lambda < \infty$.  
  
 (i) $\mathcal{Q}^\pm_\lambda$ is bounded from $\mathcal{H}$ to itself with operator norm equal to one.

(ii) $\A_1^\lb-\Delta,\ \A_2^\lb-\lambda^2+\Delta$ and $\B^\lb$ are bounded from $L^2_\tau(\Omega)$ into  $L^2_\tau(\Omega)$.  

(iii) ${\tilde\S}^\lb+\lambda^2-\Delta$ is bounded  from $L^2_\tau(\Omega;\tilde\RR^2)$ into $L^2_\tau(\Omega;\tilde\RR^2)$.  

(iv) $\tilde\T_1^\lb +\lb \nabla$ and $\tilde\T_2^\lb$ are bounded from $L^2_\tau(\Omega)$ into $L^2_\tau(\Omega;\tilde\RR^2)$.
\\
In each case  the operator norm is independent of $\lb$.  
\end{lemma}
\begin{proof} For all $h,g \in \cH$, we have 
$$
\begin{aligned}
\Big| &\int_\Omega  \int_{\RR^3}\mu_e^\pm g \mathcal{Q}^\pm_\lambda h \; \mathrm{d}v\mathrm{d}x\Big| 
 = \Big| \int_{-\infty}^0\int_\Omega  \int_{\RR^3}\lambda e^{\lambda s}\mu_e^\pm g(x,v) h(X^\pm(s),V^\pm(s)) \; \mathrm{d}v\mathrm{d}x\mathrm{d}s\Big|      \\&
\le \Big(\int_{-\infty}^0\lambda e^{\lambda s} \int_\Omega  \int_{\RR^3}|\mu_e^\pm||g(x,v)|^2\; \mathrm{d}v\mathrm{d}x\mathrm{d}s\Big)^{1/2}   
\Big(\int_{-\infty}^0\lambda e^{\lambda s}\int_\Omega  \int_{\RR^3}|\mu_e^\pm| |h(X^\pm(s),V^\pm(s))|^2 \; \mathrm{d}v\mathrm{d}x\mathrm{d}s\Big)^{1/2}     \\&
\le \|g \|_{\cH} \|h\|_{\cH}  ,
 \end{aligned}
$$
in which in the last step we  made the change of variables $(x,v) = (X^+(s;x,v),V^+(s;x,v))$ in the integral for $h$. Also, $\mathcal{Q}^\pm_\lambda (1) = 1$. This proves {\em (i)}. 

Next, by definition, we have for $h \in L^2_\tau(\Omega)$ 
$$
|\langle (\mathcal{A}^\lambda_1 - \Delta)h,h \rangle_{L^2} | = \Big| \sum_\pm \int_\Omega\int_{\RR^3}\mu^\pm_e (1 - \mathcal{Q}^\pm_\lambda)h h \; \mathrm{d}v\mathrm{d}x\Big|  \le 2 \sum_\pm \|h \|_{\cH^\pm}^2 \le 2\sup_{x} \Big (\sum_\pm \int_{\RR^3}|\mu^\pm_e| \; \mathrm{d}v\Big) \| h \|_{L^2}^2.$$
Thus, together with the decay assumption on $\mu_e^\pm$, $\mathcal{A}^\lambda_1 - \Delta$ is a bounded operator. The boundedness of $\mathcal{A}^\lambda_2+\Delta$ and $\mathcal{B}^\lambda$ follows similarly, yielding {\em (ii)}. Also, from the definition, we can write the $L^2$ product $$\langle ( \tilde{\mathcal{S}^\lambda}  - \Delta)\vh, \vh \rangle_{L^2} = - \lambda^2 \|\vh\|_{L^2}^2 - \sum_\pm \langle\mathcal{Q}^\pm_\lambda (\hat v \cdot \vh), \hat v \cdot \vh\rangle_{\cH^\pm},$$ 
and 
$$\langle ( \tilde{\mathcal{T}}_1^\lambda +\lambda \nabla ) k, \vh\rangle_{L^2} =  \sum_\pm \langle\mathcal{Q}^\pm_\lambda k , \hat v \cdot \vh  \rangle_{\cH^\pm}, \qquad \langle \tilde{\mathcal{T}}_2^\lambda k, \vh\rangle_{L^2} = - \sum_\pm \langle\mathcal{Q}^\pm_\lambda (\hat v_\varphi k) , \hat v\cdot \vh \rangle_{\cH^\pm}, $$ 
for each $k\in L^2_\tau(\Omega)$ and each $\vh \in L^2_\tau(\Omega;\tilde\RR^2)$. {\em (iii)} and {\em (iv)} are thus clear, following from the boundedness of $\mathcal{Q}^\pm_\lambda$. 
\end{proof}

\begin{corollary}\label{cor-PropAT} Let  $0 <  \lambda < \infty$.  

(i) $\mathcal{A}^\lambda_1$ and $\mathcal{A}^\lambda_2$ are well-defined operators from $\mathcal{X}\subset L^2_\tau(\Omega)$ into  $L^2_\tau(\Omega)$. 

(ii) ${\tilde\S}^\lb$ is well-defined from $\tilde{\mathcal{Y}}\subset L^2_\tau(\Omega;\tilde\RR^2)$ into $L^2_\tau(\Omega;\tilde\RR^2)$.

(iii) $\tilde{\mathcal{T}}^\lambda_1$ is well-defined from $\mathcal{X}_1 : = \{h \in H^1_\tau(\Omega)~:~ h = 0 \text{  on  }\partial\Omega\}$ into $L^2_\tau(\Omega;\tilde\RR^2)$. 

\end{corollary}

\begin{lemma}\label{lem-self-adj} Let $0<\lambda<\infty$. 

(i) $\mathcal{A}^\lambda_1$ and $\mathcal{A}^\lambda_2$ are self-adjoint operators on $L_\tau^2(\Omega)$. 

(ii) ${\tilde\S}^\lb$ is self-adjoint on $L^2_\tau(\Omega;\tilde\RR^2)$.

(iii) The adjoints of $\tilde{\mathcal{T}}^\lambda_1, \tilde{\mathcal{T}}^\lambda_2, \mathcal{B}^\lambda$ are as stated in the beginning of Section \ref{sec-defOp}. The domains of $(\tilde{\mathcal{T}}^\lambda_1)^*, (\tilde{\mathcal{T}}^\lambda_2)^*, (\mathcal{B}^\lambda)^*$ are $\{ \vh \in L^2_\tau(\Omega; \tilde \RR^2) ~:~\nabla \cdot \vh =0\}$, $L^2_\tau(\Omega; \tilde \RR^2)$, and $L^2_\tau(\Omega)$, respectively. In addition, the last two adjoints and $(\tilde\T_1^\lb +\lb \nabla)^*$ are bounded operators. 

(iv) The matrix operator on the left hand side of \eqref{matrix-Maxwell} is self-adjoint on $L_\tau^2(\Omega) \times L_\tau^2(\Omega)\times L_\tau^2(\Omega; \tilde \RR^2 )$ when considered with the domain $\mathcal{X} \times \mathcal{X} \times \tilde{\mathcal{Y}}$.
\end{lemma}
  
\begin{proof}  We first check  the adjoint formula for $\mathcal{Q}_\lambda^\pm$: 
\begin{equation}\label{adjointP}
\int_\Omega   \int_{\RR^3}\mu_e^\pm h(x,v) \mathcal{Q}^\pm_\lambda (g(x,v)) \; \mathrm{d}v\mathrm{d}x 
= \int_\Omega   \int_{\RR^3}\mu_e^\pm g(x,\mathcal{R} v) \mathcal{Q}^\pm_\lambda (h(x,\mathcal{R} v)) \; \mathrm{d}v\mathrm{d}x,          
 \end{equation} where $\mathcal{R} v := - v_r e_r - v_\theta e_\theta + v_\varphi e_\varphi$ for $v = v_re_r + v_\theta e_\theta + v_\varphi e_\varphi$. 
We shall prove \eqref{adjointP} for the $+$ case; the $-$ case is similar. We recall the definition of $\mathcal{Q}^+_\lambda$ from \eqref{def-Q} and  use the change of variables  
 $$(y,w):=(X^+(s;x,v),V^+(s;x,v)), \qquad (x,v) := (X^+(-s;y,w),V^+(-s;y,w)),$$
which has Jacobian one where it is defined. So, we can write the left side of \eqref{adjointP} as  
$$\begin{aligned}
\int_{-\infty}^0\int_\Omega   \int_{\RR^3}\lambda &e^{\lambda s}\mu_e^+h(x,v)\  g(X^+(s;x,v),V^+(s;x,v)) \; \mathrm{d}v\mathrm{d}x\mathrm{d}s
\\ &= \int_{-\infty}^0\int_\Omega   \int_{\RR^3}\lambda e^{\lambda s}\mu_e^+h(X^+(-s;y,w),V^+(-s;y,w))\  g(y,w) \; \mathrm{d}w\mathrm{d}y\mathrm{d}s.
\end{aligned}$$
Observe that the characteristics defined by the ODE \eqref{trajectory} and the specular boundary condition \eqref{traj-reflection1} 
are invariant under the time-reversal transformation   
 $s\mapsto -s$, $r \mapsto r$, $\theta \mapsto \theta$, $v_r\mapsto -v_r$, and $v_\theta \mapsto -v_\theta$, and $v_\varphi \mapsto v_\varphi$.  
Thus   
$$ X^+(-s;x,v) = X^+(s;x,\mathcal{R} v), \qquad V^+(-s;x,v) = \mathcal{R} V^+ (s;x,\mathcal{R}v), $$
at least if we avoid the boundary.  Changing variable in the $\mathrm{d}v\mathrm{d}x$ integral and using the invariance, we obtain 
$$\begin{aligned}
\int_{-\infty}^0\int_\Omega   \int_{\RR^3}\lambda &e^{\lambda s}\mu_e^+h(x,v)\ g(X^+(s;x,v),V^+(s;x,v)) \; \mathrm{d}v\mathrm{d}x\mathrm{d}s
\\ &= \int_{-\infty}^0\int_\Omega   \int_{\RR^3}\lambda e^{\lambda s}\mu_e^+h(X^+(s;y,\mathcal{R} w),\mathcal{R} V^+(s;y,\mathcal{R} w))\ g(y,w) \; \mathrm{d}w\mathrm{d}y\mathrm{d}s
\\ &= \int_{-\infty}^0\int_\Omega   \int_{\RR^3}\lambda e^{\lambda s}\mu_e^+h(X^+(s;x,v),\mathcal{R} V^+(s;x,v))\ g(x,\mathcal{R} v) \; \mathrm{d}v\mathrm{d}x\mathrm{d}s,
\end{aligned}$$
in which the last identity comes from the change of notation   $(x,v): = (y,\mathcal{R} w)$. 
By  definition of $\mathcal{Q}^+_\lambda$, this result is precisely the identity \eqref{adjointP}.

Thanks to the adjoint identity \eqref{adjointP} and the fact that $v_\varphi$ does not change under the mapping $\mathcal{R}$, the self-adjointness of the operators $\mathcal{A}_{1}^\lambda - \Delta $ and $\mathcal{A}_2^\lambda+\Delta$ is now clear.  
The self-adjointness of $\mathcal{A}_{1}^\lambda$ and $\mathcal{A}_{2}^\lambda$ now 
  follows from that of $\Delta$ with the Dirichlet boundary condition. Recall that the Dirichlet condition is built in the function space $\mathcal{X}$.  Similarly, the adjointness formula for $\mathcal{B}^\lambda$ and $(\mathcal{B}^\lambda)^*$ is also clear from \eqref{adjointP} together with the fact that $\int_{\RR^3}[(a+r\cos \theta)\mu^\pm_p + \hat v_\varphi \mu_e^\pm ] \; \mathrm{d}v = \int_{\RR^3}\D_{v_\varphi} [\mu]\; \mathrm{d}v = 0$.  

Next, by definition we have 
$$\langle ( \tilde{\mathcal{S}^\lambda}  - \Delta)\vh, \vg \rangle_{L^2} = - \lambda^2 \langle \vh,\vg \rangle_{L^2} - \sum_\pm \langle\mathcal{Q}^\pm_\lambda (\hat v \cdot \vh), \hat v \cdot \vg\rangle_{\cH^\pm},$$ 
and thus $ \tilde{\mathcal{S}^\lambda}  - \Delta$ is clearly self-adjoint in $L^2_\tau(\Omega; \tilde \RR^2)$ thanks to the identity \eqref{adjointP}. Let us check the self-adjointness of $\Delta$ with the boundary conditions built in $\tilde{\mathcal{Y}}$. Indeed, for $\vh,\vg \in \tilde{\mathcal{Y}}$, integration by parts yields
$$ \begin{aligned}
\langle \Delta \vh , \vg \rangle_{L^2} &= \langle \vh , \Delta \vg \rangle _{L^2}+ \int_{\D\Omega} (\D_r \vh \cdot \vg - \vh \cdot \D_r \vg)\; \mathrm{d}S_x \\
&= \langle \vh , \Delta \vg \rangle _{L^2}
+ \int_{\D\Omega} \{(\D_r h_r) g_r  - h_r (\D_r g_r ) + (\D_rh_\theta)g_\theta - h_\theta(\D_rg_\theta)\}\ \mathrm{d}S_x \\
&= \langle \vh , \Delta \vg \rangle_{L^2},
\end{aligned}$$
due to the boundary conditions $h_\theta = g_\theta =0$, and  $\D_r h_r + \frac{a+2\cos \theta}{a+\cos \theta} h_r   
=  \D_r g_r + \frac{a+2\cos \theta}{a+\cos \theta} g_r =0$. 
This proves that $\tilde{\mathcal{S}}^\lambda$ is self-adjoint in $L^2_\tau(\Omega; \tilde \RR^2)$ with domain $\tilde{\mathcal{Y}}$.

Finally, let us check the adjoint formulas for $\tilde{\mathcal{T}}^\lambda_j$. Again, by definition and the identity \eqref{adjointP}, we have
$$\begin{aligned}
\langle ( \tilde{\mathcal{T}}_1^\lambda +\lambda \nabla ) h, \vg\rangle_{L^2} &=  ~~\sum_\pm \langle\mathcal{Q}^\pm_\lambda h , \hat v \cdot \vg  \rangle_{\cH^\pm} = -  \sum_\pm \langle h , \mathcal{Q}^\pm_\lambda (\hat v \cdot \vg)  \rangle_{\cH^\pm}, 
\\
\langle \tilde{\mathcal{T}}_2^\lambda h, \vg\rangle_{L^2} &= - \sum_\pm \langle\mathcal{Q}^\pm_\lambda (\hat v_\varphi h) , \hat v\cdot \vg \rangle_{\cH^\pm} =  \sum_\pm \langle  h , \hat v_\varphi \mathcal{Q}^\pm_\lambda (\hat v\cdot \vg) \rangle_{\cH^\pm}. 
\end{aligned}$$ 
In addition, for each $(h,\vg) \in \mathcal{X}_1\times L^2(\Omega;\tilde\RR^2)$, integration by parts gives
$$ \langle \nabla h , \vg \rangle_{L^2} = -  \langle h, \nabla \cdot  \vg \rangle_{L^2} + \int_{\D\Omega} h \vg \cdot e_r \; \mathrm{d}S_x = -  \langle h, \nabla \cdot  \vg \rangle_{L^2} ,$$
in which the boundary term vanishes due to the Dirichlet boundary condition on $h$ (as an element in $\mathcal{X}_1$; see Corollary \ref{cor-PropAT}). This verifies the adjoint formulas of $\tilde{\mathcal{T}}^\lambda_j$ and $(\tilde{\mathcal{T}}^\lambda_j)^*$. The boundedness of these adjoint operators follows directly from the definition and the boundedness of $\mathcal{Q}^\pm_\lambda$. 

The adjoint property {\em (iv)} now follows from {\em (i)--(iii)}.   
\end{proof}

We now have the following lemma concerning the signs of two of these operators. 
\begin{lemma}\label{lem-AS}  

 (i) Let $0<\lambda<\infty$.  The operator $\mathcal{A}^\lambda_{1}$ is negative definite on $L^2_\tau  (\Omega)$ 
 with  domain $\mathcal{X}$ and it is a one-one map of $\mathcal{X}$   onto $L^2_\tau (\Omega)$.   Its inverse 
 $(\mathcal{A}^\lambda_{1})^{-1} $ maps $L^2_\tau(\Omega) $ into $ \mathcal X$ with operator bound independent of $\lb$.  
 
 (ii) For $\lb$ sufficiently large, the operator  $\tilde{\mathcal{S}}^\lambda$  is negative definite on $L^2(\Omega; \tilde \RR^2 )$ with its domain $\tilde{\mathcal{Y}}$.  The same is true of the operator $\tilde{\mathcal{S}}^0$.  
\end{lemma}

\begin{proof} 
For the first assertion, Lemma \ref{op-bounds} {\em (i)} yields   
 $$
 \Big| \int_\Omega  \int_{\RR^3}\mu_e^\pm h \mathcal{Q}^\pm_\lambda h \; \mathrm{d}v\mathrm{d}x\Big|  \le  \int_\Omega  \int_{\RR^3}|\mu_e^\pm||h|^2\; \mathrm{d}v\mathrm{d}x $$
so that 
$$  
\sum_\pm \int_\Omega  \int_{\RR^3}\mu^\pm_e h (1 - \mathcal{Q}^\pm_\lambda)h \; \mathrm{d}v   
\le 0.  $$
Thus from its definition, $\mathcal{A}_{1}^\lambda \le 0$, and $\mathcal{A}_1^\lambda h =0$ if and only if $h$ is a constant. The constant is necessarily zero due to the Dirichlet boundary condition.  
Since $\mathcal{A}_{1}^\lambda$ has discrete spectrum, it is invertible on the set orthogonal  to the 
kernel of  $\mathcal{A}_{1}^\lambda$, which is the entire $L^2_\tau  (\Omega)$ space.  

 
Next we check {\em (ii)}. By definition, for $\vh =  h_r e_r + h_\theta e_\theta \in \tilde{\mathcal{Y}}$, a direct calculation shows that 
\begin{equation}\label{hSh} \begin{aligned}
-\langle \mathcal{S}^\lambda \vh,\vh\rangle_{L^2} 
&= \langle (\lambda^2 - \Delta)\vh, \vh \rangle_{L^2}  \\&\quad +   \sum_\pm \Big[  \langle \mathcal{Q}^\pm_\lambda(\hat v_r h_r), \hat v_r h_r\rangle_{\cH}   +  \langle \mathcal{Q}^\pm_\lambda(\hat v_\theta h_\theta), \hat v_\theta h_\theta\rangle_{\cH}   -  2\langle \mathcal{Q}^\pm_\lambda(\hat v_r h_r), \hat v_\theta h_\theta\rangle_{\cH}   \Big],
 \end{aligned}\end{equation} for all $\lambda \ge 0$. Now the function space $\tilde{\mathcal{Y}}$ incorporates the boundary conditions $(a+\cos\theta)\D_r h_r + (a+2\cos \theta) h_r =0$ and $h_\theta =0$. Thus for $\vh\in \tilde{\mathcal{Y}}$ one has $$\begin{aligned}
  \langle (\lambda^2 - \Delta)\vh, \vh\rangle_{L^2}  &=  \lambda^2  \| \vh \|^2_{L^2} + \|\nabla \vh \|^2_{L^2}  - \int_{\D\Omega} \D_r (h_r e_r + h_\theta e_\theta) \cdot (h_r e_r + h_\theta e_\theta) \; \mathrm{d}S_x
  \\
  &=  \lambda^2  \| \vh \|^2_{L^2} + \|\nabla \vh\|^2_{L^2}  + \int_{\D\Omega} \frac{a+2\cos\theta}{a+\cos\theta} |h_r|^2 \; \mathrm{d}S_x.
  \end{aligned}$$
Using the boundedness of $\mathcal{Q}^\pm_\lambda$, whose operator norm is equal to one, we immediately obtain from \eqref{hSh} the inequality
$$ \begin{aligned}
-\langle \mathcal{S}^\lambda \vh,\vh\rangle_{L^2} &\ge (\lambda^2 -C_\mu) \| \vh \|^2_{L^2}  + \|\nabla \vh \|^2_{L^2}  + \int_{\D\Omega} \frac{a+2\cos\theta}{a+\cos\theta} |h_r|^2 \; \mathrm{d}S_x
 \end{aligned}
$$ for some constant $C_\mu$ that depends only on the decay assumption \eqref{mu-cond} on $|\mu_e^\pm|$. 
This proves the positive definiteness of $-\mathcal{S}^\lambda$ for $\lambda$ large.   
On the other hand, at $\lambda =0$, we recall that $\mathcal{P}^\pm$ are orthogonal projections on $\cH$,   
so that the bracketed terms in \eqref{hSh} are equal to  
$$\|\mathcal{P}^\pm(\hat v_r h_r)\|_{\cH} ^2  + \|\mathcal{P}^\pm(\hat v_\theta h_\theta)\|_{\cH}^2   -  2\langle \mathcal{P}^\pm(\hat v_r h_r), \mathcal{P}^\pm (\hat v_\theta h_\theta)\rangle_{\cH}. $$
This expression is clearly non-negative, which proves that $\mathcal{S}^0$ is also negative definite. 
\end{proof}

It will be crucial to understand the limiting behaviors at $\lambda =0$ and $\lambda \to  \infty$.  

\begin{lemma} \label{lem-PropQ} Let $\lambda,\mu > 0$ and $g \in \cH$. 


(i)  $\lim_{\lambda \to 0^+} \|\mathcal{Q}^\pm_\lambda g - \mathcal{P}^\pm g\|_{\cH} = 0$. 

(ii)  $\lim_{\lambda \to \infty} \|\mathcal{Q}^\pm_\lambda g - g\|_{\cH} =0$. 

(iii) There holds the uniform bound \begin{equation}\label{bound-opQ}\|\mathcal{Q}^\pm_\lambda - \mathcal{Q}^\pm_\mu\|_{\cH \mapsto \cH} \le 2 |\log \lambda - \log \mu|.\end{equation}


\end{lemma}

\begin{proof} We do not include the proof of {\em (i)} and {\em (ii)} as they are essentially the same as that in \cite[Lemma 4.1]{LS3} or in \cite[Lemma 2.6]{LS1} using the spectral measure. As for {\em (iii)}, it suffices to obtain the estimate \eqref{bound-opQ} for $\lambda>\mu$. We observe that 
$$\begin{aligned}
\Big| \langle (&\mathcal{Q}^\pm_\lambda - \mathcal{Q}^\pm_\mu ) h, g \rangle_\cH\Big|  
\\&= \Big| \int_{-\infty}^0\int_\Omega  \int_{\RR^3}\Big(\lambda e^{\lambda s} - \mu e^{\mu s}\Big) \mu_e^\pm h (X(s),V(s)) g(x,v)\; \mathrm{d}v\mathrm{d}x\mathrm{d}s\Big|
\\&\le  \int_{-\infty}^0 |\lambda e^{\lambda s} - \mu e^{\mu s}|\Big( \int_\Omega  \int_{\RR^3}|\mu_e^\pm|| g (x,v)|^2\; \mathrm{d}v\mathrm{d}x \Big)^{1/2} \Big( \int_\Omega \int_{\RR^3}|\mu^\pm_e || h (X(s),V(s)) |^2\; \mathrm{d}v\mathrm{d}x\Big)^{1/2} \mathrm{d}s
\\&
\le \Big(\int_{-\infty}^0|\lambda e^{\lambda s} - \mu e^{\mu s}|\; \mathrm{d}s \Big) \| h  \|_{\cH}\|g\|_{\cH}
\\&
\le \Big(\int_{-\infty}^0|\lambda -\mu| e^{\lambda s} + \mu |e^{\lambda s} - e^{\mu s}|\; \mathrm{d}s \Big) \| h  \|_{\cH}\|g\|_{\cH}
\\&
\le \frac{2|\lambda - \mu|}{\lambda}   \| h  \|_{\cH}\|g\|_{\cH} \le 2 |\log \lambda - \log \mu|  \| h  \|_{\cH}\|g\|_{\cH} 
.\end{aligned}     $$
\end{proof}

\subsection{Reduced matrix equation} 
Since the operator $\mathcal{A}_{1}^\lambda$ is invertible, we can eliminate $\phi$ in the matrix equation \eqref{matrix-Maxwell} by defining 
$$\phi := - (\mathcal{A}_{1}^\lambda)^{-1} \Big[(\mathcal{B}^\lambda)^* A_\varphi + (\tilde{\mathcal{T}}_{1}^\lambda)^* \tilde\vA\Big].$$ Thus  we introduce the operators 
$$ \begin{aligned}
\mathcal{L}^\lambda : &= \mathcal{A}_{2}^\lambda -  \mathcal{B}^\lambda (\mathcal{A}_{1}^\lambda)^{-1} (\mathcal{B}^\lambda)^*, 
\end{aligned}$$
$$\tilde{\mathcal{V}}^\lambda : = \tilde{\mathcal{T}}_2^\lambda - \tilde{\mathcal{T}}_1^\lambda (\mathcal{A}_1^\lambda)^{-1} (\mathcal{B}^\lambda)^* , \qquad\tilde{\mathcal{U}}^\lambda : =  \tilde{\mathcal{S}}^\lambda -  \tilde{\mathcal{T}}_1^\lambda (\mathcal{A}_1^\lambda)^{-1} (\tilde{\mathcal{T}}_1^\lambda)^*
 $$
and the reduced matrix operator  
\begin{equation}\label{def-M}\mathcal{M}^\lambda:= \begin{pmatrix} 
\mathcal{L}^\lambda &  (\tilde{\mathcal{V}}^\lambda)^*\\
\tilde{\mathcal{V}}^\lambda &  \tilde{\mathcal{U}}^\lambda
\end{pmatrix}.\end{equation}
By Lemmas \ref{op-bounds}, \ref{lem-self-adj}, \ref{lem-AS}, and Corollary \ref{cor-PropAT}, the operators $ \mathcal{B}^\lambda (\mathcal{A}_{1}^\lambda)^{-1} (\mathcal{B}^\lambda)^*$, $(\tilde{\mathcal{T}}_1^\lambda+\lambda\nabla) (\mathcal{A}_1^\lambda)^{-1} (\tilde{\mathcal{T}}_1^\lambda)^*$, and $(\tilde{\mathcal{T}}_1^\lambda+ \lambda \nabla) (\mathcal{A}_1^\lambda)^{-1} (\mathcal{B}^\lambda)^* $ are bounded. Consequently, the operator $\mathcal{L}^\lambda$ is well-defined from $\mathcal{X}$ to $L^2_\tau(\Omega)$,  
$\tilde{\mathcal{U}}^\lambda $  from $\tilde{\mathcal{Y}}$ to $L^2_\tau(\Omega; \tilde \RR^2 )$, 
and $\tilde{\mathcal{V}}^\lambda$  from $L^2_\tau  (\Omega)$ to $L^2_\tau(\Omega; \tilde \RR^2 )$.   
In addition, $\mathcal{L}^\lambda$ is self-adjoint on $L^2_\tau(\Omega)$, $\tilde{\mathcal{U}}^\lambda$ is self-adjoint on $L^2_\tau(\Omega; \tilde \RR^2 )$, and so the reduced matrix $\mathcal{M}^\lambda$ is self-adjoint on $L_\tau^2(\Omega) \times L_\tau^2(\Omega; \tilde \RR^2 )$ when considered with the domain $\mathcal{X} \times \tilde{\mathcal{Y}}$.


\begin{lemma}\label{lem-Mprop} Let $0<\lambda<\infty$.   

(i) The matrix equation \eqref{matrix-Maxwell} is equivalent to the reduced equation 
\begin{equation}\label{matrix-Maxwell01} 
\mathcal{M}^\lambda \begin{pmatrix} A_\varphi \\ \tilde \vA \end{pmatrix}  =  \begin{pmatrix} 
\mathcal{L}^\lambda &  (\tilde{\mathcal{V}}^\lambda)^*\\
\tilde{\mathcal{V}}^\lambda &  \tilde{\mathcal{U}}^\lambda
\end{pmatrix} \begin{pmatrix} A_\varphi \\ \tilde \vA \end{pmatrix} =0 .\end{equation} 

(ii) $\mathcal{L}^\lambda \ge 0$ for $\lambda$ large, and for each $h \in \mathcal{X}$, $\lim_{\lambda\to 0^+}\|( \mathcal{L}^\lambda - \mathcal{L}^0)h\|_{L^2}  =0$.

(iii) The smallest eigenvalue $ \kappa^\lambda: = \inf _{h} ~\langle \mathcal{L}^\lambda h, h \rangle_{L^2}$ 
 of $\mathcal{L}^\lambda$ is continuous in $\lambda>0$, where the infimum is taken over $h \in \mathcal{X}$ with $\|h\|_{L^2} =1$.

\end{lemma}
\begin{proof} Directly from the definitions, we have {\em (i)}. For each $h \in \mathcal{X}$, we have 
$$\begin{aligned}
\langle \mathcal{L}^\lambda h,h\rangle_{L^2} 
&=\langle \mathcal{A}_{2}^\lambda h,h\rangle_{L^2} -   \langle (\mathcal{A}_{1}^\lambda)^{-1} (\mathcal{B}^\lambda)^* h, (\mathcal{B}^\lambda)^* h \rangle_{L^2}
\\
&=\lambda^2 \|h\|_{L^2}^2 + \langle (-\Delta)h,h\rangle_{L^2} + \langle (\mathcal{A}_{2}^\lambda - \lambda^2 + \Delta) h,h\rangle_{L^2} -   \langle (\mathcal{A}_{1}^\lambda)^{-1} (\mathcal{B}^\lambda)^* h, (\mathcal{B}^\lambda)^* h \rangle_{L^2},
\end{aligned}$$
in which the second and fourth terms are nonnegative by Lemma \ref{lem-AS}, and the third term is bounded thanks to Lemma \ref{op-bounds} {\em (ii)} by $C_0\|h\|_{L^2}^2$ for some constant $C_0$ independent of $\lambda$.  Thus, when taking $\lambda$ large, the first term dominates and so $\mathcal{L}^\lambda \ge 0$. In addition, it follows directly from the definition that 
$$( \mathcal{A}_{2}^\lambda - \mathcal{A}_{2}^0) h = \lambda^2 h -\sum_\pm \int_{\RR^3}\mu_e^\pm \hat v_\varphi  ( \mathcal{Q}^\pm_\lambda - \mathcal{P}^\pm ) (\hat v_\varphi h) \; \mathrm{d}v$$
and so $\|( \mathcal{A}_{2}^\lambda - \mathcal{A}_{2}^0) h\|_{L^2} \le \lambda^2 \|h\|_{L^2} +C_0 \sum_\pm\|( \mathcal{Q}^\pm_\lambda - \mathcal{P}^\pm ) (\hat v_\varphi h) \|_{\cH} \to 0$ as $\lambda \to 0$, by Lemma \ref{lem-PropQ} {\em (i)}. Similarly, we have the same convergence for $\mathcal{A}_{1}^\lambda $ and $\mathcal{B}^\lambda$ as $\lambda \to 0$. Therefore, $\mathcal{L}^\lambda h$ converges strongly in the $L^2$ norm to $\mathcal{L}^0h$, for each $h \in \mathcal{X}$. Here, we note that the operator norms of $(\mathcal{A}^\lambda_1)^{-1}$ and $\mathcal{B}^\lambda$ are independent of $\lambda$. This proves {\em (ii)}. 
  
Let us check {\em (iii)}. Let $\lambda,\mu>0$. For all $h,g \in \mathcal{X}$, we write 
$$\langle ( \mathcal{A}_{2}^\lambda - \mathcal{A}_{2}^\mu) h ,g\rangle_{L^2}= (\lambda^2-\mu^2) \langle h,g\rangle_{L^2} + \sum_\pm \langle ( \mathcal{Q}^\pm_\lambda - \mathcal{Q}^\pm _\mu) (\hat v_\varphi h), \hat v_\varphi g \rangle_{\cH} ,$$
which together with Lemma \ref{lem-PropQ} {\em (iii)} yields
$\| \mathcal{A}_{2}^\lambda - \mathcal{A}_{2}^\mu \|_{L^2 \mapsto L^2} \le  |\lambda^2-\mu^2| +C_0 |\log \lambda - \log \mu|.$ Similar estimates hold for $\mathcal{B}^\lambda$ and $(\mathcal{A}^\lambda_1)^{-1}$. Consequently, we obtain 
\begin{equation}\label{norm-opL}\| \mathcal{L}^\lambda - \mathcal{L}^\mu \|_{L^2 \mapsto L^2} \le  C_0 \Big (|\lambda^2-\mu^2| +|\log \lambda - \log \mu|\Big ), \qquad \forall ~\lambda,\mu >0. \end{equation}
This proves that $\mathcal{L}^\lambda$ is continuous in the operator norm for $\lambda\in (0,\infty)$. In particular so is the lowest eigenvalue $\kappa^\lambda$ of $\mathcal{L}^\lambda$.   
\end{proof}

\begin{lemma}[Continuity of limits in $\lambda$]\label{lem-UV} Fix $\mu>0$. 

(i) $\lim_{\lambda \to \mu} \| \mathcal{L}^\lambda-  \mathcal{L}^\mu\|_{L^2\mapsto L^2} =0$. The same convergence holds for $\tilde{\mathcal{S}}^\lambda, \tilde{\mathcal{T}}^\lambda_1 + \lambda \nabla$, and $\tilde{\mathcal{T}}^\lambda_2$. 

(ii) For $\vh,\vg \in \tilde{\mathcal{Y}}$, $\lim_{\lambda \to \mu}  \langle (\tilde{\mathcal{U}}^\lambda-\tilde{\mathcal{U}}^\mu) \vh, \vg \rangle_{L^2}=0$ and $\lim_{\lambda \to \mu}  \|(\tilde{\mathcal{V}}^\lambda -\tilde{\mathcal{V}}^\mu )^*\vh \| _{L^2} =0$. The same convergence holds for the case $\mu=0$ with $\tilde{\mathcal{V}}^0 =0$.  
 

(iii) For $h\in \mathcal{X}, \vg\in \tilde{\mathcal{Y}}$, $\lim_{\lambda \to \infty} \langle \tilde{\mathcal{V}}^\lambda h,\vg \rangle _{L^2} = 0$.  

\end{lemma}

\begin{proof} Of course, the estimate \eqref{norm-opL} proves the convergence of $\mathcal{L}^\lambda$ as claimed in {\em (i)}. As for the other operators, we take $h\in \mathcal{X}, \vh \in \tilde{\mathcal{Y}}, \vg \in L^2(\Omega; \tilde \RR^2)$, and write 
$$\langle ( \tilde{\mathcal{S}}^\lambda  - \tilde{\mathcal{S}}^\mu)\vh, \vg \rangle_{L^2} = - (\lambda^2 - \mu^2)\langle \vh,\vg \rangle_{L^2} - \sum_\pm \langle(\mathcal{Q}^\pm_\lambda - \mathcal{Q}^\pm_\mu)(\hat v \cdot \vh), \hat v \cdot \vg\rangle_{\cH^\pm},$$ 
and $$\begin{aligned}
\langle ( \tilde{\mathcal{T}}_1^\lambda - \tilde{\mathcal{T}}_1^\mu) h, \vg\rangle_{L^2} + (\lambda - \mu)\langle \nabla  h, \vg\rangle_{L^2} &= ~~\sum_\pm \langle(\mathcal{Q}^\pm_\lambda - \mathcal{Q}^\pm_\mu) h , \hat v \cdot \vg  \rangle_{\cH^\pm} ,\\
\langle ( \tilde{\mathcal{T}}_2^\lambda - \tilde{\mathcal{T}}_2^\mu) h, \vg\rangle_{L^2} &= - \sum_\pm \langle( \mathcal{Q}^\pm_\lambda - \mathcal{Q}^\pm_\mu) (\hat v_\varphi h) , \hat v\cdot \vg \rangle_{\cH^\pm} . 
\end{aligned}$$
Now it is clear that estimate \eqref{bound-opQ} yields the same bound as in \eqref{norm-opL} for $\tilde{\mathcal{S}}^\lambda, \tilde{\mathcal{T}}^\lambda_1 + \lambda \nabla$, and $\tilde{\mathcal{T}}^\lambda_2$. This proves {\em (i)}. 

The failure of the continuity of $\tilde{\mathcal{U}}^\lambda,(\tilde{\mathcal{V}}^\lambda)^* $ in the operator norm is due to the presence of $\lambda \nabla \cdot $, which comes from $\tilde{\mathcal{T}}^\lambda_1$. However, this term vanishes when the operator acts on functions in the function space $\tilde{\mathcal{Y}}$ due to the Coulomb gauge constraint. Precisely, for each $\vh,\vg\in \tilde{\mathcal{Y}}$, we have 
\begin{equation}\label{Id-Ulambda}\begin{aligned}
\langle \tilde{\mathcal{U}}^\lambda \vh,\vg\rangle_{L^2} &= \langle  \tilde{\mathcal{S}}^\lambda \vh,\vg\rangle_{L^2} - \langle (\mathcal{A}_1^\lambda)^{-1} (\tilde{\mathcal{T}}_1^\lambda)^* \vh, (\tilde{\mathcal{T}}_1^\lambda )^*\vg\rangle_{L^2} 
\\
&= \langle  \tilde{\mathcal{S}}^\lambda \vh,\vg\rangle_{L^2} - \langle (\mathcal{A}_1^\lambda)^{-1} (\tilde{\mathcal{T}}_1^\lambda + \lambda \nabla)^* \vh, (\tilde{\mathcal{T}}_1^\lambda +\lambda \nabla)^*\vg\rangle_{L^2} .
\end{aligned}\end{equation}
The claimed convergence now follows directly from {\em (i)}. A similar observation applies to $\tilde{\mathcal{V}}^\lambda$. For the limit when $\lambda \to 0$, we use the strong convergence of $\mathcal{Q}^\pm_\lambda $ in Lemma \ref{lem-PropQ} {\em (i)}, instead of {\em (iii)}. We thus obtain the last statement in {\em (ii)}. 

As for {\em (iii)}, we write for $h \in \mathcal{X}$ and $\vg \in \tilde{\mathcal{Y}}$, 
$$\begin{aligned}
\langle \tilde{\mathcal{T}}_1^\lambda h, \vg\rangle_{L^2} &= - \lambda \langle \nabla  h, \vg\rangle_{L^2} + \sum_\pm \langle \mathcal{Q}^\pm_\lambda h , \hat v \cdot \vg  \rangle_{\cH^\pm}  = \sum_\pm \langle \mathcal{Q}^\pm_\lambda h , \hat v \cdot \vg  \rangle_{\cH^\pm} ,\\
\langle \tilde{\mathcal{T}}_2^\lambda  h, \vg\rangle_{L^2} &= - \sum_\pm \langle \mathcal{Q}^\pm_\lambda (\hat v_\varphi h) , \hat v\cdot \vg \rangle_{\cH^\pm} . 
\end{aligned}$$
Lemma \ref{lem-PropQ} {\em (ii)} thus yields $\langle \tilde{\mathcal{T}}_1^\lambda h, \vg\rangle_{L^2} \to \sum_\pm \langle h , \hat v \cdot \vg  \rangle_{\cH^\pm} $, which vanishes due to the oddness of $\hat v \cdot \vg$ in $(v_r,v_\theta)$, as $\lambda\to \infty$. By the same reason, we also have $\langle \tilde{\mathcal{T}}_2^\lambda  h, \vg\rangle_{L^2} \to 0$ as $\lambda \to \infty$. The claim {\em (iii)} then follows by the definition. 
\end{proof}

\begin{lemma} \label{lem-Ulambda}   
(i) There exist fixed positive numbers $\lambda_1$ and $\lambda_2$ so that for all $0<\lambda \le \lambda_1$ and all $\lambda \ge \lambda_2$, the operator $\tilde{\mathcal{U}}^\lambda$ is one to one and onto from $\tilde{\mathcal{Y}}$ to $L^2(\Omega; \tilde \RR^2 )$.  
(ii) Furthermore, there holds \begin{equation}\label{lower-boundU} - \langle \tilde{\mathcal{U}}^\lambda \vh,\vh \rangle \ge C_0 \Big( \|\vh\|^2_{L^2} + \|\nabla \vh \|_{L^2}^2 \Big),\qquad\quad \forall ~ \vh \in \tilde{\mathcal{Y}},\end{equation}
for some positive constant $C_0$ that is independent of $\lambda$ within these intervals. 
\end{lemma}
\begin{proof} 
{\em (i)}  Let us assume for a moment that \eqref{lower-boundU} is proved. It is then clear that the operator $\tilde{\mathcal{U}}^\lambda$ is one to one from $\tilde{\mathcal{Y}}$ to $L^2(\Omega; \tilde \RR^2 )$. We shall use the standard Lax-Milgram theorem to prove that the operator is onto. Let us first denote by $\tilde{\mathcal{Y}}_1$ the space of vector functions $\vh = h_r e_r + h_\theta e_\theta$ in $H^1(\Omega; \tilde \RR^2 )$ such that {\em (i)} the Coulomb gauge constraint $\nabla \cdot \vh =0$ is valid on $\Omega$ and {\em (ii)} the boundary conditions  $h_\theta=0$ and $ \D_r h_r + \frac{a+2\cos \theta }{a+\cos \theta}h_r  =0$ are also valid (in the weak sense).  We recall that by definition,
\begin{equation}\label{def-Uh} \tilde{\mathcal{U}}^\lambda \vh = (-\lambda^2 + \Delta ) \vh  + \sum_\pm \int_{\RR^3}\frac{\tilde v}{\langle v \rangle}\mu^\pm_e \mathcal{Q}^\pm_\lambda(\hat v\cdot \vh) \; \mathrm{d}v  -  \tilde{\mathcal{T}}_1^\lambda (\mathcal{A}_1^\lambda)^{-1} (\tilde{\mathcal{T}}_1^\lambda)^* \vh ,\end{equation}
and thus let us introduce a bilinear operator
$$\begin{aligned}
\mathbb{B}^\lambda(\vh,\vg): &=  \int_\Omega \Big[ \lambda^2\vh \cdot \vg + \nabla \vh \cdot \nabla \vg \Big] \; \mathrm{d}x   +  \int_{\D\Omega} \frac{a+2\cos\theta}{a+\cos\theta} h_r g_r \; \mathrm{d}S_x
\\
& \quad + \sum_\pm \langle \mathcal{Q}^\pm_\lambda(\hat v\cdot \vh), \hat v \cdot \vg \rangle_{\cH}   + \langle (\mathcal{A}_1^\lambda)^{-1} (\tilde{\mathcal{T}}_1^\lambda)^* \vh , (\tilde{\mathcal{T}}_1^\lambda)^* \vg \rangle   ,  
\end{aligned}$$
for all $\vh,\vg \in \tilde{\mathcal{Y}}_1$. By \eqref{lower-boundU}, $\mathbb{B}^\lambda$ is coercive on $\tilde{\mathcal{Y}}_1 \times \tilde{\mathcal{Y}}_1$ when $\lambda$ is small or large, and thus by the Lax-Milgram theorem, for each $f \in L^2(\Omega; \tilde \RR^2 )$, there exists an $\vh \in \tilde{\mathcal{Y}}_1$ so that $\tilde{\mathcal{U}}^\lambda \vh = f$ in the distributional sense.  Furthermore, by the equation \eqref{def-Uh}, it follows that $\Delta \vh \in L^2 (\Omega; \tilde \RR^2 )$. Thus, $\vh \in H^2(\Omega; \tilde\RR^2)\cap \tilde{\mathcal{Y}}_1 = \tilde{\mathcal{Y}}$, and the operator $\tilde{\mathcal{U}}^\lambda$ is one to one and onto from $\tilde{\mathcal{Y}}$ to $L^2(\Omega; \tilde \RR^2 )$.  

{\em (ii)}  It remains to prove the inequality \eqref{lower-boundU}. For all $\vh = h_r e_r + h_\theta e_\theta \in \tilde{\mathcal{Y}}$, similar calculations as done in \eqref{hSh} using the boundary conditions incorporated in $\tilde{\mathcal{Y}}$ yield
\begin{equation}\label{hUh} \begin{aligned}
- \langle \tilde{\mathcal{U}}^\lambda \vh,\vh\rangle_{L^2}  &=  \lambda^2  \| \vh \|^2_{L^2} + \|\nabla \vh \|^2_{L^2}  + \int_{\D\Omega} \frac{a+2\cos\theta}{a+\cos\theta} |h_r|^2 \; \mathrm{d}S_x + \langle (\mathcal{A}_{1}^\lambda)^{-1}  (\tilde{\mathcal{T}}_1^\lambda)^* \vh, (\tilde{\mathcal{T}}_1^\lambda)^* \vh \rangle 
\\&\quad +   \sum_\pm \Big[  \langle \mathcal{Q}^\pm_\lambda(\hat v_r h_r), \hat v_r h_r\rangle_{\cH}   +  \langle \mathcal{Q}^\pm_\lambda(\hat v_\theta h_\theta), \hat v_\theta h_\theta\rangle_{\cH}   -  2\langle \mathcal{Q}^\pm_\lambda(\hat v_r h_r), \hat v_\theta h_\theta\rangle_{\cH}   \Big]
 \end{aligned}\end{equation} for all $\lambda \ge 0$. Thanks to the boundedness of the operators $\mathcal{Q}^\pm_\lambda $ and $ (\mathcal{A}_{1}^\lambda)^{-1}$, we therefore obtain  
\begin{equation}\label{boundU-infty}- \langle \tilde{\mathcal{U}}^\lambda \vh,\vh\rangle_{L^2}  \ge  (\lambda^2 -C_0) \| \vh \|^2_{L^2} + \|\nabla \vh\|^2_{L^2}  + \int_{\D\Omega} \frac{a+2\cos\theta}{a+\cos\theta} |h_r|^2 \; \mathrm{d}S_x,
\end{equation}
for some fixed constant $C_0$. This proves the lower bound \eqref{lower-boundU} in the case when $\lambda$ is large. 

For the case of small $\lambda$, we prove the estimate \eqref{lower-boundU} by contradiction; that is, assume that there are sequences $\lambda_n \to 0$ and $\vh_n \in \tilde{\mathcal{Y}}$ such that $\|\vh_n\|_{L^2}^2 + \|\nabla \vh_n \|_{L^2} ^2=1$ but $\langle \tilde{\mathcal{U}}^{\lambda_n} \vh_n,\vh_n \rangle \to 0$.  We note that, up to a subsequence, the sequence $\vh_n$ converges weakly to $\vh_{0}$ in $H^1(\Omega; \tilde \RR^2 )$ and $\vh_n$ converges strongly to $\vh_0$ in $L^2(\Omega; \tilde \RR^2 )$ as $n \to \infty$. In particular, $\vh_0\in \tilde{\mathcal{Y}}_1$ and $\|\vh_0\|_{L^2}^2 +  \|\nabla \vh_0\|_{L^2}^2 =1$. Furthermore, $ \|(\tilde{\mathcal{T}}_1^{\lambda_n})^* \vh_n\|_{L^2} \to 0$ and the bracket in \eqref{hUh} with $\lambda = \lambda_n$ and $\vh = \vh_n$ converges to 
 $$\|\mathcal{P}^\pm(\hat v_r h_{r0})\|_{\cH} ^2  + \|\mathcal{P}^\pm(\hat v_\theta h_{\theta 0})\|_{\cH}^2   -  2\langle \mathcal{P}^\pm(\hat v_r h_{r0}), \mathcal{P}^\pm (\hat v_\theta h_{\theta 0})\rangle_{\cH} \ge 0, $$ as $n \to \infty$. Here, $h_{r0} = \vh_0 \cdot e_r$ and $h_{\theta 0} = \vh_0 \cdot e_\theta$. By view of the identity \eqref{hUh}, we must then have 
$$ \|\nabla\vh_0 \|^2_{L^2}  + \int_{\D\Omega} \frac{a+2\cos\theta}{a+\cos\theta} |h_{r0}|^2 \; \mathrm{d}S_x =0,$$
 which together with its boundary conditions yields $\vh_0 = 0$. This contradicts the fact that $\|\vh_0\|_{L^2}^2 +  \|\nabla \vh_0\|_{L^2}^2 =1$, and so completes the proof of \eqref{lower-boundU} and therefore of the lemma.
 \end{proof}

\begin{corollary}\label{lem-negL} If $\mathcal{L}^0 \not \ge 0$, there exists $\lambda_3>0$ so that $\mathcal{L}^\lambda\not \ge 0$ for any $\lambda \in [0,\lambda_3]$.
\end{corollary}
\begin{proof} It follows directly from the strong convergence of $\mathcal{L}^\lambda$ to $\mathcal{L}^0$; see Lemma \ref{lem-Mprop} {\em (ii)}.  
\end{proof}


\subsection{Solution of the matrix equation}
We wish to construct a nonzero solution $(k ,\vh)$ in $\mathcal{X}\times \tilde{\mathcal{Y}}$ to the reduced matrix equation \eqref{matrix-Maxwell01}:
\begin{equation}\label{matrix-Maxwell02} 
\mathcal{M}^\lambda \begin{pmatrix} k  \\ \vh \end{pmatrix} =0 ,\end{equation} 
 for some $\lambda>0$, with $\mathcal{M}^\lambda$ defined as in \eqref{def-M}. In order to solve this equation, we will count the number of negative eigenvalues and therefore we are forced to truncate the second component to finite dimensions.  
 
 
 We recall the definition \eqref{Y-space} of the space $\tilde{\mathcal{Y}}$, which consists of functions $\vh$ with values in $\tilde \RR^2$.    
 The Laplacian $\Delta$ defined on this space is elliptic with elliptic boundary conditions.  
 Let $\{\tilde\psi_j\}_{j=1}^\infty$ be the eigenfunctions of $-\Delta$  in $\tilde{\mathcal{Y}}$, 
 chosen to be an orthonormal basis of $L^2(\Omega;\tilde \RR^2 )$, where the eigenvalues are $\sigma_j$.   
 Thus $-\Delta \tilde\psi_j = \sigma_j  \tilde\psi_j$ for $j=1,2,....$
Let $\mathbb{P}_n: = \tilde{\mathcal{Y}}^* \to \RR^n$ be the following projection 
and let its adjoint $\mathbb{P}_n^* = \RR^n \to \tilde{\mathcal{Y}}$ be defined by 
 $$ 
 \mathbb{P}_n \vh = \{ \langle \vh, \tilde\psi_j \rangle \} _{j=1}^n, \qquad \mathbb{P}_n^* b = \sum _{j=1}^n b_j \tilde\psi_j $$ 
  for $\vh \in \tilde{\mathcal{Y}}^*$ and $b = (b_1,\cdots,b_n)\in \RR^n$. 
 Here $\langle\ ,\ \rangle$ denotes the inner product in $L^2(\Omega;\tilde \RR^2 )$.  
 Then for each $n$ and $\lambda$, the operator $ \mathbb{P}_n \tilde{\mathcal{U}}^\lambda \mathbb{P}_n^* $ 
 is a symmetric $n\times n$ matrix with $(j,k)$ component equal to  
 $\langle \tilde{\mathcal{U}}^\lambda \psi_k, \psi_j \rangle$.  
We denote by $\mathcal{M}_n^\lambda$ the truncated matrix operator  
$$\mathcal{M}^\lambda_n:= \begin{pmatrix} 
\mathcal{L}^\lambda &  (\tilde{\mathcal{V}}^\lambda)^* \mathbb{P}_n^*\\
\mathbb{P}_n\tilde{\mathcal{V}}^\lambda & \mathbb{P}_n\tilde{\mathcal{U}}^\lambda \mathbb{P}_n^*
\end{pmatrix}$$
which is a well-defined self-adjoint operator from $\mathcal{X} \times \RR^n$ to $L^2(\Omega) \times \RR^n$ 
with discrete spectrum.   We first show that for each $n$ the truncated equation can be solved.  

\begin{lemma} \label{lem-kerMn}   
Assume that $\mathcal{L}^0\not \ge 0$.   There exist fixed numbers $0<\lb_4 < \lb_5 <\infty$ such that, 
for all $n \ge 1$, there exist a number $\lambda_n\in [\lambda_4,\lambda_5]$   
 and a nonzero vector function $(k_n ,b_n) \in \mathcal{X}\times \RR^n$ such that 
 $$\mathcal{M}_n^{\lambda_n} \begin{pmatrix} k_n  \\ b_n \end{pmatrix} =0 .$$
\end{lemma}

\begin{proof} By Lemma \ref{lem-Mprop} {\em (ii)}, Lemma \ref{lem-Ulambda}, and Corollary \ref{lem-negL}, we can choose $\lambda_4, \lambda_5>0$ independent of $n$ so that $\mathcal{L}^\lambda \not \ge 0$ for all $\lambda \in (0,\lambda_4]$, $\mathcal{L}^\lambda \ge 0$ for $\lambda > \lambda_5$, and $\tilde{\mathcal{U}}^\lambda$ is invertible for $\lambda \in (0,\lambda_4]$ and $\lambda \ge \lambda_5$. It follows that $ - \mathbb{P}_n\tilde{\mathcal{U}}^\lambda \mathbb{P}_n^*$ is a symmetric positive matrix and is therefore invertible for each $\lambda$ in $(0,\lambda_4] \cup [\lambda_5,\infty)$. 
So, the matrix $\mathcal{M}^\lambda_n$ has the same number of negative eigenvalues as 
$$ 
(\mathcal{G}^\lambda_n)^*\mathcal{M}_n^\lambda \mathcal{G}^\lambda_n:=  \begin{pmatrix} 
\mathcal{L}^\lambda - (\mathbb{P}_n\tilde{\mathcal{V}}^\lambda)^*(\mathbb{P}_n\tilde{\mathcal{U}}^\lambda \mathbb{P}_n^*)^{-1} \mathbb{P}_n\tilde{\mathcal{V}}^\lambda &0\\0 & \mathbb{P}_n\tilde{\mathcal{U}}^\lambda \mathbb{P}_n^*
\end{pmatrix} , $$ 
where $\mathcal{G}^\lambda_n:=\begin{pmatrix} 
I & 0\\ -(\mathbb{P}_n\tilde{\mathcal{U}}^\lambda \mathbb{P}_n^*)^{-1} \mathbb{P}_n\tilde{\mathcal{V}}^\lambda 
& I  \end{pmatrix}.$ 
In addition, since the matrix $ \mathbb{P}_n\tilde{\mathcal{U}}^\lambda \mathbb{P}_n^*$ is negative definite, it has exactly $n$ negative eigenvalues for each $\lambda$ in $(0,\lambda_4] \cup [\lambda_5,\infty)$.  
Now for $\lambda \ge \lambda_5$ and for all $h \in \mathcal{X}$, Lemma \ref{lem-Mprop} {\em (ii)} yields  
$$\langle \mathcal{L}^\lambda h- (\mathbb{P}_n\tilde{\mathcal{V}}^\lambda)^*(\mathbb{P}_n\tilde{\mathcal{U}}^\lambda \mathbb{P}_n^*)^{-1} \mathbb{P}_n\tilde{\mathcal{V}}^\lambda h ,h\rangle  = \langle \mathcal{L}^\lambda h,h\rangle  - \langle (\mathbb{P}_n\tilde{\mathcal{U}}^\lambda \mathbb{P}_n^*)^{-1} \mathbb{P}_n\tilde{\mathcal{V}}^\lambda h, \mathbb{P}_n\tilde{\mathcal{V}}^\lambda h \rangle \ge 0.$$ 
Thus $\mathcal{M}^\lambda_n$ has exactly $n$ negative eigenvalues for $\lambda \ge \lambda_5$,  
all of which come  from the lower right corner of the matrix.  

Next we study the behavior of the matrix when $\lambda$ is small. We claim that 
 for each $h\in \mathcal{X}$ and each $n\ge1$, the convergence
 \begin{equation}\label{conv-LU} 
 \langle \mathcal{L}^\lambda h,h\rangle  
- \langle (\mathbb{P}_n\tilde{\mathcal{U}}^\lambda \mathbb{P}_n^*)^{-1} \mathbb{P}_n\tilde{\mathcal{V}}^\lambda h,  
 \mathbb{P}_n\tilde{\mathcal{V}}^\lambda h \rangle  
 \to \langle \mathcal{L}^0 h,h \rangle
 \end{equation} 
holds as $\lambda \to 0$.    By Lemma \ref{lem-Mprop}, $\mathcal{L}^\lambda h$ converges 
 strongly in the $L^2$ norm to $\mathcal{L}^0h$ as $\lambda \to 0$.  
 Thus to prove the claim it remains to show that the second term on the left of \eqref{conv-LU} tends to zero 
 for each $h \in \mathcal{X}$. 
Indeed, 
taking account of the lower bound \eqref{lower-boundU} on $-\tilde{\mathcal{U}}^\lambda$, we have 
$$ -\langle \mathbb{P}_n\tilde{\mathcal{U}}^\lambda \mathbb{P}_n^* b , b \rangle  = -\langle \tilde{\mathcal{U}}^\lambda (\mathbb{P}_n^* b ) ,   \mathbb{P}_n^* b  \rangle \ge C_0 \|  \mathbb{P}_n^* b  \|_{L^2}^2,$$  
for the same $C_0$ as defined in \eqref{lower-boundU}, for arbitrary $b \in \RR^n$.   
If we now take $b = (\mathbb{P}_n\tilde{\mathcal{U}}^\lambda \mathbb{P}_n^*)^{-1} a$ in this inequality for arbitrary $a\in \RR^n$, it follows that 
$$  \| \mathbb{P}_n^* (\mathbb{P}_n\tilde{\mathcal{U}}^\lambda \mathbb{P}_n^*)^{-1} a \|_{L^2}^2   
\le C_0^{-1}| \langle a  ,  (\mathbb{P}_n\tilde{\mathcal{U}}^\lambda \mathbb{P}_n^*)^{-1} a \rangle|  
= C_0^{-1}| \langle \mathbb{P}_n^*a  ,  \mathbb{P}_n^* (\mathbb{P}_n\tilde{\mathcal{U}}^\lambda \mathbb{P}_n^*)^{-1} a \rangle| ,$$
and therefore 
$$  \| \mathbb{P}_n^* (\mathbb{P}_n\tilde{\mathcal{U}}^\lambda \mathbb{P}_n^*)^{-1} a \|_{L^2}  
\le C_0^{-1} \|\mathbb{P}_n^*a \|_{L^2} $$
for all $a\in \RR^n$.  Substituting $a = \mathbb{P}_n\tilde{\mathcal{V}}^\lambda h$, for any $h\in \mathcal{X}$, we have 
$$\begin{aligned} 
|\langle \mathbb{P}_n^*(\mathbb{P}_n\tilde{\mathcal{U}}^\lambda 
\mathbb{P}_n^*)^{-1} \mathbb{P}_n\tilde{\mathcal{V}}^\lambda h,\tilde{\mathcal{V}}^\lambda h \rangle |  
\le C_0^{-1} \| \tilde{\mathcal{V}}^\lambda h\|_{L^2}^2 ,
\end{aligned}$$
which converges to zero as $\lambda \to 0$ by Lemma \ref{lem-UV} {\em (ii)}. This proves the claim \eqref{conv-LU} for each $n\ge 1$. 

It follows from \eqref{conv-LU} and the assumption that $\mathcal{L}^0 \not \ge 0$, 
that the operator 
$ \mathcal{L}^\lambda - (\mathbb{P}_n\tilde{\mathcal{V}}^\lambda)^*(\mathbb{P}_n\tilde{\mathcal{U}}^\lambda \mathbb{P}_n^*)^{-1} \mathbb{P}_n\tilde{\mathcal{V}}^\lambda$ 
must have at least one negative eigenvalue for small $\lambda$.  
Thus  taking $\lambda_4$ even smaller if necessary but still independent of $n$, 
we have  shown that for all $\lambda \in (0,\lambda_4]$, 
the operator $\mathcal{L}^\lambda - (\mathbb{P}_n\tilde{\mathcal{V}}^\lambda)^*(\mathbb{P}_n\tilde{\mathcal{U}}^\lambda \mathbb{P}_n^*)^{-1} \mathbb{P}_n\tilde{\mathcal{V}}^\lambda$ 
has at least one negative eigenvalue.  
Therefore the whole matrix $\mathcal{M}^\lambda_n$ has at least $n+1$ negative eigenvalues 
for $\lambda \in (0,\lambda_4]$ but has exactly $n$ negative eigenvalues for $\lambda \ge \lambda_5$.   
The lemma follows by continuity of the least eigenvalue of $\mathcal{M}^\lambda_n$, which comes directly from Lemmas \ref{lem-Mprop} and \ref{lem-UV}. 
\end{proof}

We are now ready to construct a nontrivial solution to the matrix equation \eqref{matrix-Maxwell02} 
by passage to the limit as $n \to \infty$.   

\begin{lemma}\label{lem-kerM} 
Assume that $\mathcal{L}^0\not \ge 0$. There exist a number $\lambda_0>0$ and a nonzero vector function 
$(k_0,\vh_0) \in \mathcal{X} \times \tilde{\mathcal{Y}}$ such that 
 \begin{equation}\label{matrix-eqMaxwell}\mathcal{M}^{\lambda_0} \begin{pmatrix} k_0 \\ \vh_0 \end{pmatrix} =0 .\end{equation}
\end{lemma}
\begin{proof} 
By Lemma \ref{lem-kerMn}, for each $n \ge 1$, there exist $\lambda_n \in [\lambda_4,\lambda_5]$ 
and nonzero functions $(k_n , b_n) \in \mathcal{X}\times \RR^n$ such that 
\begin{equation}\label{approx-soln}
\begin{aligned}
\mathcal{L}^{\lambda_n} k_n +  (\tilde{\mathcal{V}}^{\lambda_n})^* \mathbb{P}_n^* b_n  &= 0\\
\mathbb{P}_n\tilde{\mathcal{V}}^{\lambda_n} k_n +  \mathbb{P}_n\tilde{\mathcal{U}}^{\lambda_n} \mathbb{P}_n^* b_n &=0 .
\end{aligned}
\end{equation}
We normalize 
$(k_n ,b_n)$  so that $\|k_n  \|_{L^2} + \|\mathbb{P}_n^* b_n \|_{L^2} =1$.   
Taking the standard inner product in $\RR^n$ of the second equation in \eqref{approx-soln} against $b_n$, 
 and applying estimate \eqref{boundU-infty}, which holds for all $\lambda> 0$,    we obtain 
\begin{equation}\label{Id-Vn}  \langle \tilde{\mathcal{V}}^{\lambda_n} k_n , \mathbb{P}_n^*b_n\rangle_{L^2}  = -\langle  \tilde{\mathcal{U}}^{\lambda_n} \mathbb{P}_n^* b_n, \mathbb{P}_n^* b_n\rangle_{L^2}  \ge  (\lambda^2_n -C_0) \| \mathbb{P}_n^*b_n \|^2_{L^2} + \|\nabla \mathbb{P}_n^*b_n \|^2_{L^2} ,\end{equation}
with the same constant $C_0$ independent of $n$ as in \eqref{boundU-infty}. By definition, the left-hand side is 
\begin{equation}\label{bound-Vn} \begin{aligned} \langle \tilde{\mathcal{V}}^{\lambda_n} k_n , \mathbb{P}_n^*b_n\rangle_{L^2}  &= \langle \tilde{\mathcal{T}}_2^{\lambda_n} k_n, \mathbb{P}_n^*b_n\rangle_{L^2} - \langle  (\mathcal{A}_1^{\lambda_n})^{-1} (\mathcal{B}^{\lambda_n})^*k_n , (\tilde{\mathcal{T}}_1^{\lambda_n})^*\mathbb{P}_n^*b_n\rangle_{L^2} 
\\ &= \langle \tilde{\mathcal{T}}_2^{\lambda_n} k_n, \mathbb{P}_n^*b_n\rangle_{L^2} - \langle  (\mathcal{A}_1^{\lambda_n})^{-1} (\mathcal{B}^{\lambda_n})^* k_n, (\tilde{\mathcal{T}}_1^{\lambda_n} + \lambda_n \nabla)^*\mathbb{P}_n^*b_n\rangle_{L^2}, 
\end{aligned}
\end{equation}
in which the last line is due to the fact that $\nabla \cdot \mathbb{P}_n^*b_n =0$. By Lemmas \ref{op-bounds} and \ref{lem-AS}, the operators $\tilde{\mathcal{T}}_2^{\lambda_n}, (\mathcal{A}_1^{\lambda_n})^{-1} (\mathcal{B}^{\lambda_n})^*, (\tilde{\mathcal{T}}_1^{\lambda_n} + \lambda_n \nabla)^*$ are bounded with their norm independent of $\lambda_n$. Thus, $ \langle \tilde{\mathcal{V}}^{\lambda_n} k_n , \mathbb{P}_n^*b_n\rangle_{L^2}  $ is bounded by $C_0 \|k_n\|_{L^2}\|\mathbb{P}_n^* b_n\|_{L^2}$.  By the normalization and the bounds on $\lb_n$, 
 the estimate \eqref{Id-Vn} yields that $\mathbb{P}_n^* b_n$ is bounded in $H^1(\Omega;\tilde \RR^2 )$.   
 Now we take the inner product 
 of the second equation  in  \eqref{approx-soln} with an arbitrary vector $c_n\in \RR^n$ 
 and rewrite the result as 
$$ \begin{aligned}
-  \langle \Delta \mathbb{P}_n^* b_n ,  \mathbb{P}_n^* c_n \rangle _{L^2}  =  
 \langle  \tilde{\mathcal{V}}^{\lambda_n} k_n + (\tilde{\mathcal{U}}^{\lambda_n} - \Delta) \mathbb{P}_n^* b_n , 
 \mathbb{P}_n^* c_n \rangle _{L^2}.  \end{aligned}$$
The first term $\langle  \tilde{\mathcal{V}}^{\lambda_n} k_n , \mathbb{P}_n^* c_n \rangle $ on the right can be estimated as in \eqref{bound-Vn}. For the second term, we write 
\begin{equation}\label{id-Ubound}\begin{aligned}
\langle (\tilde{\mathcal{U}}^{\lambda_n} - \Delta) \vh,\vg\rangle_{L^2} = \langle (\tilde{\mathcal{S}}^{\lambda_n} - \Delta) \vh,\vg\rangle_{L^2} - \langle (\mathcal{A}_1^{\lambda_n})^{-1} (\tilde{\mathcal{T}}_1^{\lambda_n} + \lambda_n \nabla)^* \vh, (\tilde{\mathcal{T}}_1^{\lambda_n} +\lambda_n \nabla)^*\vg\rangle_{L^2} ,
\end{aligned}\end{equation}
for $\vh , \vg \in \tilde{\mathcal{Y}}$. Again, the operators appearing on the right are uniformly bounded by Lemmas  \ref{op-bounds} and \ref{lem-AS}. Thus, we have 
\begin{equation}\label{bound-DPn} \begin{aligned}
| \langle \Delta \mathbb{P}_n^* b_n ,  \mathbb{P}_n^* c_n \rangle _{L^2} | \le  C_0 (\|k_n\|_{L^2} + \| \mathbb{P}_n^* b_n \|_{L^2}) \|  \mathbb{P}_n^* c_n\|_{L^2}  \end{aligned}\end{equation}
 for some $C_0$ independent of $n$. Additionally, we choose the components of $c_n$ as $c_{nj} = -\sigma_j b_{nj}$ for $j=1, ...,n$, so that 
 $$ \mathbb{P}_n^* c_n  = -\sum_1^n b_{nj} \sigma_j \tilde\psi_j    = \Delta \sum_1^n b_{nj} \tilde\psi_j  =  \Delta  \mathbb{P}_n^* b_n.$$ 
The last two relations  show that $\Delta \mathbb{P}_n^* b_n $ is uniformly bounded in $L^2(\Omega;\tilde \RR^2 )$. Because of the ellipticity of $\Delta$ with its boundary conditions coming from the space $\tilde{\mathcal{Y}}$, we deduce that 
 $\mathbb{P}_n^* b_n$ is bounded in $H^2(\Omega;\tilde \RR^2 )$.

Therefore  the first equation in \eqref{approx-soln} implies the boundedness of $\mathcal{L}^{\lambda_n} k_n$ 
in $L^2(\Omega)$.    Since $ \mathcal{L}^{\lambda_n} -\lambda_n^2 + \Delta $ is bounded from $L^2(\Omega)$ to $L^2(\Omega)$ 
uniformly in $\lb_n$, we deduce that $\Delta k_n$ is bounded in $L^2(\Omega)$.  
But $k_n$ satisfies the Dirichlet boundary condition and therefore $k_n$ is bounded in $H^2(\Omega)$.  
Thus, up to subsequences, we can assume that 
$\lambda_n \to \lambda_0 \in [\lambda_4,\lambda_5]$, $k_n  \to k_0  $ strongly in $H^1(\Omega)$ 
and weakly in $H^2(\Omega)$, 
while $\mathbb{P}_n^*b_n \to \vh_0$ strongly in $H^1(\Omega;\tilde \RR^2 )$ and weakly in $H^2(\Omega; \tilde \RR^2 )$.   
Clearly $(k_0 ,\vh_0)$ is nonzero since by the strong convergence, $\|k_0 \|_{L^2} + \|\vh_0\|_{L^2} =1$.


We will prove that the triple $(\lambda_0,k_0 ,\vh_0)$ solves \eqref{matrix-eqMaxwell} 
and that $(k_0 ,\vh_0)$ indeed belongs to $\mathcal{X}\times \tilde{\mathcal{Y}}$.   
We will check that \eqref{matrix-eqMaxwell} is valid in the distributional sense. 
In order to do so, take any $g \in \mathcal{X}$.    Then  
$$ \begin{aligned}
|\langle (\mathcal{L}^{\lambda_n} k_n - \mathcal{L}^{\lambda_0} k_0), g \rangle_{L^2}| 
&\le |\langle \mathcal{L}^{\lambda_n} (k_n - k_0) , g \rangle_{L^2}|  
+ |\langle (\mathcal{L}^{\lambda_n} - \mathcal{L}^{\lambda_0}) k_0  , g \rangle_{L^2}|    
\\
&\le  \| k_n - k_0 \|_{L^2} \|\mathcal{L}^{\lambda_n} g\|_{L^2}  
+ \|(\mathcal{L}^{\lambda_n} - \mathcal{L}^{\lambda_0})\|_{L^2\to L^2} \|k_0 \|_{L^2} \|  g\|_{L^2} ,   
\end{aligned}$$
which converges to zero as $n\to \infty$ by the facts (see Lemma \ref{lem-UV} {\em (i)}) 
that $k_n \to k_0 $ strongly in $L^2(\Omega)$ and $\mathcal{L}^{\lambda_n}  \to \mathcal{L}^{\lambda_0}$ in the operator norm.  
 Similarly, we have 
$$\begin{aligned} 
|\langle ( \tilde{\mathcal{V}}^{\lambda_n})^* \mathbb{P}_n^* b_n- (\tilde{\mathcal{V}}^{\lambda_0})^* \vh_0, g \rangle_{L^2}| 
&\le |\langle ( \tilde{\mathcal{V}}^{\lambda_n})^* ( \mathbb{P}_n^* b_n- \vh_0) , g \rangle_{L^2}|  + |\langle ( \tilde{\mathcal{V}}^{\lambda_n} - \tilde{\mathcal{V}}^{\lambda_0})^* \vh_0, g \rangle_{L^2}|  
\\
&\le \|\mathbb{P}_n^* b_n- \vh_0 \|_{L^2} \| \tilde{\mathcal{V}}^{\lambda_n} g\|_{L^2} + \|( \tilde{\mathcal{V}}^{\lambda_n} - \tilde{\mathcal{V}}^{\lambda_0})^* \vh_0 \|_{L^2} \| g\|_{L^2} ,   
\end{aligned} $$   
which again converges to zero as $n\to \infty$.  {\it This implies the first equation} in \eqref{matrix-eqMaxwell}. 

Next, for any  $\vg \in \tilde{\mathcal{Y}}$, since $\{\tilde\psi_j\}_{j=1}^\infty$ is a basis in $\tilde{\mathcal{Y}}$ 
which is orthonormal with respect to the usual $L^2$ norm, 
there exists a sequence   $c_n  = (c_{n1},c_{n2},\cdots, c_{nn}) \in \RR^n$ 
such that $\vg_n = \mathbb{P}_n^* c_n= \sum_{j=1}^n c_{nj} \psi_j   \to \vg$ strongly in $\tilde{\mathcal{Y}}$.   
Now we write (using the $\RR^n$ and $L^2$ inner products)  
$$\begin{aligned}
\langle \mathbb{P}_n \tilde{\mathcal{U}}^{\lambda_n} \mathbb{P}_n^* b_n, c_n \rangle - \langle \tilde{\mathcal{U}}^{\lambda_0} \vh_0, \vg \rangle 
& = \langle \tilde{\mathcal{U}}^{\lambda_n} \mathbb{P}_n^* b_n, \vg_n  \rangle - \langle \tilde{\mathcal{U}}^{\lambda_0} \vh_0, \vg \rangle 
\\ & = \langle \tilde{\mathcal{U}}^{\lambda_n} \mathbb{P}_n^* b_n, \vg_n - \vg \rangle 
+  \langle  (\mathbb{P}_n^* b_n - \vh_0),  (\tilde{\mathcal{U}}^{\lambda_n})^* \vg \rangle  
+ \langle  ( \tilde{\mathcal{U}}^{\lambda_n}  - \tilde{\mathcal{U}}^{\lambda_0}) \vh_0, \vg\rangle . 
\end{aligned}$$
The last term 
converges to zero as $n\to \infty$ thanks to Lemma \ref{lem-UV} {\em (ii)}. 
The next-to-last term goes to zero because $\mathbb{P}_n^* b_n \to \vh_0$ strongly in $L^2$ and 
$ (\tilde{\mathcal{U}}^{\lambda_n})^* \vg \in L^2 $.  
We write the first term on the right as  
$$\langle \tilde{\mathcal{U}}^{\lambda_n} \mathbb{P}_n^* b_n, \vg_n - \vg \rangle  
=  \langle (\tilde{\mathcal{U}}^{\lambda_n } - \Delta) \mathbb{P}_n^* b_n, \vg_n - \vg \rangle  
+  \langle \mathbb{P}_n^* b_n, \Delta(\vg_n - \vg )\rangle,$$
and use the expression \eqref{id-Ubound} for the first term on the right. We conclude that $\langle \tilde{\mathcal{U}}^{\lambda_n} \mathbb{P}_n^* b_n, \vg_n - \vg \rangle  $ also tends to zero because  $\|\vg_n - \vg\|_{L^2}+ \|\Delta(\vg_n - \vg )\|_{L^2} \to 0$, as $n \to \infty$.

Finally, by writing 
$$\begin{aligned}
\langle \mathbb{P}_n \tilde{\mathcal{V}}^{\lambda_n} k_n, c_n \rangle - \langle \tilde{\mathcal{V}}^{\lambda_0} k_0, \vg \rangle 
& = \langle \tilde{\mathcal{V}}^{\lambda_n} k_n , \vg_n  \rangle - \langle \tilde{\mathcal{V}}^{\lambda_0} k_0, \vg \rangle 
\\ & = \langle \tilde{\mathcal{V}}^{\lambda_n} k_n, \vg_n - \vg \rangle +  \langle \tilde{\mathcal{V}}^{\lambda_n} (k_n - k_0), \vg \rangle  + \langle  ( \tilde{\mathcal{V}}^{\lambda_n}  - \tilde{\mathcal{V}}^{\lambda_0}) k_0, \vg\rangle , 
\end{aligned}$$
we easily obtain the convergence of $\langle \mathbb{P}_n \tilde{\mathcal{V}}^{\lambda_n} k_n, c_n \rangle $ to  $\langle \tilde{\mathcal{V}}^{\lambda_0} k_0, \vg \rangle$. 

Putting all the limits together, we have found that the triple $(\lambda_0,k_0,\vh_0)$ solves the matrix equation \eqref{matrix-eqMaxwell} in the distributional sense, with $k_0 \in H^2(\Omega)$ and $\vh_0 \in H^2(\Omega;\tilde \RR^2 )$. 
Because $\mathbb{P}_n^*b_n \to \vh_0$ strongly in $H^{3/2}(\partial\Omega;\tilde \RR^2 )$, it follows that 
$\vh_0$ satisfies the boundary conditions $0 = e_\theta\cdot\vh  =  \nabla_x\cdot((e_r\cdot\vh)e_r)$ on $\partial\Omega$.  
It also follows 
that $(k_0,\vh_0)$ indeed belongs to $\mathcal{X}\times \tilde{\mathcal{Y}}$.   
\end{proof}

\subsection{Existence of a growing mode}
We are now ready to construct a growing mode of the linearized Vlasov-Maxwell systems \eqref{lin-VM} and \eqref{Maxwell-system} with the boundary conditions.  
\begin{lemma}\label{lem-growingmode}  Assume that $\mathcal{L}^0\not \ge 0$. There exists a growing mode $(e^{\lambda_0 t}f^\pm,e^{\lambda_0 t}\vE,e^{\lambda_0 t}\vB)$ of the linearized Vlasov-Maxwell system, for some $\lambda_0>0$, $f^\pm \in \cH^\pm$  
and  $\vE,\vB \in H^1(\Omega; \RR^3)$.  
\end{lemma}

\begin{proof} Let $(\lambda_0,k_0,\vh_0)$ be the triple constructed in Lemma \ref{lem-kerM}. We then define the electric and magnetic potentials: 
$$\phi:= - (\mathcal{A}_{1}^{\lambda_0})^{-1} \Big[(\mathcal{B}^{\lambda_0})^* k_0 + (\tilde{\mathcal{T}}_{1}^{\lambda_0})^* \vh_0\Big], \qquad \vA := \vh_0 + k_0 e_\varphi,$$
and the particle distribution:
\begin{equation}\label{constr-f}f^\pm(x,v)  := \pm \mu^\pm_e (1-\mathcal{Q}^\pm_{\lambda_0})\phi  
 \pm  (a+r\cos \theta)  \mu_p^\pm   A_\varphi   \pm \mu^\pm_e \mathcal{Q}^\pm_{\lambda_0}(\hat v\cdot \vA). 
\end{equation}
We also define electromagnetic fields by $\vE := -\nabla \phi - \lambda_0 \vA$ and  $\vB := \nabla \times \vA $. It is then clear by our construction that $(e^{\lambda_0 t} \phi, e^{\lambda_0 t}\vA)$ solves the potential form of the Maxwell equations \eqref{Maxwell-pots}, and so $(e^{\lambda_0 t} \vE, e^{\lambda_0 t}\vB)$ solves the linearized Maxwell system. In addition, since $\phi,A_\varphi \in \mathcal{X} $ and $\tilde\vA \in \tilde{\mathcal{Y}}$, it follows that $\vE$ and $\vB$ belong to $H^1(\Omega; \RR^3)$ and satisfy the specular boundary conditions. The specular boundary condition for $f^\pm$ also follows directly by definition \eqref{def-f} and the fact that $\mathcal{Q}^\pm_\lambda(g)$ is specular on the boundary if $g$ is.  

It therefore remains to check the Vlasov equations for $e^{\lambda_0 t}f^\pm$. Let us verify the equation for $f^+$; the 
verification for $f^-$ is  similar. For sake of brevity, let us denote $g^+: = f^+ - \mu_e^+ \phi - (a + r \cos\theta) \mu_p^+ A_\varphi$ and $h^+: = \mu_e^+ (\hat v \cdot \vA - \phi)$. The identity \eqref{constr-f} for $f^+$ simply reads $$ g^+ = \mathcal{Q}^+_{\lambda_0} h^+. $$
We shall verify the Vlasov equation for $f^+$, which can be restated as  
\begin{equation}\label{Vlasov-g} (\lambda_0 + \oD^+) g^+ = \lambda_0 h^+ \end{equation}
 in the distributional sense. For each $v = v_re_r + v_\theta e_\theta + v_\varphi e_\varphi$, we write $\mathcal{R} v  := - v_r e_r - v_\theta e_\theta + v_\varphi e_\varphi$ and $(\mathcal{R} g)(x,v) := g(x,\mathcal{R} v)$ 
 as in  \eqref{adjointP}.    Then $\mathcal{R}^2 = Id$, $\mathcal{R}\oD^+ \mathcal{R} = -\oD^+$ and, in view of \eqref{adjointP}, the adjoint of $\mathcal{Q}^+_{\lambda_0}$ in the space $\cH$ is identical to $\mathcal{R} \mathcal{Q}_{\lambda_0}^+ \mathcal{R}$.  
 Thus for each test function $k = k(x,v)\in C_c^1(\Omega\times \RR^3)$ that has toroidal symmetry and satisfies the specular condition,  we have 
  $$\begin{aligned} 
 \langle (\lambda_0 &+ \oD^+) g^+ , k\rangle_{\cH}  =  \langle g^+ , (\lambda_0 - \oD^+) k\rangle_{\cH}  
 =  \langle \mathcal{Q}^+_{\lambda_0} h^+ , (\lambda_0 - \oD^+) k\rangle_{\cH} 
 \\
& =  \langle \mathcal{R} h^+ , \mathcal{Q}^+_{\lambda_0} \mathcal{R} ({\lambda_0} - \oD^+) k\rangle_{\cH}
  =  \langle \mathcal{R} h^+ , \mathcal{Q}^+_{\lambda_0}  ({\lambda_0} + \oD^+) \mathcal{R} k\rangle_{\cH}
 =  \langle \mathcal{R} h^+ , {\lambda_0}  \mathcal{R} k\rangle_{\cH}  
=  \langle {\lambda_0}  h^+ , k\rangle_{\cH}, 
\end{aligned}$$
where we have used the fact that $\mathcal{Q}^+_{\lambda_0}({\lambda_0} + \oD^+) k = {\lambda_0} k$, which follows directly from the definition of $\mathcal{Q}^+_{\lambda_0}$.   This proves the identity \eqref{Vlasov-g}.  
That is, the Vlasov equation for $f^+$ is now verified.    
\end{proof}

\section{Examples} \label{sec-examples}
The purpose of this section is to exhibit some explicit examples of stable and unstable equilibria, that is to find examples so that $\mathcal{L}^0 \ge 0$ (which implies the stability) or $\mathcal{L}^0\not \ge 0$ (instability). We recall that 
\begin{equation}\label{def-L0}
\mathcal{L}^0 = \mathcal{A}_2^0 -  \mathcal{B}^0 (\mathcal{A}_1^0)^{-1} (\mathcal{B}^0)^*  \end{equation}
and its domain is $\mathcal{X}$. The operators $\mathcal{A}^0_j$ and $\mathcal{B}^0$ are defined as in \eqref{operators-0}. For each $h \in \mathcal{X}$, we have
\begin{equation}\label{cond-L0}
\langle \mathcal{L}^0 h,h\rangle_{L^2} = \langle \mathcal{A}_2^0 h,h\rangle_{L^2} - \langle (\mathcal{A}_1^0)^{-1}(\mathcal{B}^0)^* h,(\mathcal{B}^0)^* h \rangle_{L^2} .
\end{equation}
Since $-\mathcal{A}^0_1$ is positive definite with respect to the $L^2$ norm, the second term is nonnegative. In order to investigate the sign of the first term, we recall that 
\begin{equation}\label{def-A2}\mathcal{A}_2^0 h = - \Delta h + \frac{1}{(a+r\cos\theta)^2} h  - \sum_\pm  \int_{\RR^3}\hat v_\varphi \Big [ (a+r\cos\theta) \mu_p^\pm   h  
 +  \mu_e^\pm \mathcal{P}^\pm(\hat v_\varphi h)  \Big] \; \mathrm{d}v , \end{equation}
and thus by taking integration by parts (see \eqref{product-A2}) and using $h=0$ on $\partial \Omega$, we have 
\begin{equation}\label{prod-hAh}\begin{aligned} 
\langle \mathcal{A}_2^0 h,h\rangle_{L^2} & = \int_\Omega   \Big( |\nabla h|^2 
+ \frac 1 {(a+r\cos\theta)^2} | h|^2 \Big)\; \mathrm{d}x
 - \sum_\pm \int _\Omega   \int_{\RR^3}(a+r\cos\theta)\mu^\pm_p \hat v_\varphi|  h|^2\; \mathrm{d}v\mathrm{d}x  \\&\quad+\sum_\pm \|\mathcal{P}^\pm(\hat v_\varphi h)\|^2_{\cH}
.\end{aligned}\end{equation}
We observe that only the second term on the right of \eqref{prod-hAh} does not have a definite sign. For the rest of this section, we will provide examples so that this term dominates the other two terms, which are always nonnegative. 

\subsection{Stable equilibria}
We begin with some simple examples of stable equilibria.  


\begin{theorem}\label{lem-stab} 
Let $(\mu^\pm,\phi^0,A_\varphi^0)$ be an inhomogenous equilibrium.  

(i) If 
\begin{equation}\label{suff-stab}
p \mu^\pm_p(e,p)\le 0, \qquad \quad \forall ~ e,p, 
 \end{equation} 
then the equilibrium is spectrally stable provided that $A_\varphi^0$ is sufficiently small in $L^\infty(\Omega)$. 

(ii) If 
\begin{equation}\label {suff-stab2}
|\mu_p^\pm (e,p)|  \le  \frac \epsilon {1+|e|^\gamma} , \end{equation}
for some $\gamma>3$, with $\epsilon$ sufficiently small but $A_\varphi^0$  not necessarily small, 
then the equilibrium is spectrally stable.  

\end{theorem}
\begin{proof} It suffices to show that $\mathcal{A}_2^0 \ge 0$. 
Let us look at the second integral  of $ \langle \mathcal{A}_2^0 h,h\rangle_{L^2}$ in \eqref{prod-hAh}, for each $h \in \mathcal{X}$. 
By the definition of $p^\pm = (a+r\cos \theta)(v_\varphi \pm  A_\varphi^0(r,\theta))$, we may write  
$$\begin{aligned}
 \int _\Omega   \int_{\RR^3}(a+r\cos\theta)\mu^\pm_p \hat v_\varphi|  h|^2\; \mathrm{d}v\mathrm{d}x  =  \int _\Omega   \int_{\RR^3}\frac{p^\pm  \mu^\pm_p}{ \langle v \rangle} |  h|^2\; \mathrm{d}v\mathrm{d}x \mp  \int _\Omega   \int_{\RR^3}\frac{(a+r\cos\theta) A^0_\varphi \mu^\pm_p}{ \langle v \rangle} | h|^2\; \mathrm{d}v\mathrm{d}x .
\end{aligned}$$
Let us consider case {\it (i)}. Since $p\mu^\pm_p \le 0$, the above yields 
$$ - \sum_\pm \int _\Omega   \int_{\RR^3}(a+r\cos\theta)\mu^\pm_p \hat v_\varphi|  h|^2\; \mathrm{d}v\mathrm{d}x  \ge 
  - (1+a) \sup_{x } |A^0_\varphi| \Big( \sup_x \int_{\RR^3}\frac{ |\mu_p^+|+ |\mu_p^-| }{\langle v \rangle} \; \mathrm{d}v \Big)\int_{\Omega} |h|^2 \; \mathrm{d}x 
. $$
Now by the Poincar\'e inequality, we have $\|h\|_{L^2} \le c_0 \|\nabla h \|_{L^2}$ for $h \in \mathcal{X}$ and for some fixed constant $c_0$. In addition, thanks to the decay assumption \eqref{mu-cond}, the supremum over $x\in \Omega$ of 
$\int_{\RR^3}\langle v \rangle^{-1}(|\mu_p^+|+ |\mu_p^-| ) \; \mathrm{d}v$ is finite. 
Thus if the sup norm of $A_\varphi^0$ is sufficiently small, or more precisely if $A_\varphi^0$ satisfies
\begin{equation}\label{sup-psibound} c_0 (1+a)\sup _x |A^0_\varphi|\Big( \sup_x \int_{\RR^3}\langle v \rangle^{-1}(|\mu_p^+|+ |\mu_p^-| ) \; \mathrm{d}v \Big) \le 1,\end{equation}
then the second term in $\langle \mathcal{A}_2^0 h,h \rangle_{L^2}$ is smaller than the first, and so the operator $\mathcal{A}_2^0$ is nonnegative. 

Let us consider case {\it (ii)}. As above, we only have to bound the second term in 
$ \langle \mathcal{A}_2^0h,h\rangle_{L^2}$. The assumption \eqref{suff-stab2} yields
$$
\Big | \int _\Omega   \int_{\RR^3}(a+r\cos\theta)\mu^\pm_p \hat v_\varphi|  h|^2\; \mathrm{d}v\mathrm{d}x   \Big|  
 \le \epsilon\Big(\sup_{x\in \Omega}\int_{\RR^3}\frac{(a+1)}{1+|e|^\gamma} \mathrm{d}v\Big) \|h\|_{L^2}^2 \le  C\epsilon \|h \|_{L^2}^2. $$
If $\epsilon$ is sufficiently small, the second term is smaller than the positive terms.  
\end{proof}

\subsection{Unstable equilibria}\label{sec-unstab-ex}
Let us now turn to some examples of unstable equilibria. It certainly suffices to find a single function $h\in \mathcal{X}$ such that 
$\langle \mathcal{L}^0 h,h\rangle_{L^2} <0$, or specifically to show that the second term in $\langle\mathcal{A}^0_2 h,h\rangle_{L^2}$ dominates the remaining terms. 
For the rest of this section, we limit ourselves to a purely magnetic equilibrium $(\mu^\pm,0,A^0_\varphi)$ 
with $\phi=0$. Furthermore we assume that 
\begin{equation}\label{simplified-cond} 
\mu^+(e ,p) = \mu^-(e ,-p),\qquad \forall e,p.\end{equation}
This assumption holds for example if $\mu^+ = \mu^- = \mu$ and $\mu$ is an even function of $p$. As will be seen below, the assumption greatly simplifies the verification of the spectral condition on $\mathcal{L}^0$. 

Let us recall 
$$e = \langle v \rangle , \qquad p^\pm = (a+r\cos \theta) (v_\varphi \pm A^0_\varphi).$$
We begin with some useful properties of the projection $\mathcal{P}^\pm$. 

\begin{lemma}\label{lem-projP} There hold

(i) For $k\in \ker\oD^\pm$ and $h\in \cH$ so that $kh \in \cH$, we have $ \mathcal{P}^\pm (k h) = k \mathcal{P}^\pm h.$

(ii) Assume \eqref{simplified-cond}. Let  $\mathcal{R}_\varphi v$ denote the reflected point of $v$ across the hyperplane $\{e_r,e_\theta\}$ in $\RR^3$, and define $\mathcal{R}_\varphi g(x,v) = g(x,\mathcal{R}_\varphi v)$. For each function $g  \in \cH$, we have $$\mathcal{R}_\varphi\mathcal{P}^+ (\mathcal{R}_\varphi g) = \mathcal{P}^-g.$$
In particular, $\mathcal{P}^+ h = \mathcal{P}^-h$ for $h = h(r,\theta)$. 

\end{lemma}

\begin{proof} We note that $k \mathcal{P}^\pm h \in \ker \oD^\pm$ since both $k$ and $\mathcal{P}^\pm h$ belong to $\ker\oD^\pm$. Now, for all $m \in \ker\oD^\pm$, we have 
$$ \langle \mathcal{P}^\pm (kh), m\rangle _{\cH} = \langle  kh,  \mathcal{P}^\pm m\rangle _{\cH} = \langle kh, m\rangle _{\cH} =  \langle  \mathcal{P}^\pm h, km\rangle _{\cH} = \langle  k \mathcal{P}^\pm h, m\rangle _{\cH} .$$
By taking $m = \mathcal{P}^\pm (k h) - k \mathcal{P}^\pm h$, we obtain the identity in {\em (i)}. 

Next, let us prove {\em (ii)}. In view of the assumption \eqref{simplified-cond}, we have
\begin{equation}\label{cal-mu} \mathcal{R}_\varphi \mu^+(e,p^+) = \mu^+(e,-p^-) = \mu^- (e,p^-), \qquad \mathcal{R}_\varphi \mu^-(e,p^-) = \mu^-(e,-p^+) = \mu^+ (e,p^+).\end{equation}
 In addition, from the definition of $\oD^\pm$ in \eqref{def-opD}, we observe that $ \mathcal{R}_\varphi \oD^+ \mathcal{R}_\varphi = \oD^-$. That is, the differential operator $\oD^-$acting on $g$ is the same as the operator $\mathcal{R}_\varphi \oD^+$ acting on $\mathcal{R}_\varphi g$. This together with \eqref{cal-mu} proves {\em (ii)}.
\end{proof}

\begin{lemma}\label{lem-simplified} 
If $(\mu^\pm,0,A^0_\varphi)$ is an equilibrium such that $\mu^\pm$ satisfies \eqref{simplified-cond}, then $\mathcal{B}^0 =0$ and so for all $h \in \mathcal{X}$ 
\begin{equation}\label{def-newA2}
\mathcal{L}^0h = \mathcal{A}_2^0 h = - \Delta h + \frac{1}{(a+r\cos\theta)^2} h  - 2 \int_{\RR^3}\hat v_\varphi \Big [ (a+r\cos\theta) \mu_p^-   h  
 +  \mu_e^- \mathcal{P}^-(\hat v_\varphi h)  \Big] \; \mathrm{d}v .\end{equation}
\end{lemma} 

\begin{proof} By definition, we may write 
$$\mathcal{B}^0 h  = - \int_{\RR^3}\hat v_\varphi k(x,v) \; \mathrm{d}v , \qquad  k(x,v): = \mu_e^+ (e,p^+) ( 1 - \mathcal{P}^+) h +  \mu_e^- (e,p^-) ( 1 - \mathcal{P}^-) h .$$
We will show that $k(x,v)$ is in fact even in $v_\varphi$, and thus $\mathcal{B}^0$ must vanish by integration. Indeed, by \eqref{cal-mu} and Lemma \ref{lem-projP}, {\em (ii)}, we have 
$$\begin{aligned}k(x,\mathcal{R}_\varphi v) &=  \mathcal{R}_\varphi\mu_e^+ (e,p^+  ) ( 1 - \mathcal{P}^+) h + \mathcal{R}_\varphi  \mu_e^- (e, p^-) ( 1 - \mathcal{P}^-) h
\\
 &=  \mu_e^- (e,p^-) ( 1 - \mathcal{P}^-) h + \mu_e^+ (e,p^+) ( 1 - \mathcal{P}^+) h
 \\&= k(x,v).
 \end{aligned}$$  
This proves the first identity in \eqref{def-newA2}. For the second identity, we perform the change of variable 
 $v \to \mathcal{R}_\varphi v$ in the integral terms of $\mathcal{A}^0_2$ in \eqref{def-A2}. We get 
$$ \begin{aligned}
\int_{\RR^3}\hat v_\varphi  \mu_p^+ (e,p^+)   \; \mathrm{d}v
& = -  \int_{\RR^3} \hat v_\varphi  \mu_p^+ (e,\mathcal{R}_\varphi p^+)    \; \mathrm{d}v = \int_{\RR^3} \hat v_\varphi  \mu_p^- (e,p^-)     \; \mathrm{d}v
\\
\int_{\RR^3}\hat v_\varphi  \mu_e^+(e,p^+) \mathcal{P}^+(\hat v_\varphi h)  \; \mathrm{d}v  &= -\int_{\RR^3}\hat v_\varphi  \mu_e^+(e,\mathcal{R}_\varphi p^+) \mathcal{R}_\varphi\mathcal{P}^+(\hat v_\varphi h)  \; \mathrm{d}v   =\int_{\RR^3}\hat v_\varphi  \mu_e^-(e,p^-) \mathcal{P}^-(\hat v_\varphi h)  \; \mathrm{d}v . 
 \end{aligned}$$
This proves \eqref{def-newA2} and  completes the proof of the lemma. 
\end{proof}

Thanks to Lemma \ref{lem-simplified}, the problem now depends only on the $-$ particles (electrons), and thus we shall drop the minus superscript in $p^-, \mu^-, \oD^-$, and $\mathcal{P}^-$ for the rest of this section. Integrating by parts, we have 
$$
( \mathcal{L}^0 h,h)_{L^2} = \int_\Omega   \Big( |\nabla h|^2 
+ \frac 1 {(a+r\cos\theta)^2} | h|^2 \Big)\; \mathrm{d}x
 - 2\int _\Omega   \int_{\RR^3}(a+r\cos\theta)\mu_p \hat v_\varphi|  h|^2\; \mathrm{d}v\mathrm{d}x +2\|\mathcal{P}(\hat v_\varphi h)\|^2_{\cH}
$$
for any $h \in \mathcal{X}$. Let us take $h = h_*(r,\theta)$ to be a function in $\mathcal{X}$ such that the first term in the above calculation is identical to one. Such a function $h_*$ exists in $\mathcal{X}$; for example, a normalized toroidal eigenfunction that is associated with the least eigenvalue of $-\Delta$ with the Dirichlet boundary condition will do. Thus, recalling that $e = \langle v \rangle $ and $ p = (a+r\cos\theta) (v_\varphi - A^0_\varphi)$, we write $( \mathcal{L}^0 h_*,h_*)_{L^2}$ as
\begin{equation}\label{key-prodL}
\begin{aligned}
( \mathcal{L}^0h_*,h _*)_{L^2} &= 1 
 - 2\int _\Omega   \int_{\RR^3}\frac{p \mu_p}{e} | h_*|^2\; \mathrm{d}v\mathrm{d}x  - 2 \int_\Omega \int_{\RR^3} \frac{\mu_p}{e}  (a+r\cos\theta)A^0_\varphi   | h_*|^2 \; \mathrm{d}v\mathrm{d}x +2 \|\mathcal{P} (\hat v_\varphi h_*) \|^2_{\cH}
 \\
 &= 1 + I + II + III.
 \end{aligned}
\end{equation}

We now scale in the variable $p$ to get the following result.  
\begin{theorem}\label{lem-unstab-inhomo} Let $\mu^\pm$ satisfy \eqref{simplified-cond} and let $\mu = \mu^-$. Assume that
\begin{equation}\label{mg-unstab-in}
p \mu_p(e,p)\ge c_0 p^2   \nu(e), \qquad \quad \forall ~ e,p,
 \end{equation} 
for some positive constant $c_0$ and some nonnegative function $\nu(e)$ such that $\nu \not \equiv 0$. For each $K>0$, 
define $\mu^{(K),\pm}(e,p): = \mu^\pm(e, K p)$.  
Assume that $A_\varphi^{(K),0}$ is a bounded solution of the equation 
 \begin{equation}\label{equilibrium-psi} \begin{aligned}
\Big(- \Delta+\frac{1}{(a+r\cos\theta)^2}\Big) A_\varphi^{(K),0} &=  \int_{\RR^3} \hat v_\theta\Big[\mu^{(K),+}(e,p^{(K),+}) - \mu^{(K),-}(e,p^{(K),-}) \Big]\;\mathrm{d}v,
\end{aligned}
\end{equation}
with $p^{(K),\pm} = (a+r\cos\theta) (v_\varphi \pm A_\varphi^{(K),0})$ and with $A_\varphi^{(K),0} =0$ on the boundary $\D\Omega$. 
Then there exists a positive number $K_0$ such that the purely magnetic equilibria $(\mu^{(K),\pm},0,A_\varphi^{(K),0})$ are spectrally unstable for all $K\ge K_0$.
\end{theorem}

\begin{proof} It suffices to show that $ \langle \mathcal{L}^0 h_*,h_*\rangle_{L^2}<0$.  Let us give bounds on $I,II,III$ defined as above in $ \langle \mathcal{L}^0 h_*,h_*\rangle_{L^2}$. By a view of the assumption \eqref{mg-unstab-in} and the fact that $\nu(e)$ is even in $v_\varphi$, we have
$$
\begin{aligned}  I &= 
 -  2\int _\Omega  \Big(\int_{\RR^3}e^{-1} Kp  \mu_p (e, Kp) \; \mathrm{d}v \Big)  |h_*|^2\; \mathrm{d}x \\& \le -  2c_0 K^2 \int _\Omega \Big(\int_{\RR^3}e^{-1}(a+r\cos\theta)^2 (v_\varphi - A_\varphi^{(K),0})^2   \nu(e) \; \mathrm{d}v \Big)  |h_*|^2\; \mathrm{d}x
 \\
 & \le -  2c_0 K^2 \int _\Omega \Big(\int_{\RR^3}e^{-1}v^2_\varphi  \nu(e) \; \mathrm{d}v \Big)   (a+r\cos\theta)^2 |h_*|^2\; \mathrm{d}x 
 \\&\le  -c_1 K^2  \| (a+r\cos\theta) h_*\|^2_{L^2(\Omega)} ,   
 \end{aligned}$$
 where $c_1>0$ is independent of $K$.  Next, by the decay assumption \eqref{mu-cond} on $\mu_p$, we obtain 
$$\begin{aligned} 
II \quad &\le\quad  C_0K\|A_\varphi^{(K),0}\|_{L^\infty} 
 \sup_{r\in [0,1]} \Big( \int_{\RR^3} e^{-1}|\mu_p(e,Kp)| \; \mathrm{d}v\Big)
 \\&  \le\quad C_0 C_\mu K\|A_\varphi^{(K),0}\|_{L^\infty}  \int_{\RR^3}\frac{1}{\langle v \rangle(1 + \langle v \rangle^\gamma)} \; \mathrm{d}v \le C_0 C_\mu K\|A_\varphi^{(K),0}\|_{L^\infty} ,
\end{aligned}$$ with $C_0 = 2(1+a)\|h_*\|_{L^2}^2$. Similarly, $$\begin{aligned}
III \quad &\le \quad C_0
 \sup_{r\in [0,1]} \Big( \int_{\RR^3} |\mu_e(e,Kp)| \; \mathrm{d}v\Big)
\\
&\le \quad  C_0 C_\mu\int_{\RR^3}\frac{1}{1 + \langle v \rangle^\gamma} \; \mathrm{d}v \le C_0 C_\mu,
\end{aligned}$$
with $\gamma>2$ and for some constant $C_\mu$ independent of $K$. 

Combining these estimates, we have therefore obtained
$$
\begin{aligned}
 \langle \mathcal{L}^0 h_*,h_*\rangle_{L^2}  &\le 1 -  c_1 K^2   \| (a+r\cos\theta) h_*\|^2_{L^2(\Omega)}
 + C_0C_\mu ( 1+ K\|A_\varphi^{(K),0}\|_{L^\infty}).
 \end{aligned}
 $$
The $L^2$ norm of $(a+r\cos\theta) h_*$ is clearly nonzero.  We claim that $A_\varphi^{(K),0}$ is uniformly bounded 
 independently of  $K$. 
  Indeed, recalling that  $A_\varphi^{(K),0}$ satisfies the elliptic equation \eqref{equilibrium-psi} and using the decay assumption \eqref{mu-cond} on $\mu^\pm$, we have 
$$ \Big| \Big(-\Delta+\frac 1{(a+r\cos\theta)^2}\Big) A_\varphi^{(K),0} \Big |  \le C_\mu\int_{\RR^3}\frac{1}{1+\langle v \rangle ^{\gamma}} \; \mathrm{d}v \le C_\mu  ,$$
for some constant $C_\mu$  independent of $K$, for some $\gamma>3$. By the standard maximum principle for the elliptic operator, $A_\varphi^{(K),0}$ is bounded uniformly in $K$ since $C_\mu$ is independent of $K$. This proves the claim.  Summarizing, we conclude that  $ \langle \mathcal{L}^0 h_*,h_*\rangle_{L^2} $ is dominated 
for large $K$ by $I$ and it is therefore strictly negative. 
\end{proof}

Additionally, we have the following result for homogenous equilibria, meaning that $\vE^0 = \vB^0 =0$. 
\begin{theorem} \label{lem-mg-unstab} 
Let $\mu^\pm = \mu^\pm(e,p)$ be an homogenous equilibrium satisfying \eqref{simplified-cond}, and let $\mu = \mu^-$. Assume that 
\begin{equation}\label{mg-unstab}
p \mu_p(e,p) + e\mu_e(e,p) > 0, \qquad \quad \forall ~ e,p.
 \end{equation} 
Then there exists a positive number $K_0$ such that the rescaled homogenous equilibria $\mu^{(K),\pm}(e,p): = K \mu^\pm(e, p)$ are spectrally unstable, for all $K\ge K_0$. 
\end{theorem}

\begin{proof} In the homogenous case $A^0_\varphi =0$, \eqref{key-prodL} becomes
$$\begin{aligned}
( \mathcal{L}^0h_*,h _*)_{L^2} &= 1 
 - 2 K \int _\Omega   \int_{\RR^3}\frac{p \mu_p}{e} | h_*|^2\; \mathrm{d}v\mathrm{d}x  +2 K \|\mathcal{P} (\hat v_\varphi h_*) \|^2_{\cH}
\\
&\le 1 
 - 2 K \int _\Omega   \int_{\RR^3}\Big[ \frac{p \mu_p}{e} + \mu_e\Big] | h_*|^2\; \mathrm{d}v\mathrm{d}x
  \end{aligned}$$
The integral is clearly positive thanks to the assumption \eqref{mg-unstab}. Thus, $( \mathcal{L}^0\psi _*,\psi_*)_{L^2}$ is strictly negative for large $K$.  
\end{proof}

Because the projections $\mathcal{P}^\pm$ play such a prominent role in our analysis, 
we present an explicit calculation of them, at least in the homogeneous case for which  $e=\langle v \rangle$ and 
$p=(a+r\cos \theta)v_\varphi$.   Let $\mathbb{D}$ be the unit disk in the plane  
and let $\Theta$ be the usual change of variables from cartesian coordinates $y=(y_1,y_2)$ on the disk 
to polar coordinates $(r,\theta)$. 

\begin{lemma}\label{lem-Ph} 
Assume that $\vE^0 = \vB^0 = 0$. Let $h = h(r,\theta)\in L_\tau^\infty(\Omega)$. Then $$\mathcal{P}^\pm h  =  g(\langle v\rangle, (a+r\cos\theta)v_\varphi), $$
where  $g(e,p)$ is the average value of $h \circ \Theta$ on $\mathbb{S}_{e,p}$ and the set $\mathbb{S}_{e,p}$ is the intersection of the disk $\mathbb{D}$ and the half-plane $\{ y_1 > |p|/\sqrt{e^2-1} - a \}$. 
\end{lemma}

		\begin{proof} 
We note that the kernel of $\oD^\pm$ contains all functions of $e$ and $p$, and in particular, for each $h\in \cH$, $ \mathcal{P}^\pm h$ is a function of $e$ and $p$. Now for $h  = h(r,\theta)\in L^\infty(\Omega)$ and for an arbitrary bounded function $\xi = \xi(e,p)$, it follows from the orthogonality of $\mathcal{P}^\pm$ and $1-\mathcal{P}^\pm$ that 
\begin{equation}
\label{cal-Ph01}\begin{aligned}
 0 = \langle (1-\mathcal{P}^\pm)h , \xi \rangle_{\cH}  &= 
 2\pi \int_{\RR} \int_{\tilde \RR^2} \int_0^{2\pi} \int_0^1 
 \xi(e,p)\  |\mu^\pm_e(e,p)|\  \{(1-\mathcal{P}^\pm)h \} \ 
 r (a+r\cos\theta)\mathrm{d}r \mathrm{d}\theta \mathrm{d}\tv \mathrm{d} v_\varphi .
 \end{aligned}
 \end{equation}
Here $\tv = v_r e_r + v_\theta e_\theta, e = \sqrt{1+|\tv|^2+ |v_\varphi|^2}$ and $p = (a+r\cos\theta) v_\varphi$. 
We make the change of variables $(r,\theta,\tv,v_\varphi) \mapsto (r,\theta,e,p,\omega)$, 
where $\omega\in [0,2\pi]$ denotes the angle between $\tv$ and $e_r$. It follows that 
$$r (a+r\cos\theta) \mathrm{d}r \mathrm{d}\theta \mathrm{d} v_\varphi  \mathrm{d}\tv  = r (a+r\cos\theta)\mathrm{d}r \mathrm{d}\theta ~\mathrm{d} v_\varphi ~|\tv| \mathrm{d}|\tv| \mathrm{d}\omega =  r  \mathrm{d}r \mathrm{d}\theta ~\mathrm{d} p ~e \mathrm{d}e \mathrm{d}\omega.$$ 
The identity \eqref{cal-Ph01} then yields 
$$ 0 = 4 \pi^2\int_1^\infty \int_{|p| <  \sqrt{(a+1)(e^2-1)}}\xi(e,p)\ |\mu^\pm_e(e,p)|\Big( \int_{I_{e,p}} (1-\mathcal{P}^\pm)h ~r  \mathrm{d}r \mathrm{d}\theta \Big)\mathrm{d} p \; e \mathrm{d}e,$$
where $I_{e,p}$ 
denotes the subset of $(r,\theta )\in (0,1)\times (0,2\pi)$ such that $a+r\cos\theta > |p|/\sqrt{e^2-1}$. 
Since $\xi = \xi(e,p)$ is an arbitrary function of $(e,p)$, 
it follows that the integral in $(r,\theta)$ must vanish for each $(e,p)$.  Hence  
$$(\mathcal{P}^\pm h)(e,p) = \frac{1}{ \int_{I_{e,p}} r  \mathrm{d}r \mathrm{d}\theta} 
 \int_{I_{e,p}} h(r,\theta) ~r  \mathrm{d}r \mathrm{d}\theta.$$
Considering $I_{e,p}$ in cartesian coordinates in the disk, we have $I_{e,p} = \Theta(\mathbb{S}_{e,p})$.  
\end{proof}

In the inhomogenous case $\vB^0 \not =0$, a similar calculation yields the same formula for the projection $\mathcal{P}^\pm$ 
except that the subset  $\mathbb{S}_{e,p}$ is no longer as explicit as  in Lemma \ref{lem-Ph}.   
In fact, $\mathbb{S}_{e,p}$ is a very complicated set.   See \cite{LS1} for the 1.5D case on the circle.

 \appendix 
\section{Toroidal coordinates}\label{app-cal}
We compute derivatives in the toroidal coordinates  $ x_1 = (a+ r\cos \theta ) \cos \varphi $,  $x_2 = (a+r \cos \theta) \sin \varphi$,  $x_3=  r \sin \theta
$. We recall the corresponding unit vectors  
$$\left\{ \begin{aligned} e_r &= (\cos \theta \cos \varphi,\cos \theta \sin \varphi, \sin \theta),
\\e_\theta &= (-\sin \theta \cos \varphi, -\sin \theta \sin \varphi, \cos \theta), 
\\
e_\varphi &= (-\sin \varphi,\cos \varphi,0). \end{aligned}\right.$$
Easy calculations show  $$ \frac{\D}{\D r} = e_r \cdot \nabla _x , \qquad \frac{\D}{\partial \theta} = r e_\theta \cdot \nabla _x, \qquad  \frac{\D}{\partial \varphi} = (a+r\cos \theta) e_\varphi
\cdot \nabla_x .$$
Using this and noting that $\{e_r,e_\theta,e_\varphi\}$ forms a basis in $\RR^3$, we get for any function $\psi(r,\theta,\varphi)$ and vector function $\vA(r,\theta,\varphi)$ 
$$ \begin{aligned}
\nabla_x \psi &= e_r  \frac {\D \psi}{\D r} + \frac 1r e_\theta \frac {\D \psi}{\D \theta}  + \frac{1}{a+r\cos \theta} e_\varphi \frac{\D \psi}{\D \varphi},
\\
\nabla_x \cdot \vA &= \frac{1}{r(a+r\cos \theta)} \frac {\D}{\D r} (r (a+r\cos \theta) A_r) + \frac{1}{r(a+r\cos \theta)} \frac{\D}{\D \theta} ((a+r\cos \theta)A_\theta) \\&\quad + \frac{1}{a+r\cos \theta} \frac{\D  A_\varphi}{\D\varphi}  ,
\\ \nabla_x \times \vA &= \frac{e_r}{a+r\cos \theta} \Big[ \frac{\D A_\theta}{\D \varphi} - \frac 1r \frac {\D}{\D \theta}( (a+r\cos \theta)   A_\varphi )\Big] 
\\&\quad + \frac{e_\theta}{a+r\cos \theta} \Big[  \frac {\D}{\D r} ((a+r\cos \theta) A_\varphi) - \frac {\D}{\D \varphi} A_r \Big] + \frac{e_\varphi}{r} \Big[ \frac {\D A_r}{\D \theta} - \frac {\D (rA_\theta)}{\D r}\Big ].
\end{aligned}$$
In addition, we also have
$$\begin{aligned}
\Delta_x  \psi &= \frac{1}{r(a+r\cos \theta)} \frac{\D}{\D r}(r(a+r\cos \theta)\D_r \psi) + \frac{1}{r^2(a+r\cos \theta)} \frac{\D}{\D \theta} ((a+r\cos \theta) \D_\theta \psi ) \\& \quad + \frac{1}{(a+r\cos \theta)^2 } \frac{\D ^2 \psi}{\D \varphi^2},
\\
 \Delta \vA &= ~~ \Big[ \Delta A_r - \frac{1}{r^2} A_r- \frac{\cos^2 \theta}{(a+r\cos \theta)^2}A_r -  \frac{2}{r^2} \D_\theta A_\theta  + \frac{\sin \theta }{r (a+r\cos \theta)} A_\theta + \frac{\cos \theta \sin \theta}{(a+r\cos \theta)^2}A_\theta \Big] e_r 
\\&\quad  + \Big[\Delta A_\theta - \frac{1}{r^2}A_\theta  - \frac{\sin^2 \theta}{(a+r\cos \theta)^2}A_\theta+  \frac{2}{r^2} \D_\theta A_r -  \frac{\sin \theta }{r (a+r\cos \theta)} A_r + \frac{\cos \theta \sin \theta}{(a+r\cos \theta)^2}A_r \Big] e_\theta 
\\&\quad + \Big[\Delta  A_\varphi  -  \frac{1}{(a+r\cos \theta)^2} A_\varphi \Big] e_\varphi .
\end{aligned}$$

\section{Scalar operators}
For sake of completeness, we record here details of the operators $\tilde{\mathcal{S}}^\lambda$,  $\tilde{\mathcal{T}}^\lambda_1$, $\tilde{\mathcal{T}}^\lambda_2$, and their adjoints, which are introduced in Section \ref{sec-defOp}. If we introduce scalar operators $\tilde{\mathcal{S}}^\lambda_{jk} h:= \tilde{\mathcal{S}}^\lambda (h e_j) \cdot e_k$ for $j,k \in  \{r, \theta\}$, then we readily get 
  
$$\begin{aligned}
 \tilde{\mathcal{S}}_{rr}^\lambda h  & =  \Big (\lambda^2 - \Delta + \frac 1 {r^2} + \frac{\cos^2 \theta}{(a+r\cos \theta)^2} \Big)   h  -  \sum_\pm \int_{\RR^3}\hat v_r \mu^\pm_e \mathcal{Q}^\pm_\lambda(\hat v_r h)\; \mathrm{d}v  
\\ \tilde{\mathcal{S}}_{\theta\theta}^\lambda h  & =\Big ( \lambda^2  - \Delta + \frac 1 {r^2} +  \frac{\sin^2 \theta}{(a+r\cos \theta)^2} \Big)   h  -  \sum_\pm  \int_{\RR^3}\hat v_\theta \mu^\pm_e \mathcal{Q}^\pm_\lambda(\hat v_\theta h)\; \mathrm{d}v 
\\
\tilde{\mathcal{S}}^\lambda_{r\theta} h   & =  - \Big( \frac{2}{r^2} \D_\theta  -  \frac{\sin \theta }{r (a+r\cos \theta)}  - \frac{\cos \theta \sin \theta}{(a+r\cos \theta)^2}\Big) h   +  \sum_\pm \int_{\RR^3}\hat v_r \mu^\pm_e \mathcal{Q}^\pm_\lambda(\hat v_\theta h)\; \mathrm{d}v ,
\\
\tilde{\mathcal{S}}^\lambda_{\theta r} h   & =    \Big( \frac{2}{r^2} \D_\theta  -  \frac{\sin \theta }{r (a+r\cos \theta)}  +  \frac{\cos \theta \sin \theta}{(a+r\cos \theta)^2}\Big)  h   + \sum_\pm \int_{\RR^3}\hat v_\theta \mu^\pm_e \mathcal{Q}^\pm_\lambda(\hat v_r h)\; \mathrm{d}v .
\end{aligned}$$
Similarly, if we introduce $\tilde{\mathcal{T}}^\lambda_{jk} h := \tilde{\mathcal{T}}^\lambda_j h \cdot e_k$ for $j = 1,2$ and $k = r, \theta$, then we get

$$\begin{aligned}
\tilde{\mathcal{T}}^\lambda_{1r} h   &=  \frac{\lambda}{r(a+r\cos \theta)} \frac {\D}{\D r} (r (a+r\cos \theta) h ) + \sum_\pm  \int_{\RR^3}\mu^\pm_e \mathcal{Q}^\pm_\lambda(\hat v_r h)\; \mathrm{d}v ,
\\
\tilde{\mathcal{T}}^\lambda_{1\theta } h   &= \frac{\lambda}{r(a+r\cos \theta)} \frac{\D}{\D \theta} ((a+r\cos \theta) h)  + \sum_\pm \int_{\RR^3}\mu^\pm_e \mathcal{Q}^\pm_\lambda(\hat v_\theta h)\; \mathrm{d}v ,
\\
\tilde{\mathcal{T}}^\lambda_{2r} h   &=- \sum_\pm  \int_{\RR^3}\hat v_\varphi \mu^\pm_e \mathcal{Q}^\pm_\lambda(\hat v_r h)\; \mathrm{d}v 
,\qquad \tilde{\mathcal{T}}^\lambda_{2\theta} h   = - \sum_\pm \int_{\RR^3}\hat v_\varphi \mu^\pm_e \mathcal{Q}^\pm_\lambda(\hat v_\theta h)\; \mathrm{d}v ,
\end{aligned}$$
whose adjoints read
$$\begin{aligned}
(\tilde{\mathcal{T}}^\lambda_{1r})^* h   &=  - \lambda \D_r h  - \sum_\pm  \int_{\RR^3}\hat v_r\mu^\pm_e \mathcal{Q}^\pm_\lambda h  \; \mathrm{d}v, \qquad 
(\tilde{\mathcal{T}}^\lambda_{2r})^* h   = \sum_\pm  \int_{\RR^3}\hat v_r \mu^\pm_e \mathcal{Q}^\pm_\lambda(\hat v_\varphi h)\; \mathrm{d}v, 
\\
(\tilde{\mathcal{T}}^\lambda_{1\theta})^* h &= - \frac 1r \lambda \D_\theta h  - \sum_\pm  \int_{\RR^3}\hat v_\theta \mu^\pm_e  \mathcal{Q}^\pm_\lambda h  \; \mathrm{d}v, \qquad 
(\tilde{\mathcal{T}}^\lambda_{2\theta})^* h   =
  \sum_\pm \int_{\RR^3}\hat v_\theta \mu^\pm_e \mathcal{Q}^\pm_\lambda(\hat v_\varphi h)\; \mathrm{d}v .
\end{aligned}$$

\section{Equilibria} \label{App-equilibria}
In this appendix, we very easily prove the regularity of any toroidally symmetric equilibrium of the Vlasov-Maxwell system. 
For the sake of completeness,  we then very easily prove the existence of equilibria under certain  conditions 
which are clearly not optimal.  
An alternative existence theorem without a smallness assumption can be found in \cite{FB}. 

As discussed in Section \ref{sec-equilibria}, the potentials of any toroidally symmetric equilibrium satisfy the elliptic system 
\begin{equation}\label{equilibria}\begin{aligned} - \Delta \phi^0 &= F_1(x,\phi^0,A^0_\varphi): =\int_{\RR^3} (\mu^+(e^+,p^+) - \mu^-(e^-,p^-))\; \mathrm{d}v ,\\
\Big(-\Delta +  \frac{1}{(a+r\cos \theta)^2} \Big) A^0_\varphi  &=  F_2(x,\phi^0,A^0_\varphi): = \int_{\RR^3}\hat v_\varphi (\mu^+(e^+,p^+) - \mu^-(e^-,p^-))\; \mathrm{d}v ,\end{aligned}\end{equation}
with $e^\pm= \langle v \rangle  \pm \phi^0$ and $p^\pm= (a+r\cos \theta)(v_\varphi \pm  A_\varphi^0)$, together with the boundary conditions
\begin{equation}\label{equil-bcs} 
\phi^0 = const., \qquad  A_\varphi = \frac{const.}{a+\cos \theta}, \qquad \quad x\in \D\Omega.\end{equation}
Because  $\phi^0 = const.$ and $A^0_\varphi = \frac{const.}{a+r\cos\theta}$ 
are solutions of the homogeneous system \eqref{equilibria}-\eqref{equil-bcs} 
(that is, with the right hand sides of \eqref{equilibria} being zero),  
we can assume without loss of generality that the constants in \eqref{equil-bcs} are zero.  

\begin{lemma}[Regularity of equilibria]  
If  $\mu^\pm(e,p)$ are nonnegative $C^1$ functions of $e,p$ that satisfy the decay assumption \eqref{mu-cond} 
and  $(\phi^0,A^0_\varphi)\in C(\overline\Omega)$ is a solution of \eqref{equilibria}, 
then $(\phi^0,A^0_\varphi)\in C^{2+\alpha}(\overline\Omega)$ and $\vE^0, \vB^0 \in C^{1+\alpha}(\overline\Omega)$ 
for all $0<\alpha<1$.  
\end{lemma}

\begin{proof} 
By hypothesis, the right hand sides of \eqref{equilibria} are continuous. 
Standard elliptic theory shows that $\phi^0,A^0_\varphi$ belong to $C^{1+\alpha}$. 
Now the right hand sides of \eqref{equilibria} belong to $C^1$ smooth.  
Again by ellipticity, $\phi^0,A^0_\varphi \in C^{2+\alpha}(\overline \Omega)$.  
\end{proof}

\begin{lemma}[Existence of equilibria]
Let $\mu^\pm(e,p)$ be nonnegative $C^1$ functions of $e,p$ that satisfy the decay assumption \eqref{mu-cond} 
with a constant $C_\mu$.  There exists $\ep_0>0$ such that if $C_\mu<\ep_0$, 
then there exists an equilibrium $(\phi^0,A^0_\varphi)$ which satisfies \eqref{equilibria} and \eqref{equil-bcs} and 
$\vE^0, \vB^0 \in C^{1+\alpha}(\overline\Omega)$ for all $0<\alpha<1$.  
\end{lemma}
\begin{proof}
As mentioned above,  it is sufficient to construct a nontrivial solution $(\phi^0,A^0_\varphi)$ 
of the elliptic problem \eqref{equilibria} with homogeneous Dirichlet boundary conditions. 
We denote the right side of  \eqref{equilibria} by $F = (F_1, F_2)$.  
For any $\alpha$, let   
$X = \{(\phi^0,A^0_\varphi)\in C^\alpha (\overline \Omega) \times C^\alpha(\overline \Omega) :  \sup_x|\phi^0(x)|\le\frac12 \}$   
furnished with the norm $\|(\phi^0,A^0_\varphi)\|_X = \|\phi^0\|_{C^\alpha} + \|A^0_\varphi\|_{C^\alpha}$.   
Now we have for all $x, y \in \Omega$ 
$$\begin{aligned}
 | F_1&(x,\phi^0(x),A^0_\varphi(x))  - F_1(y,\phi^0(y),A^0_\varphi(y))  | 
 \\&\le \sum_\pm \int_{\RR^3} \Big |\mu^\pm(e^\pm(x,v),p^\pm(x,v)) - \mu^\pm(e^\pm(y,v),p^\pm(y,v))\Big | \; \mathrm{d}v 
 \\&\le C_\mu   \sum_\pm\Big(\int_{\RR^3}\frac{1}{1+|\tilde e^\pm(x,y,v)|^\gamma} \; \mathrm{d}v \Big)\Big[ |x-y| + |\phi^0(x) - \phi^0(y)| + |A^0_\varphi(x) - A^0_\varphi(y)| \Big],
 \end{aligned}$$
where $\tilde e^\pm(x,y,v) = \min\{e^\pm(x,v), e^\pm(y,v)\}$. 
The same bound is valid for $F_2$.  
For $(\phi^0,A^0_\varphi)\in X$ we have 
$|\tilde e^\pm(x,y,v)|  \ge   \langle v \rangle - \sup_{x}|\phi^0(x)|  \ge  \langle v \rangle - \frac12 \ge \frac12 \langle v \rangle .$  
Thus 
if $(\phi^0,A^0_\varphi)$ belongs to $X$,  so does the right side $F(x,\phi^0,A^0_\varphi)$ of \eqref{equilibria} and 
\begin{equation}\label{ineq-F12} 
\| F(\cdot,\phi^0,A^0_\varphi) \|_X \le C \ep_0  \Big[ 1 + \|(\phi^0,A^0_\varphi)\|_X\Big]   
\end{equation} 
for a fixed constant $C$.
Now if $(\phi^0, A^0_\varphi) \in X$, standard elliptic theory implies that there exists a unique solution $(\phi^1, A^1_\varphi) \in C^{2+\alpha}(\overline \Omega)\times C^{2+\alpha}(\overline \Omega)$ of the {\em linear} problem 
$$ -\Delta \phi^1 = F_1(x,\phi^0,A^0_\varphi), \qquad \Big(-\Delta +  \frac{1}{(a+r\cos \theta)^2} \Big) A^1_\varphi = F_2(x,\phi^0,A^0_\varphi),$$
with $\phi^1 = A^1_\varphi = 0$ on the boundary $\partial\Omega$.  Furthermore, we have 
$$\|(\phi^1,A^1_\varphi)\|_{C^{2+\alpha}} \le C' \| (F_1(\cdot,\phi^0,A^0_\varphi), F_2(\cdot,\phi^0,A^0_\varphi)) \|_{C^\alpha} $$ 
for some universal constant $C'$.  
So if we define $\mathcal{T}(\phi^0,A^0_\varphi) = (\phi^1,A^1_\varphi)$, then $\mathcal{T}$ maps $X$ into itself.  
In the same way we easily obtain 
$$\| F(\phi^1,A^1_\varphi) -  F(\phi^2,A^2_\varphi) \|_X 
\le C''\ep_0  \|(\phi^1,A^1_\varphi)  - (\phi^2,A^2_\varphi)\|_X .  $$
 Taking $\ep_0$ sufficiently small, this proves that $\mathcal{T}$ is a contraction and so it has a unique fixed point.  
\end{proof}

We note that our equilibrium is nontrivial, even if we additionally assume that $\mu^+(e ,p) = \mu^-(e ,-p)$ 
as done in the subsection \ref{sec-unstab-ex}. 
Indeed, with this assumption the function 
$h(x,v) = \mu^+(\langle v\rangle, (a+r\cos\theta)v_\varphi) - \mu^-(\langle v\rangle, (a+r\cos\theta)v_\varphi)$ 
is odd in $v_\varphi$, and so  $F_2(x,0,0)$ does not vanish for generic functions $\mu^\pm$.  
That is, $\phi^0=A^0_\varphi=0$ is not a solution to the elliptic problem \eqref{equilibria}.

\section{Particle trajectories}\label{app-particle}
In this appendix, we will show that for almost every particle in $\Omega \times \mathbb{R}^3$, the corresponding trajectory hits the boundary at most a finite number of times in each finite time interval.  
To accomplish this, we follow the method of \cite{Cerci}. 
For convenience, we denote the particle trajectories by 
$$ \Phi_s(x,v) : = (X(s;x,v), V(s;x,v))$$
for each $(x,v)\in \overline{\Omega} \times \mathbb{R}^3$ and $s\in \mathbb{R}$. 
For each $s$, as long as the map $\Phi_s(\cdot)$ is well-defined, its Jacobian determinant is time-independent 
and so is equal to one. Thus $\Phi_s(\cdot)$ is a measure-preserving map 
in $\overline{\Omega} \times \mathbb{R}^3$, with respect to the usual Lebesgue measure $\mu$. 
We denote by $\sigma$ the induced surface measure on $\partial\Omega \times \RR^3$. 

For each $e_0\ge 0$ we denote by $\Sigma_{0}$ the set of points $(x,v) \in \partial\Omega \times \RR^3$ 
for which $|e(x,v)| \le e_0$ and $v_r\le0$. Since the fields are bounded functions, 
$\Sigma_{0}$ is a bounded set in $\RR^6$. 
In particular, $\sigma(\Sigma_{0}) <\infty$. 
Now considering $(x,v) \in \Sigma_{0}$, the particle trajectory starting from $(x,v)$ at $s=0$ 
can be continued by the ODEs \eqref{trajectory} for a certain time.  
We denote by $\alpha (x,v)$ the first positive time when the trajectory hits the boundary, 
that is,  $X(\alpha(x,v);x,v) \in \partial\Omega$, and we  introduce
$$ \Psi(x,v) : = \overline{\Phi}_{\alpha(x,v)}(x,v), 
\text{ where } \overline{\Phi}_s (x,v) = (X, -V_r, V_\theta, V_\varphi)(s;x,v).$$
 The particles that hit the boundary with $v_r = 0$ have measure zero.    
 Indeed, we denote 
$ \Sigma_1 = \Sigma_0 \setminus \bigcup_{k\ge 0} \Psi^{-k}(\{v_r=0\}) ,$ 
with $\Psi^{0}(\cdot)$ being the identity map. 
We let $s_n = \sum_{k\ge 0}^n \alpha (\Psi^k(x,v))$ be the time of the $n$th collision.
By the time-reversibility of the particle trajectories, 
the set $\Psi^{-n}(\{v_r=0\}) = \Phi_{-s_n}(\{v_r = 0\})$  has $\sigma$-measure zero 
for every $n \ge 0$, since $\{v_r = 0\}$ has $\sigma$-measure zero. 
Thus $\Psi$ is a well-defined map from $\Sigma_{1}$ to $\Sigma_{1}$. 

Since $\Phi_s(\cdot)$ is $\mu$-measure preserving, so is $\Psi(\cdot)$ with respect to the 
surface measure $\sigma$ on $\partial\Omega \times \RR^3$.    Next we define  
$$Z : = \Big\{ (x,v) \in \Sigma_{1} ~:~ \sum_{k\ge 0} \alpha (\Psi^k(x,v)) < \infty \Big\} .$$
That is, $Z$ is the set of particles whose trajectories bounce off the boundary infinitely many times 
within a finite time interval.   
We now claim that $\sigma(Z) = 0$.   Clearly it suffices to prove that 
$$Z_\epsilon : = \Big\{ (x,v) \in \Sigma_1 ~:~ \epsilon < \sum_{k\ge 0} \alpha (\Psi^k(x,v)) < \infty \Big\} $$
has $\sigma$-measure zero, for arbitrary small positive $\epsilon$.  
Since 
$\Psi(\cdot)$ is $\sigma$-measure preserving, the Poincar\'e recurrence theorem (for instance, see \cite{Ba}) 
states that either the set $Z_\epsilon$ has $\sigma$-measure zero 
or there is a point $(x,v) \in Z_\epsilon$ so that $\Psi^{n_j} (x,v) \in Z_\epsilon$ 
for some increasing sequence of integers $n_j$, $j\ge 1$.  
The latter case would yield by definition of $Z_\epsilon$ that 
$$ \epsilon< \sum_{\ell\ge 0} \alpha (\Psi^\ell \Psi^{n_j}(x,v))   = \sum_{k\ge n_j} \alpha (\Psi^k(x,v)) .$$
Thus $\sum_{k\ge 0} \alpha (\Psi^k(x,v))$ diverges, so that $(x,v)\not\in Z_\epsilon$.  
This contradiction proves the claim.  

Finally, let $Z^*$ be the set of points in $\Omega \times \RR^3$ whose trajectories hit $Z\bigcup (\cup_{k\ge 0} \Psi^{-k}(\{v_r=0\}))$ within a finite time interval. Since the fields are uniformly bounded, the trajectories in $Z^*$ have bounded lengths. It follows that $Z^*$ must have $\mu$-measure zero, since $Z\bigcup (\cup_{k\ge 0} \Psi^{-k}(\{v_r=0\}))$ has $\sigma$-measure zero. This proves that for almost every particle in $\Omega \times \RR^3$, the corresponding trajectory hits the boundary at most a finite number of times in each finite time interval.

%

\begin{remark} In case both radial and longitudinal symmetry are assumed, so that the fields depend only on the radial variable $r$ as in \cite{NStr1}, it is particularly easy to see that every particle that
initially hits the boundary with $v_r \not = 0$ will bounce off the boundary at most finite number of times in an arbitrary finite time interval. Indeed, for such a particle, at each bouncing time $s_j$, the velocity $V(s_j; x,v)$ is constant and satisfies 
$$\begin{aligned}
 V_\theta(s_j;x,v) &= p - \psi^0(1) = v_\theta,
 \\ |V_r(s_j; x,v)|^2 &= (e - \varphi^0(1))^2 - 1 - (p - \psi^0(1))^2  = |v_r|^2 \not = 0,
 \end{aligned}$$
where $e =  \sqrt{1+|v|^2} + \varphi^0(r)$ and $p = r (v_\theta + \psi^0(r))$ are the invariant particle energy and angular momentum. Thus the bouncing times $s_j$ depend only on the initial data $(x,v)$ and not on the previous bouncing times. So in this case it is elementary that there are at most a finite number of times when the trajectory hits the boundary within 
a given finite time. The set of particles that satisfy $v_r^2 = (e - \varphi^0(1))^2 - 1 - (p - \psi^0(1))^2   = 0$ clearly has measure zero. 
\end{remark}

\end{document}